\documentclass[oneside,english]{amsart}
\usepackage[T1]{fontenc}
\usepackage[latin9]{inputenc}
\usepackage{color}
\usepackage{babel}
\usepackage{amstext}
\usepackage{amsthm}
\usepackage{amssymb}
\usepackage[unicode=true,pdfusetitle,
 bookmarks=true,bookmarksnumbered=false,bookmarksopen=false,
 breaklinks=false,pdfborder={0 0 0},pdfborderstyle={},backref=false,colorlinks=false]
 {hyperref}

\makeatletter

\providecommand{\tabularnewline}{\\}

\numberwithin{equation}{section}
\numberwithin{figure}{section}
\theoremstyle{plain}
\newtheorem{thm}{\protect\theoremname}[section]
\theoremstyle{remark}
\newtheorem{rem}[thm]{\protect\remarkname}
\theoremstyle{plain}
\newtheorem*{prop*}{\protect\propositionname}
\theoremstyle{plain}
\newtheorem{prop}[thm]{\protect\propositionname}
\theoremstyle{plain}
\newtheorem{cor}[thm]{\protect\corollaryname}
\theoremstyle{plain}
\newtheorem{lem}[thm]{\protect\lemmaname}
\theoremstyle{definition}
\newtheorem{defn}[thm]{\protect\definitionname}

\makeatother

\providecommand{\corollaryname}{Corollary}
\providecommand{\definitionname}{Definition}
\providecommand{\lemmaname}{Lemma}
\providecommand{\propositionname}{Proposition}
\providecommand{\remarkname}{Remark}
\providecommand{\theoremname}{Theorem}

\begin{document}
\global\long\def\R{\mathbf{\mathbb{R}}}%
\global\long\def\C{\mathbf{\mathbb{C}}}%
\global\long\def\Z{\mathbf{\mathbb{Z}}}%
\global\long\def\N{\mathbf{\mathbb{N}}}%
\global\long\def\T{\mathbb{T}}%
\global\long\def\Im{\mathrm{Im}}%
\global\long\def\Re{\mathrm{Re}}%
\global\long\def\H{\mathcal{H}}%
\global\long\def\M{\mathbb{M}}%
\global\long\def\P{\mathbb{P}}%
\global\long\def\L{\mathcal{L}}%
\global\long\def\F{\mathcal{\mathcal{F}}}%
\global\long\def\s{\sigma}%
\global\long\def\Rc{\mathcal{R}}%
\global\long\def\W{\tilde{W}}%

\global\long\def\G{\mathcal{G}}%
\global\long\def\d{\partial}%
 
\global\long\def\jp#1{\langle#1\rangle}%
\global\long\def\norm#1{\|#1\|}%
\global\long\def\mc#1{\mathcal{\mathcal{#1}}}%

\global\long\def\Right{\Rightarrow}%
\global\long\def\Left{\Leftarrow}%
\global\long\def\les{\lesssim}%
\global\long\def\hook{\hookrightarrow}%

\global\long\def\D{\mathbf{D}}%
\global\long\def\rad{\mathrm{rad}}%

\global\long\def\env{\mathrm{Env}}%
\global\long\def\re{\mathrm{re}}%
\global\long\def\im{\mathrm{im}}%
\global\long\def\err{\mathrm{Err}}%

\global\long\def\d{\partial}%
 
\global\long\def\jp#1{\langle#1\rangle}%
\global\long\def\norm#1{\|#1\|}%
\global\long\def\ol#1{\overline{#1}}%
\global\long\def\wt#1{\widehat{#1}}%
\global\long\def\tilde#1{\widetilde{#1}}%

\global\long\def\br#1{(#1)}%
\global\long\def\Bb#1{\Big(#1\Big)}%
\global\long\def\bb#1{\big(#1\big)}%
\global\long\def\lr#1{\left(#1\right)}%

\global\long\def\ve{\varepsilon}%
\global\long\def\la{\lambda}%
\global\long\def\al{\alpha}%
\global\long\def\be{\beta}%
\global\long\def\ga{\gamma}%
\global\long\def\La{\Lambda}%
\global\long\def\De{\Delta}%
\global\long\def\na{\nabla}%

\global\long\def\ep{\epsilon}%
\global\long\def\fl{\flat}%
\global\long\def\sh{\sharp}%
\global\long\def\calN{\mathcal{N}}%
\global\long\def\avg{\mathrm{avg}}%

\title[Critical LWP of NLS on $\T^{d}$]{Critical local well-posedness of the nonlinear Schr{\"o}dinger equation
on the torus}
\author{Beomjong Kwak}
\email{beomjong@kaist.ac.kr}
\address{Department of Mathematical Sciences, Korea Advanced Institute of Science
and Technology, 291 Daehak-ro, Yuseong-gu, Daejeon 34141, Korea}
\author{Soonsik Kwon}
\email{soonsikk@kaist.edu}
\address{Department of Mathematical Sciences, Korea Advanced Institute of Science
and Technology, 291 Daehak-ro, Yuseong-gu, Daejeon 34141, Korea}
\begin{abstract}
In this paper, we study the local well-posedness of nonlinear Schr{\"o}dinger
equations on tori $\T^{d}$ at the critical regularity. We focus on
cases where the nonlinearity $|u|^{a}u$ is non-algebraic with small
$a>0$. We prove the local well-posedness for a wide range covering
the mass-supercritical regime. Moreover, we supplementarily investigate
the regularity of the solution map.

In pursuit of lowering $a$, we prove a bilinear estimate for the
Schr{\"o}dinger operator on tori $\T^{d}$, which enhances previously
known multilinear estimates.
We design a function space adapted to the new bilinear estimate and
a package of Strichartz estimates, which is not based
on conventional atomic spaces.
\end{abstract}

\maketitle

\section{Introduction}

\subsection{Statement of the problem and main results}

The subject of this paper is the critical local well-posedness and
ill-posedness of the Cauchy problem for the nonlinear Schr{\"o}dinger equation
(NLS) on periodic spaces $\T^{d}$
\begin{equation}
\begin{cases}
iu_{t}+\De u=\pm|u|^{a}u=:\mathcal{N}(u)\\
u(0)=u_{0}\in H^{s}(\T^{d})
\end{cases},\tag{{NLS}}\label{eq:NLS}
\end{equation}
where $u:\R\times\T^{d}\rightarrow\C$.

The nonlinearity $|u|^{a}u$ is of a single power type for $a>0$.
When $a$ is an even integer, the nonlinearity $\left|u\right|^{a}u$
is algebraic. Otherwise, $|u|^{a}u$ is said to be non-algebraic.
In this work, we are particularly interested in the case of non-algebraic
nonlinearity, especially when $a$ is small. At a glance, for small
non-algebraic $a$, one can observe that the regularity of solutions
has certain restrictions and anticipate that the solution map is less
regular. However, it turns out that there is a genuine difficulty:\emph{
the nonlinear term $|u|^{a}u$ is not sufficiently decomposable.}
For instance, if one tries to take a paraproduct decomposition, there
is not enough summability from existing technology. In this work,
we investigate these issues and overcome them by introducing new function
spaces and bilinear estimates. For the negative direction, we also
study limitations on the regularity of the solution map.

In view of scaling considerations, the critical Sobolev regularity
is
\begin{equation}
s:=s_{c}=\frac{d}{2}-\frac{2}{a}.\label{eq:s_c}
\end{equation}
Since we consider only the critical local problem in $H^{s_{c}}(\T^{d})$,
we simply denote $s=s_{c}$. We say (\ref{eq:NLS}) is mass-critical
if $s=0$ and energy-critical if $s=1$.

Firstly, we state our main theorem, the critical local well-posedness
of a wide range of NLS on $\T^{d}$.
\begin{thm}
\label{thm:LWP s<a} Let $a>\frac{4}{d}$ (or equivalently, $s>0$).
Assume $s<1+a$. Then, (\ref{eq:NLS}) is locally well-posed in the
critical Sobolev space $H^{s}(\T^{d})$.
\end{thm}
For a technical statement of Theorem \ref{thm:LWP s<a}, see Proposition
\ref{prop:Y^s LWP}. Theorem \ref{thm:LWP s<a} extends many existing
results on the critical local well-posedness and covers a new regime
of small $a$. This covers a wide range of the mass-supercritical
regime $s>0$. In particular, this includes all energy-critical cases, for which the result is new  for dimensions $d\ge5$. The
restriction $s<1+a$ arises from the fact that $|u|^{a}u$ does not
have regularities higher than $1+a$ for smooth functions $u$ in
general.

Next, we consider the regularity of the solution map. In Theorem \ref{thm:LWP s<a},
we know only that the flow map is continuous from $H^{s}(\T^{d})$
to $C^{0}([0,T];H^{s}(\T^{d}))$. Yet, for a narrower range of $a$,
we show the Lipschitz regularity of the flow map.
\begin{thm}
\label{thm:Lipschitz}Assume that
\begin{equation}
a>\max\left\{ \frac{4}{d},1\right\} \text{ and }s<a.\label{eq:Lipschitz a assumption}
\end{equation}
Then, (\ref{eq:NLS}) is locally Lipschitz well-posed in $H^{s}(\T^{d})$.
\end{thm}
On the other hand, if $a$ is even lower, one expects the solution
map to be less regular. As a negative result, we show that when $a$
is smaller than $1$, the solution map fails to be locally Lipschitz
continuous. More precisely, we have the failure of $\alpha$-H{\"o}lder
continuity.
\begin{thm}
\label{thm:a<1}Assume
\begin{equation}
0<a<1\text{ and }0<s<1+\frac{1}{a}.\label{eq:assuption thm a<1}
\end{equation}
Then, the solution map fails to be locally $\al$-H{\"o}lder continuous
in $H^{s}(\T^{d})$ for each $\al>a$.

More explicitly, there is no radius $\epsilon>0$ and time $T>0$
such that for every $u_{0},v_{0}\in H^{s}(\T^{d})$ with $\norm{u_{0}}_{H^{s}},\norm{v_{0}}_{H^{s}}<\epsilon$,
the corresponding solutions $u$ and $v$ to (\ref{eq:NLS}) satisfy
$\norm{u-v}_{C^{0}H^{s}\left([0,T]\times\T^{d}\right)}\les\norm{u_{0}-v_{0}}_{H^{s}}^{\al}$.
\end{thm}
\begin{rem}
Our proof does not rely on number-theoretic arguments on frequencies.
Thus, the proof works for irrational tori $\widetilde{\T}^{d}=\R^{d}/(\theta_{1}\Z\times\cdots\times\theta_{d}\Z)$
with any $\theta_{j}>0$. For simplicity, in this paper, we assume
our domain is the square torus $\T^{d}=\R^{d}/(2\pi\Z)^{d}$.
\end{rem}
\begin{rem}
In Theorem \ref{thm:LWP s<a}, one can derive the exponent restriction
by using (\ref{eq:s_c}). When $d\le7$, $s<1+a$ is void, so it holds
true for any $a>\frac{4}{d}$. When $d\ge8$, there is an uncovered
band:
\[
\begin{cases}
\frac{4}{d}<a<\frac{d-2-\sqrt{d^{2}-4d-28}}{4} & \qquad\text{or}\\
a>\frac{d-2+\sqrt{d^{2}-4d-28}}{4}.
\end{cases}
\]
\end{rem}
\begin{rem}
In Theorem \ref{thm:Lipschitz}, the restriction on exponents for
the Lipschitz continuity for each dimension is as follows: 
\begin{align*}
(d\le4)\qquad & a>\frac{4}{d} & \qquad & s>0,\\
(d=5)\qquad & a>1 & \qquad & s>\frac{1}{2},\\
(d\ge6)\qquad & a>\frac{d+\sqrt{d^{2}-32}}{4} & \qquad & s>\frac{d+\sqrt{d^{2}-32}}{4}.
\end{align*}
In particular, we note that the energy-critical case ($s=1$) is Lipschitz
well-posed when $d\le5$. For $d=3,4$, the LWP was previously proved
via a contraction mapping \cite{herr2011global,killip2016scale}.
\end{rem}
\begin{rem}
In Theorem \ref{thm:a<1}, the range of exponents of (\ref{eq:assuption thm a<1})
for each dimension is as follows:
\begin{align*}
(d\le8)\qquad & a<1 & \qquad & s<\frac{d}{2}-2,\\
(d\ge9)\qquad & a<\frac{6}{d-2} & \qquad & s<\frac{d}{6}+\frac{2}{3}.
\end{align*}
In particular, we note that the energy-critical case ($s=1$) fails
to be Lipschitz well-posed when $d\ge7$. However, when $d=6\,(a=1)$,
Lipschitz continuity of the solution map is inconclusive.
\end{rem}
\subsection{Previous works}

In \cite{bourgain1993fourier}, Bourgain obtained a range of scale-invariant
Strichartz estimates on square tori with a certain amount of loss
of regularity, with which $X^{s,b}$ spaces were also first introduced.
He used these to obtain several local and small data global well-posedness
results for subcritical NLS on tori. As a tool for constructing function
spaces adapted to critical dispersive equations, atomic spaces $U^{p}\text{ and }V^{p}$
have been successfully used. $U^{p}\text{ and }V^{p}$ spaces were
developed for the Schr{\"o}dinger operator in \cite{tataru2008large},
\cite{koch2007priori}, and many others. Based on atomic structures,
the critical function spaces $X^{s}$ and $Y^{s}$ for NLS on (partially)
periodic domains were introduced in \cite{herr2014strichartz} and
\cite{herr2011global}.

Based on the development of function spaces, several local and global
well-posedness results of NLS on periodic domains were shown for algebraic
cases. In \cite{herr2014strichartz}, using the $X^{s}$ spaces and
multilinear Strichartz estimates, Herr, Tataru, and Tzvetkov obtained
the local well-posedness and small data global well-posedness of the
energy-critical NLS in $H^{1}(\R^{2}\times\T^{2})$ and $H^{1}(\R^{3}\times\T)$
with arbitrary torus parts $\T^{m}$ (including irrational tori).
In \cite{herr2011global}, the same authors showed the local well-posedness
and small data global well-posedness of the energy-critical NLS
in $H^{1}(\T^{3})$ for rational tori. In \cite{wang2003periodic},
the author developed scaling-critical multilinear Strichartz estimates
and proved results for a larger range of exponent $a$. In \cite{guo2014strichartz},
new scaling-critical Strichartz estimates on irrational tori were
proved. As an application, they proved the critical local well-posedness
of NLS in several regimes of algebraic nonlinearities. This result
was further enhanced in \cite{strunk2014strichartz}.

Afterward, Bourgain and Demeter \cite{bourgain2015proof} established
a Strichartz estimate with an arbitrarily small loss of scale and
regularity on general irrational tori  as an application of their
celebrated $\ell^{2}$-decoupling result. For rational tori, this
result can be strengthened to a scale-invariant version by the argument
in \cite{bourgain1993fourier}. For irrational tori, the corresponding
scale-invariant Strichartz estimate was shown in \cite{killip2016scale}.
As an application of this, they obtained the local well-posedness
and small data global well-posedness result for energy-critical NLS
on $\T^{3}$ and $\T^{4}$.

The large data global well-posedness of the energy-critical defocusing
NLS in $H^{1}(\T^{3})$ was shown in \cite{ionescu2012energy}. In
\cite{ionescu2012global}, the same result was shown in $H^{1}(\R\times\T^{3})$.
For focusing equations, the large data global well-posedness of the
energy-critical focusing NLS in $H^{1}(\T^{4})$ was shown in \cite{YUE2021754}.
The aforementioned well-posedness works for algebraic nonlinearities
use multilinear estimates and are based on contraction mapping arguments.

When $|u|^{a}u$ is non-algebraic, Lee \cite{lee2019local} proved
the well-posedness of $H^{s}$-critical NLS in $H^{s}(\T^{3})$ for
$a\ge2$ (or equivalently, $s\ge1/2$). One main new ingredient of
\cite{lee2019local} was the \emph{Bony linearization }\cite{bony1981calcul}
for non-algebraic nonlinearities. When a nonlinearity $f(u)$ has
sufficient regularity, one takes a paraproduct decomposition of $f(u)$
in terms of $u_{N}$ and $\partial f(u_{\le N})$. For given $f:\C\rightarrow\C$
and $u:\T^{d}\rightarrow\C$, we write
\begin{align*}
f(u) & =\sum_{N\in2^{\N}}f(P_{\le N}u)-f(P_{\le N/2}u)\\
 & =\sum_{N\in2^{\N}}\int_{0}^{1}P_{N}u\d_{z}f(P_{\le N/2}u+\theta P_{N}u)d\theta+\int_{0}^{1}\overline{P_{N}u}\d_{\overline{z}}f(P_{\le N/2}u+\theta P_{N}u)d\theta.
\end{align*}
Lee \cite{lee2019local} used a contraction mapping argument based
on previously known multilinear estimates \cite{herr2011global,killip2016scale}
and the Bony linearization. The condition $a\ge2$ was required for
triple iterations of Bony linearizations. 

\subsection{A new estimate, a function space $Z^{s}$, and proofs of the main results}

The main difficulty of our well-posedness results, Theorem \ref{thm:LWP s<a}
and Theorem \ref{thm:Lipschitz}, lies in the previously unresolved regime
$a<2$. Here, we investigate the limitations of existing techniques
for lower $a$, which stems from the linear level, and introduce our
new main ingredients to overcome them: a bilinear Strichartz estimate
and a function space $Z^{s}$.

To date, critical Strichartz estimates have been established on pure
tori $\T^{d}$ only with a loss of regularity. To compensate for the
regularity loss, multilinear estimates have been used. On $\R^{d}$,
for dyadic frequencies $N\gg R$, we have the bilinear estimate
\begin{equation}
\norm{P_{N}e^{it\De}\phi P_{R}e^{it\De}\psi}_{L_{t,x}^{2}(\R\times\R^{d})}\les N^{\frac{d-1}{2}}R^{-\frac{1}{2}}\norm{\phi}_{L^{2}}\norm{\psi}_{L^{2}}.\label{eq:R^d bilinear}
\end{equation}
We do not expect an estimate like (\ref{eq:R^d bilinear}) on $\T^{d}$,
even with a finite time cutoff. Indeed, a trivial choice $\psi\equiv1$
gives a simple counterexample for (\ref{eq:R^d bilinear}) on $\T^{d}$.
Still, the following version of the bilinear estimate was previously
known:
\begin{prop*} \cite{herr2011global,killip2016scale}
 For $d\ge3$, when $N_{1}\ge N_{2}$, there exists $\delta>0$ such
that
\begin{equation}
\norm{P_{N_{1}}uP_{N_{2}}v}_{L_{t,x}^{2}(I\times\T^{d})}\les_{I}N_{2}^{\frac{d-2}{2}}\left(\frac{N_{2}}{N_{1}}+\frac{1}{N_{2}}\right)^{\delta}\norm u_{Y^{0}}\norm v_{Y^{0}}.\label{eq:HerrTataruTzvetkov-type}
\end{equation}
\end{prop*}
A key strength of (\ref{eq:HerrTataruTzvetkov-type}) is the decay
factor $\delta>0$. The first proof of (\ref{eq:HerrTataruTzvetkov-type})
used spacetime almost orthogonalities, requiring both $u$ and $v$
to be free evolutions, and was applicable to algebraic nonlinearities.
A weaker version with $\delta=0$ allows a simpler proof by partitioning
the frequency domain $\Z^{d}$ into congruent cubes. When $a\ge2$,
such a weaker estimate is sufficient for the local well-posedness by paraproduct
decompositions (see \cite{killip2016scale,lee2019local}).

The decay factor $\delta>0$, however, becomes crucial when $a<2$.
Decomposing the nonlinear term $\left|u\right|^{a}u$ into a product
of a linear part and the rest, say, of the form
\[
\left|u\right|^{a}u=u\times A=\sum_{N,R\in2^{\N}}P_{N}u\times P_{R}A,
\]
all frequency sizes of $A$ contribute critically if one uses
a bilinear estimate without decay on a high-low product. Since $H^{s}(\T^{d})$
is $\ell^{2}$-based, we expect $A$ to lie in any $\ell^{1}$-based
critical Besov space only if $a\ge2$, otherwise causing a logarithmic
loss in the summation.

We extend (\ref{eq:HerrTataruTzvetkov-type}) to general Sobolev regularities
through a new approach. For nice functions $u$ and $A$, and dyadic numbers
$N\gtrsim R$, we show
\begin{align}
\norm{\chi_{[-1,1]}\cdot P_{N}u\cdot P_{R}A}_{(Z^{0})'} & \les\norm u_{Z^{0}}\left(\text{\ensuremath{\left(N/R\right)^{-\s_{1}}}}+R^{-2\s_{1}}\right)R^{\theta}\norm A_{B_{r_{0},r_{0}}^{\frac{1}{r_{0}}-\frac{1}{q_{0}}}L^{r_{0}}},\label{eq:key prop}
\end{align}
where $Z^{0}$ and $(Z^{0})'$ are the new function space of this
paper and its spacetime dual norm, respectively, and $\s_{1},q_{0},r_{0}$,
and $\theta$ denote the exponents defined in (\ref{eq:ss1s2<<})
and Lemma \ref{lem:strip1 claim1-1}. (See (\ref{eq:strip1 claim1-1})
for the precise form of (\ref{eq:key prop}).) While (\ref{eq:HerrTataruTzvetkov-type})
in \cite{herr2011global} was shown by estimating almost orthogonalities
on the Fourier side, we detect the decay factor $\s_{1}>0$ for (\ref{eq:key prop})
based on the Galilean structure and spacetime Besov regularities of
the Schr{\"o}dinger operator.

Relying crucially on (\ref{eq:key prop}), the proof of Theorem \ref{thm:LWP s<a}
proceeds as follows: In view of Theorem \ref{thm:a<1}, the solution
map is not Lipschitz continuous for small $a$, so we do not use a
contraction mapping argument. Instead, we separately show the existence
of solutions, a decay of high-frequency pieces, and a contraction-type
estimate in a space of lower regularity. Using Bony linearizations
and (\ref{eq:key prop}), we construct an a priori bound on a solution
for a short time and show the local existence by taking a weak limit.
For the continuity of the solution map, we further obtain extra a
priori decay on the high-frequency part of a solution. Then, the problem
reduces to showing a bootstrapping estimate of the difference between
two solutions, which follows immediately from (\ref{eq:key prop}).

For initial data with large $H^{s}$ norms, we face an obstacle in
the choice of function spaces. Earlier works on the critical well-posedness
of NLS on $\T^{d}$ (\cite{herr2011global,wang2003periodic,guo2014strichartz,strunk2014strichartz,killip2016scale})
used atomic-based norms $X^{s}$ and $Y^{s}$. The estimates used
in the proof of Theorem \ref{thm:LWP s<a} could also be shown in
terms of $Y^{s}$. However, if one uses an atomic-based norm such
as $Y^{s}$, the norm of a free evolution does not shrink sufficiently
on any short time interval, making the bootstrapping inequalities not
obvious for large initial data.

For the regime $a>2$, earlier authors resolved the issue by estimating
the high-frequency portion of $u$ separately. More precisely, they
showed bootstrap bounds on $\norm{P_{\ge N}u}_{Y^{1}}$ for $N\gg1$
by paraproduct decompositions on the nonlinear term $\calN(u)$. Such
decompositions require a certain regularity of $\mathcal{N}(u)$ (or
equivalently, a high power $a$).

We construct a new function space $Z^{s}$ on $\R\times\T^{d}$ adapted
to conventional linear estimates, the bilinear Strichartz estimate
(\ref{eq:key prop}), and the desired norm-shrinking property. More
precisely, the $Z^{s}$ space has the following favorable properties:
\begin{enumerate}
\item Boundedness of the retarded dual Schr{\"o}dinger propagator from $(Z^{0})'$
to $Z^{0}$, (\ref{eq:Z^s-Y^s})
\item Strichartz embeddings into Sobolev spaces, (\ref{eq:Sobolev embed Z^s}) and
(\ref{eq:time Besov embed Z^s})
\item Shrinking of the norm to zero as we give shorter time cutoffs; for
$u\in Z^{s}$, $\norm{\chi_{[0,T]}u}_{Z^{s}}\rightarrow0$ as $T\rightarrow0$,
(\ref{eq:shrink Z^s})
\end{enumerate}
Based on this new $Z^{s}$ space, the proof of Theorem \ref{thm:LWP s<a}
works consistently for arbitrarily large initial data.

The proof of Theorem \ref{thm:Lipschitz} is similar to that of Theorem \ref{thm:LWP s<a}
at the level of functional estimates. For the Lipschitz regularity
of the solution map, we use a contraction mapping argument in Theorem
\ref{thm:Lipschitz}. Although a conventional contraction argument
is used, the main difficulty of Theorem \ref{thm:LWP s<a} that requires
(\ref{eq:key prop}) and $Z^{s}$ spaces is still present for the regime
$1<a<2$, and the machinery built for Theorem \ref{thm:LWP s<a} is
thoroughly used.

The negative counterpart of Theorem \ref{thm:Lipschitz} is addressed
in Theorem \ref{thm:a<1} by constructing an explicit counterexample.
In particular, Theorem \ref{thm:a<1} shows that the main assumption
$a>1$ of Theorem \ref{thm:Lipschitz}, which is crucially required
for a difference form for a contraction inequality, is indeed almost
sharp. A key observation for the construction is an oscillating behavior
of the frequency-localized Schr{\"o}dinger kernel $e^{it\De}\delta_{N}$
on $\T$, (\ref{eq:L^inftyL^2}). We show that the $L^{2}$ and $L^{\infty}$
norms of $e^{it\De}\delta_{N}$ are comparable on a large set of times
$t$, which implies that $e^{it\De}\delta_{N}$ mostly tends to oscillate
rather than concentrate.

The rest of the paper is organized as follows: In Section \ref{sec:Preliminaries},
we provide preliminary materials, such as notations, Strichartz estimates,
and atomic spaces. In Section \ref{sec:-Zspaces}, we define the function
space $Z^{s}$ and show related bilinear estimates. In Section \ref{sec:Proof of LWP},
we provide the proof of Theorem \ref{thm:LWP s<a}. In Section \ref{sec:Lipschitz},
we show Theorem \ref{thm:Lipschitz}. In Section \ref{sec:Proof-of-Theorem a<1},
we prove Theorem \ref{thm:a<1}.

\subsection*{Acknowledgements}

The authors are partially supported by National Research Foundation
of Korea, NRF-2019R1A5A1028324 and NRF-2022R1A2C1091499.

\section{\label{sec:Preliminaries}Preliminaries}

\subsection{Notations}

We denote $A\les B$ if $A\le CB$ for some constant $C$.

Given a set $E\subset\R^{d}$ or $\T^{d}$, we denote by $\chi_{E}$
the sharp cutoff function of $E$.

\subsubsection*{Fourier truncations}

We handle functions of spacetime variables $f(t,x)$ and $f(x)$ for
$x\in\T^{d}$ and $t\in\R$. We denote the Fourier transform (or the Fourier
series) of $f$ with respect to the associated variables $x$, $t$,
and $(t,x)$ by $\F_{x}f,\F_{t}f$, and $\F_{t,x}f$, respectively.
For simplicity, we also denote the spatial Fourier transform by $\widehat{f}$
and the spacetime Fourier transform by $\widetilde{f}$.

We use frequency truncation operators. For spatial frequencies, we
use sharp cutoffs. For time frequencies, we use smooth cutoffs. We
denote by $P_{C}$ the spatial frequency cutoff projection for a given
set $C\subset\Z^{d}$; that is, $P_{C}$ is the Fourier multiplier
operator associated with the characteristic function $\chi_{C}$.
For most cases, we use the Littlewood-Paley projection. We denote
the set of natural numbers by $\N=\left\{ 0\right\} \cup\Z_{+}$ and
dyadic numbers by $2^{\N}$. For a dyadic number $N\in2^{\N}$, we
denote the Littlewood-Paley operators by
\[
P_{\le N}:=P_{[-N,N]^{d}}\text{ and }P_{N}:=P_{\le N}-P_{\le N/2},
\]
where we set $P_{\le1/2}:=0$. In particular, the cutoff $P_{1}=P_{\le1}$
contains the zero frequency mode. For simplicity, we denote $u_{N}=P_{N}u$
and $u_{\le N}=P_{\le N}u$ for $u:\T^{d}\rightarrow\C$.

For time Fourier projections, we use the superscript $t$; $P_{\le N}^{t}$
is a smooth time Littlewood-Paley operator. Let $\varphi:\R\rightarrow[0,\infty)$
be a smooth even bump function such that $\varphi|_{[-1,1]}\equiv1$
and $\text{supp}(\varphi)\subset[-\frac{11}{10},\frac{11}{10}]$.
For a dyadic number $N\in2^{\N}$, we denote by $\varphi_{N}:\R\rightarrow[0,\infty)$ the function $\varphi_{N}(t)=\varphi(t/N)$. We denote by $P_{\le N}^{t}$ the
Fourier multiplier operator induced by $\varphi_{N}$.

In Section \ref{sec:Proof-of-Theorem a<1}, we will use a smooth cutoff
for the spatial frequency truncation operator on $\T$. For a dyadic
number $N\in2^{\N}$, we denote by $\P_{\le N}$ the smooth Littlewood-Paley
operator on $\T$, i.e., $\text{\ensuremath{\P}}_{\le N}$ denotes
the Fourier multiplier operator induced by $\varphi_{N}$. We also
denote by $\delta_{N}=\P_{N}\delta$ the function on $\T$ defined
as $\F_{x}^{-1}\varphi_{N}$.

\subsubsection*{Paraproducts}

We use paraproduct decompositions on spatial frequencies. Given functions
$u$ and $v$ defined on either $\T^{d}$ or $\R\times\T^{d}$, we
denote their paraproducts by 
\[
\pi_{>}(u,v):=\sum_{M\ge32N}u_{M}v_{N}\text{ and }\pi_{<}(u,v):=\sum_{N\ge32M}u_{M}v_{N},
\]
where the summations are made over dyadic numbers. Similarly, we also
use the notations $\pi_{\ge}(u,v):=\sum_{M\ge\frac{1}{16}N}u_{M}v_{N}$
and $\pi_{\le}(u,v):=\sum_{N\ge\frac{1}{16}M}u_{M}v_{N}$.

\subsubsection*{Interpolations}

We use various function spaces for functions defined on $\T^{d}$
or $\R\times\T^{d}$, such as $L^{p}$-based spaces, atomic spaces,
and the spaces $X^{s}$ and $Y^{s}$ generated from the atomic spaces.
Each space is a Banach space and we denote the norm of a Banach space
$B$ by $\norm f_{B}$.

$B'$ denotes the dual of $B$ with respect to the inner product $\left\langle u,v\right\rangle :=\int\overline{u}v$
over the domain $\T^{d}$ or $\R\times\T^{d}$.

A finite collection of Banach spaces $(B_{1},\ldots,B_{n})$ is said
to be an interpolation tuple if $B_{1},\ldots,B_{n}$ can be embedded
simultaneously in a Hausdorff topological vector space. For an interpolation
tuple $(B_{1},\ldots,B_{n})$, we define the intersection and the
sum of Banach spaces $\bigcap_{j=1}^{n}B_{j}$ and $\sum_{j=1}^{n}B_{j}$
by the norms
\[
\|u\|_{\bigcap_{j=1}^{n}B_{j}}:=\max_{j}\|u\|_{B_{j}}
\]
and
\[
\norm u_{\sum_{j=1}^{n}B_{j}}:=\inf_{\substack{u_{1}+\ldots+u_{n}=u\\
u_{j}\in B_{j}
}
}\sum_{j}\norm{u_{j}}_{B_{j}},
\]
respectively.

We use conventional notations for interpolation spaces. Let $\theta\in(0,1)$
and let $(B_{0},B_{1})$ be an interpolation couple. The complex interpolation
space between $B_{0}$ and $B_{1}$ of exponent $\theta$ is denoted
by $[B_{0},B_{1}]_{\theta}$.

Given $q\in[1,\infty]$, the real interpolation space between $B_{0}$
and $B_{1}$ of exponent $\theta$ and parameter $q$ is
denoted by $(B_{0},B_{1})_{\theta,q}$. (For more details, see, for
example, \cite{bergh2012interpolation}.)

\subsection{Function spaces}

Here, we collect well-known facts regarding function spaces.

Given $q\in[1,\infty]$ and a Banach space $E$ defined on $\T^{d}$,
the mixed norm $L^{q}E=L_{t}^{q}E$ is defined as 
\[
\norm u_{L^{q}E}:=\left(\int_{\R}\norm{u(t)}_{E}^{q}dt\right)^{1/q}.
\]
We omit the subscript $t$ for simplicity of notation.

More generally, for $m\in\N$, we denote by $W^{m,q}E$ the norm
\[
\norm u_{W^{m,q}E}:=\sum_{j=0}^{m}\norm{\d_{t}^{j}u}_{L^{q}E}.
\]

Given $p\in(1,\infty)$ and $s\in\R$, we denote by $H^{s,p}(\T^{d})$
the (fractional regularity) Sobolev space given by the norm $\norm f_{H^{s,p}}=\norm{\F_{x}^{-1}\left(\widehat{f}(\xi)\cdot\jp{\xi}^{s}\right)}_{L^{p}(\T^{d})}$,
where $\jp{\xi}$ denotes $\sqrt{1+\left|\xi\right|^{2}}$.

Given $s\in\R$, $p,q\in[1,\infty]$, and a Banach space $E$ defined
on $\T^{d}$, we define the (vector-valued) Besov space $B_{p,q}^{s}E=(B_{p,q}^{s})_{t}E_{x}$
as the dyadic summation of time frequency cutoffs
\[
\|u\|_{B_{p,q}^{s}E}:=\left(\sum_{N\in2^{\N}}N^{qs}\|P_{N}^{t}u\|_{L^{p}E}^{q}\right)^{1/q}+\|P_{\le1}^{t}u\|_{L^{p}E}.
\]
More generally, for a spacetime Banach space $F$ of functions defined
on $\R\times\T^{d}$, we denote by $\ell_{s;\tau}^{q}F$ the norm
\[
\|u\|_{\ell_{s;\tau}^{q}F}:=\left(\sum_{N\in2^{\N}}N^{qs}\|P_{N}^{t}u\|_{F}^{q}\right)^{1/q}+\norm{P_{\le1}^{t}u}_{F}.
\]
In particular, the time Besov space $B_{p,q}^{s}E$ is norm-equivalent
to $\ell_{s;\tau}^{q}L^{p}E$.

For spatial frequencies, we use the notation $\ell_{s}^{q}$ for Banach
spaces on both $\T^{d}$ and $\R\times\T^{d}$. For Banach spaces
$E$ and $F$ defined on $\T^{d}$ and $\R\times\T^{d}$, respectively,
we define
\[
\|u\|_{\ell_{s}^{q}E}:=\left(\sum_{N\in2^{\N}}N^{qs}\|u_{N}\|_{E}^{q}\right)^{1/q}+\|u_{\le1}\|_{E}
\]
and
\[
\|u\|_{\ell_{s}^{q}F}:=\left(\sum_{N\in2^{\N}}N^{qs}\|u_{N}\|_{F}^{q}\right)^{1/q}+\norm{u_{\le1}}_{F}.
\]
Unlike the Besov space notation, both $\ell_{s}^{q}$ and $\ell_{s;\tau}^{q}$
can be applied to a spacetime function space, so we use a subscript
$\tau$ to distinguish them. When $s=0$, we omit the subscripts from
$\ell_{0}^{q}$ and $\ell_{0;\tau}^{q}$, denoting them by $\ell^{q}$
and $\ell_{\tau}^{q}$, respectively.
\begin{prop}
\cite{amann1997operator}We have the following embedding relations:
\begin{enumerate}
\item For $s\in(0,1)$ and $p\in(1,\infty)$, we have
\begin{equation}
\norm u_{B_{p,p}^{s}(\T^{d})}^{p}\sim\norm u_{L^{p}}^{p}+\int_{\T^{d}\times\T^{d}}\left(\frac{\left|u(x)-u(y)\right|}{|x-y|^{s}}\right)^{p}\frac{d(x,y)}{|x-y|^{d}}.\label{eq:slobo x}
\end{equation}
Similarly, for a Banach space $E$ on $\T^{d}$, we have
\begin{equation}
\norm u_{B_{p,p}^{s}E}^{p}\sim\norm u_{L^{p}E}^{p}+\int_{\R\times\R}\left(\frac{\norm{u(t_{1})-u(t_{2})}_{E}}{|t_{1}-t_{2}|^{s}}\right)^{p}\frac{d(t_{1},t_{2})}{|t_{1}-t_{2}|}.\label{eq:slobo tx}
\end{equation}
\item For $s\in\R$ and $p\in[2,\infty)$, we have
\begin{equation}
\norm u_{B_{p,p}^{s}(\T^{d})}\les\norm u_{H^{s,p}(\T^{d})}\les\norm u_{B_{p,2}^{s}(\T^{d})}.\label{eq:BcWcB}
\end{equation}
When $p\in(1,2]$, the opposite embedding relation holds.
\end{enumerate}
\end{prop}
\begin{proof}
(\ref{eq:slobo tx}) is introduced, for example, in \cite[(5.8)]{amann1997operator}.
(\ref{eq:slobo x}) and (\ref{eq:BcWcB}) are known properties for
$\R^{d}$; (\ref{eq:slobo x}) is a special case of \cite[(5.8)]{amann1997operator},
and for (\ref{eq:BcWcB}), see, for example, \cite[Theorem 6.4.4]{bergh2012interpolation}.
These results can be shown similarly on $\T^{d}$ via the Littlewood-Paley
theory on $\T^{d}$.
\end{proof}
The proposition below provides facts for function-valued Besov spaces
regarding embeddings, Bernstein inequalities, and interpolations.
\begin{prop}
\label{prop:Besov props} \cite{amann1997operator,amann2000compact,nakamura2016modified} In
this proposition, we denote by $s_{\theta}$ the number $s_{\theta}=(1-\theta)s_{0}+\theta s_{1}$,
where the numbers $s_{0}$ and $s_{1}$ are given in each corresponding statement
and $\theta\in(0,1)$ is an arbitrary number. Similarly, we denote by
$p_{\theta}$ and $q_{\theta}$ the numbers such that $\frac{1}{p_{\theta}}=\frac{1-\theta}{p_{0}}+\frac{\theta}{p_{1}}$
and $\frac{1}{q_{\theta}}=\frac{1-\theta}{q_{0}}+\frac{\theta}{q_{1}}$,
respectively.

Let $E$ and $E_{j},j=0,1$ be Banach spaces on $\T^{d}$. For a spacetime
function $f:\R\times\T^{d}\rightarrow\C$, we have the following embedding
relations:
\begin{itemize}
\item For $1\le p<\infty$ and $m\in\N$, we have
\begin{equation}
\norm f_{B_{p,\infty}^{m}E}\les\norm f_{W^{m,p}E}\les\norm f_{B_{p,1}^{m}E}.\label{eq:B c W c B}
\end{equation}
\item For $1\le\tilde p<p<\infty$ and $1\le q\le\infty$, we have
\begin{equation}
\norm f_{L^{p,q}E}\les\norm f_{B_{\tilde p,q}^{1/\tilde p-1/p}E}.\label{eq:Besov-Lorentz embed}
\end{equation}
\item For $1\le\tilde p<p<\infty$ and $M\in2^{\N}$, we have
\begin{equation}
\|P_{M}^{t}f\|_{L^{p}E}\les M^{1/\tilde p-1/p}\|P_{M}^{t}f\|_{L^{\tilde p}E}.\label{eq:Besov Bernstein}
\end{equation}
\item For $1\le\tilde p<p<\infty$, $1\le q\le\infty$, and $s\in\R$, we
have
\begin{equation}
\norm f_{B_{p,q}^{s}E}\les\norm f_{B_{\tilde p,q}^{s+1/\tilde p-1/p}E}.\label{eq:Besov embed}
\end{equation}
\item For $p,q\in(1,\infty)$ and $s\in\R$, assuming further that $E'$
is separable, we have
\begin{equation}
\norm f_{(B_{p,q}^{s}E)'}\sim\norm f_{B_{p',q'}^{-s}E'}.\label{eq:Besov dual}
\end{equation}
\item For $p\in[1,\infty)$, $q_{0},q_{1},\eta\in[1,\infty]$, and $s_{0},s_{1}\in\R$
such that $s_{0}\ne s_{1}$, we have
\begin{equation}
\norm f_{(B_{p,q_{0}}^{s_{0}}E,B_{p,q_{1}}^{s_{1}}E)_{\theta,\eta}}\sim\norm f_{B_{p,\eta}^{s_{\theta}}E}.\label{eq:Besov real interp E s0s1}
\end{equation}
\item For $p_{0},p_{1}\in[1,\infty)$, $q_{0},q_{1}\in[1,\infty]$, $s_{0},s_{1}\in\R$,
and an interpolation couple $(E_{0},E_{1})$, we have
\begin{equation}
\norm f_{[B_{p_{0},q_{0}}^{s_{0}}E_{0},B_{p_{1},q_{1}}^{s_{1}}E_{1}]_{\theta}}\sim\norm f_{B_{p_{\theta},q_{\theta}}^{s_{\theta}}[E_{0},E_{1}]_{\theta}}.\label{eq:Besov complex interp}
\end{equation}
\end{itemize}
\end{prop}
\begin{proof}
(\ref{eq:B c W c B}) is given in \cite[(3.6)]{amann2000compact}.
(\ref{eq:Besov-Lorentz embed}) is given in \cite[Lemma 2.4(1)]{nakamura2016modified}.
(\ref{eq:Besov Bernstein}) and (\ref{eq:Besov embed}) are direct
consequences of (\ref{eq:Besov-Lorentz embed}). (\ref{eq:Besov dual})
is given in \cite[(5.22)]{amann1997operator}. In \cite[Lemma 5.1]{amann1997operator},
it is shown that $B_{p,q}^{s}E$ are retracts of $\ell_{q}^{s}(L^{p}E)$
with a common retraction, which implies (\ref{eq:Besov complex interp})
and (\ref{eq:Besov real interp E s0s1}) (see, for example, \cite[Section 6.4, 5.6]{bergh2012interpolation}). 
\end{proof}
\begin{prop}
Let $E_j,j=1,2,3$ be Banach spaces on $\T^{d}$ satisfying
the inequality
\begin{equation}
\left|\int_{\T^{d}}fghdx\right|\les\norm f_{E_{1}}\norm g_{E_{2}}\norm h_{E_{3}}.\label{eq:fgh}
\end{equation}
Let $s_{j}\in\R,p_{j}\in(1,\infty),q_{j}\in[1,\infty],j=1,2,3$ be
exponents satisfying the inequalities 
\[
s_{1}+s_{2}+s_{3}>0,\qquad s_{2}+s_{3}>0,\qquad\frac{1}{p_{1}}>s_{1},
\]
and the scaling conditions
\[
\frac{1}{p_{1}}+\frac{1}{p_{2}}+\frac{1}{p_{3}}=s_{1}+s_{2}+s_{3}+1\text{ and }\frac{1}{q_{1}}+\frac{1}{q_{2}}+\frac{1}{q_{3}}=1.
\]
We have the estimate
\begin{equation}
\sum_{L\les M\les N}\left|\int_{\R\times\T^{d}}P_{L}^{t}f\cdot P_{M}^{t}g\cdot P_{N}^{t}hdxdt\right|\les\norm f_{B_{p_{1,}q_{1}}^{s_{1}}E_{1}}\norm g_{B_{p_{2},q_{2}}^{s_{2}}E_{2}}\norm h_{B_{p_{3},q_{3}}^{s_{3}}E_{3}}.\label{eq:Besov trilinear}
\end{equation}
\end{prop}
\begin{proof}
Since $s_{1}$ is of the lowest frequency and $s_{1}+s_{2}+s_{3}>0$,
by increasing $s_{1}$ and decreasing $s_{2}+s_{3}$, we may assume
$s_{1}>0$ in advance.

Similarly, since the frequencies $M$ and $N$ are comparable and
$\frac{1}{p_{2}}+\frac{1}{p_{3}}=s_{2}+s_{3}+s_{1}+1-\frac{1}{p_{1}}>s_{2}+s_{3}>0$,
by perturbing $s_{2}$ and $s_{3}$ keeping $s_{2}+s_{3}$ fixed,
we may assume $\frac{1}{p_{2}}>s_{2}>0$ and $\frac{1}{p_{3}}>s_{3}>0$
in advance.

For each $j=1,2,3$, by $\frac{1}{p_{j}}>s_{j}>0$ and (\ref{eq:Besov-Lorentz embed}), we have the
embedding $B_{p_{j},q_{j}}^{s_{j}}E\hook L^{\tilde p_{j},q_{j}}E$,
where $\tilde p_{j}$ is the exponent $\frac{1}{\tilde p_{j}}:=\frac{1}{p_{j}}-s_{j}$.
By scaling conditions, we have $\frac{1}{\tilde p_{1}}+\frac{1}{\tilde p_{2}}+\frac{1}{\tilde p_{3}}=\frac{1}{q_{1}}+\frac{1}{q_{2}}+\frac{1}{q_{3}}=1$,
which implies (\ref{eq:Besov trilinear}).
\end{proof}
As a particular consequence, we have product rules for time Besov
spaces.
\begin{cor}
Let $E_j,j=1,2,3$ be Banach spaces satisfying (\ref{eq:fgh}).
Let $s_{j}\in\R,p_{j}\in(1,\infty),q_{j}\in[1,\infty],j=1,2,3$ be
parameters such that $s_{1}+s_{2}+s_{3}>0$, $\frac{1}{p_{j}}>s_{j}$,
$\frac{1}{p_{1}}+\frac{1}{p_{2}}+\frac{1}{p_{3}}=s_{1}+s_{2}+s_{3}+1$,
and $\frac{1}{q_{1}}+\frac{1}{q_{2}}+\frac{1}{q_{3}}=1$.
\begin{enumerate}
\item Assume $s_{1}+s_{2}>0$, $s_{1}+s_{3}>0$, and $s_{2}+s_{3}>0$. Then,
we have
\begin{equation}
\norm{uv}_{B_{p_{3}',q_{3}'}^{-s_{3}}E_{3}'}\les\norm u_{B_{p_{1},q_{1}}^{s_{1}}E_{1}}\norm v_{B_{p_{2},q_{2}}^{s_{2}}E_{2}}.\label{eq:Besov product rule}
\end{equation}
\item Assume $s_{1}+s_{2}>0$ and $s_{2}+s_{3}>0$. Then, we have
\begin{equation}
\norm{\pi_{\le}(u,v)}_{B_{p_{3}',q_{3}'}^{-s_{3}}E_{3}'}\les\norm u_{B_{p_{1},q_{1}}^{s_{1}}E_{1}}\norm v_{B_{p_{2},q_{2}}^{s_{2}}E_{2}}.\label{eq:Besov paraproduct rule,<=00003D}
\end{equation}
\end{enumerate}
\end{cor}
\begin{proof}
(\ref{eq:Besov product rule}) and (\ref{eq:Besov paraproduct rule,<=00003D})
are direct consequences of (\ref{eq:Besov trilinear}) and dualities.
For (\ref{eq:Besov product rule}), the high-frequency terms can be
either $(u,v)$, $(u,\pi_{\le}(u,v))$, or $(v,\pi_{\le}(u,v))$,
so we assume all of $s_{1}+s_{2}>0$, $s_{1}+s_{3}>0$, and $s_{2}+s_{3}>0$.
For (\ref{eq:Besov paraproduct rule,<=00003D}), the high-frequency
terms can be either $(u,v)$ or $(v,\pi_{\le}(u,v))$, so we require
only $s_{1}+s_{2}>0$ and $s_{2}+s_{3}>0$.
\end{proof}
Proposition \ref{prop:Besov props}, (\ref{eq:Besov product rule}),
and (\ref{eq:Besov paraproduct rule,<=00003D}) can be shown similarly
on (scalar-valued) Besov spaces on $\T^{d}$, unless they are $L^{1}$
or $L^{\infty}$-based. This can be done via estimates on Littlewood-Paley
convolution kernels on $\T^{d}$.

Next, we state the fractional chain rule for H{\"o}lder continuous functions.
\begin{lem}
Let $F\in C^{0,\al}(\C)$, $\alpha\in(0,1)$. Let $s\in(0,\al)$,
$\s>0$, and $p,p_{1},p_{2}\in(1,\infty)$ be exponents satisfying
$\frac{1}{p}=\frac{1}{p_{1}}+\frac{1}{p_{2}}$ and $(\al-\frac{s}{\sigma})p_{1}>1$.
We have
\begin{equation}
\|F(u)\|_{H^{s,p}}\les\|u\|_{L^{(\al-\frac{s}{\sigma})p_{1}}}^{\al-\frac{s}{\s}}\cdot\|u\|_{H^{\s,\frac{s}{\sigma}p_{2}}}^{\frac{s}{\sigma}}.\label{eq:low fractional Holder}
\end{equation}
\end{lem}
\begin{proof}
In \cite{visan2007defocusing}, (\ref{eq:low fractional Holder})
is proved for $\s<1$. For $\s\ge1$, we choose $\tilde{\s}\in(\frac{s}{\al},1)$
and let $\tilde p_{1}:=\frac{\left(\al-\frac{s}{\s}\right)}{\left(\al-\frac{s}{\tilde{\s}}\right)}p_{1}$
and $\tilde p_{2}:=\frac{1}{p}-\frac{1}{\tilde p_{1}}$. Since $\tilde{\s}<1$,
by the complex interpolation, we have
\begin{align*}
\|F(u)\|_{H^{s,p}} & \les\|u\|_{L^{(\al-\frac{s}{\tilde{\sigma}})\tilde p_{1}}}^{\al-\frac{s}{\tilde{\s}}}\cdot\|u\|_{H^{\tilde{\s},\frac{s}{\tilde{\sigma}}\tilde p_{2}}}^{\frac{s}{\tilde{\sigma}}}\\
 & =\|u\|_{L^{(\al-\frac{s}{\sigma})p_{1}}}^{\al-\frac{s}{\tilde{\s}}}\cdot\|u\|_{H^{\tilde{\s},\frac{s}{\tilde{\sigma}}\tilde p_{2}}}^{\frac{s}{\tilde{\sigma}}}\\
 & \les\norm u_{L^{(\al-\frac{s}{\s})p_{1}}}^{\al-\frac{s}{\tilde{\s}}}\cdot\norm u_{L^{(\al-\frac{s}{\s})p_{1}}}^{\frac{s}{\tilde{\s}}-\frac{s}{\s}}\norm u_{H^{\s,\frac{s}{\s}p_{2}}}^{\frac{s}{\s}}\\
 & =\norm u_{L^{(\al-\frac{s}{\s})p_{1}}}^{\al-\frac{s}{\s}}\cdot\norm u_{H^{\s,\frac{s}{\s}p_{2}}}^{\frac{s}{\s}}.
\end{align*}
\end{proof}
A similar result on higher H{\"o}lder regularities can be deduced.
\begin{lem}
Let $\al\ge1$ and $m\in\Z$. Let $F:\C\rightarrow\C$ be the function
$F(z):=|z|^{\al-m}z^{m}$. Let $s\in[0,\al)$ and $p,p_{1},p_{2}\in(1,\infty)$
be exponents satisfying $\frac{1}{p}=\frac{\al-1}{p_{1}}+\frac{1}{p_{2}}$.
Then, for $u:\T^{d}\rightarrow\C$, we have
\begin{equation}
\|F(u)\|_{H^{s,p}}\les\|u\|_{L^{p_{1}}}^{\al-1}\|u\|_{H^{s,p_{2}}}.\label{eq:high fractional Holder}
\end{equation}
\end{lem}
\begin{proof}
We use an induction on $\al$. First, we focus on the case of $1\le\al\le2$.
When $0\le s\le1$, (\ref{eq:high fractional Holder}) is the well-known
fractional chain rule. Assume $1<s<\al$. Define numbers $\tilde q,r,\tilde r\in(1,\infty)$
as $\frac{s}{\tilde q}=\frac{s-1}{p_{1}}+\frac{1}{p_{2}}$, $\frac{1}{r}=\frac{1}{p}-\frac{1}{p_{2}}$,
and $\frac{1}{\tilde r}=\frac{1}{p}-\frac{1}{\tilde q}$, where we
used $s<\al$. Using $\frac{1}{\tilde r}=\frac{\al+\frac{1}{s}-2}{p_{1}}+\frac{1-\frac{1}{s}}{p_{2}}$,
we apply (\ref{eq:low fractional Holder}) to $\|(\na F)(u)\|_{H^{s-1,\tilde r}}$
with $\s=s$. Using Wirtinger derivatives, we have
\begin{align*}
\|\na(F(u))\|_{H^{s-1,p}} & =\|\na u\cdot\d_{z}F(u)+\overline{\na u}\cdot\d_{\overline{z}}F(u)\|_{H^{s-1,p}}\\
 & \les\|\na u\|_{H^{s-1,p_{2}}}\left(\|\d_{z}F(u)\|_{L^{r}}+\|\d_{\overline{z}}F(u)\|_{L^{r}}\right)\\
 & +\|\na u\|_{L^{\tilde q}}\left(\|\d_{z}F(u)\|_{H^{s-1,\tilde r}}+\|\d_{\overline{z}}F(u)\|_{H^{s-1,\tilde r}}\right)\\
 & \les\norm u_{H^{s,p_{2}}}\norm u_{L^{p_{1}}}^{\al-1}+\norm u_{L^{p_{1}}}^{\frac{s-1}{s}}\norm u_{H^{s,p_{2}}}^{\frac{1}{s}}\cdot\norm u_{L^{p_{1}}}^{\al-1-\frac{s-1}{s}}\norm u_{H^{s,p_{2}}}^{\frac{s-1}{s}}\\
 & \les\norm u_{L^{p_{1}}}^{\al-1}\norm u_{H^{s,p_{2}}},
\end{align*}
which implies
\[
\norm{F(u)}_{H^{s,p}}\les\norm{F(u)}_{L^{p}}+\|\na(F(u))\|_{H^{s-1,p}}\les\norm u_{L^{p_{1}}}^{\al-1}\norm u_{H^{s,p_{2}}}.
\]
Now, we fix an integer $N\ge2$ and assume that (\ref{eq:high fractional Holder})
holds when $\al\le N$. We claim that (\ref{eq:high fractional Holder})
holds for $\al\le N+1$ as well.

Define the numbers $q_{k},r_{k},\tilde q_{k},\tilde r_{k}$, $0\le k\le\left\lfloor s\right\rfloor -1$
as $\frac{s}{q_{k}}=\frac{k}{p_{1}}+\frac{s-k}{p_{2}}$, $\frac{1}{r_{k}}=\frac{1}{p}-\frac{1}{q_{k}}$
and $\frac{s}{\tilde q_{k}}=\frac{s-1-k}{p_{1}}+\frac{1+k}{p_{2}}$,
$\frac{1}{\tilde r_{k}}=\frac{1}{p}-\frac{1}{\tilde q_{k}}$. We have
\begin{align*}
\|\na(F(u))\|_{H^{s-1,p}} & =\|\na u\cdot\d_{z}F(u)+\overline{\na u}\cdot\d_{\overline{z}}F(u)\|_{H^{s-1,p}}\\
 & \les\sum_{k=0}^{\left\lfloor s\right\rfloor -1}\|\na u\|_{H^{s-1-k,q_{k}}}\left(\|\d_{z}F(u)\|_{H^{k,r_{k}}}+\|\d_{\overline{z}}F(u)\|_{H^{k,r_{k}}}\right)\\
 & +\sum_{k=0}^{\left\lfloor s\right\rfloor -1}\|\na u\|_{H^{k,\tilde q_{k}}}\left(\|\d_{z}F(u)\|_{H^{s-1-k,\tilde r_{k}}}+\|\d_{\overline{z}}F(u)\|_{H^{s-1-k,\tilde r_{k}}}\right)\\
 & \les\norm u_{L^{p_{1}}}^{\al-1}\norm u_{H^{s,p_{2}}},
\end{align*}
where we used complex interpolations to bound norms of $\na u$, and
used (\ref{eq:low fractional Holder}) and the induction hypothesis
to bound norms of $\d_{z}F(u)$ and $\d_{\overline{z}}F(u)$.

It follows that
\[
\norm{F(u)}_{H^{s,p}}\les\norm{F(u)}_{L^{p}}+\|\na(F(u))\|_{H^{s-1,p}}\les\norm u_{L^{p_{1}}}^{\al-1}\norm u_{H^{s,p_{2}}},
\]
which finishes the proof by induction on $N$.
\end{proof}
Next, we propose a Sobolev-Slobodeckij version of the fractional H{\"o}lder
inequality.
\begin{lem}
Fix $s\in(0,1)$, $p\in(1,\infty)$, $\al\in(0,1)$, and $F\in C^{0,\al}(\C)$.
For $u\in B_{p,p}^{s}(\T^{d})$, we have
\begin{equation}
\|F(u)\|_{B_{p/\al,p/\al}^{s\al}(\T^{d})}\les\|u\|_{B_{p,p}^{s}(\T^{d})}^{\al}.\label{eq:frac Holder x}
\end{equation}
\end{lem}
\begin{proof}
The proof naturally follows from the H{\"o}lder continuity of $F$:
\begin{align*}
\|F(u)\|_{B_{p/\al,p/\al}^{s\al}}^{p/\al} & \sim\norm{F(u)}_{L^{p/\al}}^{p/\al}+\int_{\T^{d}\times\T^{d}}\left(\frac{\left|F(u(x))-F(u(y))\right|}{|x-y|^{s\al}}\right)^{p/\al}\frac{d(x,y)}{|x-y|^{d}}\\
 & \les\norm u_{L^{p}}^{p}+\int_{\T^{d}\times\T^{d}}\left(\frac{\left|u(x)-u(y)\right|}{|x-y|^{s}}\right)^{p}\frac{d(x,y)}{|x-y|^{d}}\les\|u\|_{B_{p,p}^{s}}^{p}.
\end{align*}
\end{proof}
We also have a spacetime version of the fractional chain rule for
Besov spaces.
\begin{lem}
Let $s_{0},s_{1}>0$ be exponents satisfying $2s_{0}+s_{1}<1$. Fix
$p\in(1,\infty),\al\in(0,1)$, and a function $F\in C^{0,\al}(\C)$.
For $u:\R\times\T^{d}\rightarrow\C$, we have
\begin{equation}
\|F(u)\|_{B_{p/\al,p/\al}^{s_{0}\al}B_{p/\al,p/\al}^{s_{1}\al}}\les\|u\|_{L^{p}B_{p,p}^{2s_{0}+s_{1}}\cap B_{p,p}^{s_{0}+s_{1}/2}L^{p}}^{\al}.\label{eq:frac Holder tx}
\end{equation}
\end{lem}
\begin{proof}
From the assumptions $s_{0},s_{1}>0$, $0<2s_{0}+s_{1}<1$, and $\al\in(0,1)$,
we have $s_{0},s_{1},s_{0}\al,s_{1}\al,2s_{0}+s_{1},s_{0}+s_{1}/2\in(0,1)$.
Thus, (\ref{eq:slobo x}), (\ref{eq:slobo tx}), and (\ref{eq:frac Holder x})
are applicable to each Besov space in (\ref{eq:frac Holder tx}).

By (\ref{eq:slobo tx}), we have
\begin{align*}
\|F(u)\|_{B_{p/\al,p/\al}^{s_{0}\al}B_{p/\al,p/\al}^{s_{1}\al}} & \les\norm{F(u)}_{L^{p/\al}B_{p/\al,p/\al}^{s_{1}\al}}\\
 & +\left(\int_{\R\times\R}\frac{\norm{F(u)(t,\cdot)-F(u)(s,\cdot)}_{B_{p/\al,p/\al}^{s_{1}\al}}^{p/\al}}{|t-s|^{s_{0}p}}\cdot\frac{d(t,s)}{|t-s|}\right)^{\al/p}\\
 & =I+II.
\end{align*}
We estimate $I$ using (\ref{eq:frac Holder x}):
\[
I=\norm{F(u)}_{L^{p/\al}B_{p/\al,p/\al}^{s_{1}\al}}\les\norm u_{L^{p}B_{p,p}^{s_{1}}}^{\al}\les\norm u_{L^{p}B_{p,p}^{2s_{0}+s_{1}}}^{\al}.
\]
To estimate $II$, we further decompose $II$ using (\ref{eq:slobo x}):
\begin{align*}
II & \sim\norm{F(u)}_{B_{p/\al,p/\al}^{s_{0}\al}L^{p/\al}}\\
 & +\left(\int_{\R\times\R}\frac{1}{|t-s|^{s_{0}p}}\int_{\T^{d}\times\T^{d}}\frac{1}{|x-y|^{s_{1}p}}\right.\\
 & \left.\left|F(u)(t,x)-F(u)(s,x)-F(u)(t,y)+F(u)(s,y)\right|^{p/\al}\frac{d(x,y)}{|x-y|^{d}}\frac{d(t,s)}{|t-s|}\right)^{\al/p}\\
 & =II_{A}+II_{B}.
\end{align*}

We estimate $II_{A}$ using an argument similar to (\ref{eq:frac Holder x}):
\[
\norm{F(u)}_{B_{p/\al,p/\al}^{s_{0}\al}L^{p/\al}}\les\norm u_{B_{p,p}^{s_{0}}L^{p}}^{\al}\les\norm u_{B_{p,p}^{s_{0}+s_{1}/2}L^{p}}^{\al}.
\]

We estimate $II_{B}$. For variables $t_{0},t_{1}\in\R$ and $x_{0},x_{1}\in\T^{d}$,
let $c_{j}:=\left|u(t_{j},x_{0})-u(t_{j},x_{1})\right|$ and $d_{j}:=\left|u(t_{0},x_{j})-u(t_{1},x_{j})\right|$
then we have
\begin{align}
 & \left|F(u)(t_{0},x_{0})-F(u)(t_{0},x_{1})-F(u)(t_{1},x_{0})+F(u)(t_{1},x_{1})\right|^{1/\al}\label{eq:cd}\\
 & \les\min\left\{ \max\left\{ c_{0},c_{1}\right\} ,\max\left\{ d_{0},d_{1}\right\} \right\} \nonumber \\
 & =\max_{i,j\in\left\{ 0,1\right\} }\min\left\{ c_{i},d_{j}\right\} .\nonumber 
\end{align}
By (\ref{eq:cd}) and the symmetry of $(x,y)$ and $(t,s)$ in integrals,
we have
\begin{align*}
\left(II_{B}\right)^{p/\al}= & \int_{\R\times\R}\int_{\T^{d}\times\T^{d}}\left|F(u)(t,x)-F(u)(s,x)-F(u)(t,y)+F(u)(s,y)\right|^{p/\al}\\
 & \cdot\frac{d(x,y)}{|x-y|^{s_{1}p+d}}\cdot\frac{d(t,s)}{|t-s|^{s_{0}p+1}}\\
\text{\ensuremath{\les}} & \int_{\R\times\R}\int_{\T^{d}\times\T^{d}}\min\left\{ |u(t,x)-u(s,x)|,|u(t,x)-u(t,y)|\right\} ^{p}\\
 & \cdot\frac{d(x,y)}{|x-y|^{s_{1}p+d}}\cdot\frac{d(t,s)}{|t-s|^{s_{0}p+1}}\\
\les & \int_{|t-s|\ge|x-y|^{2}}\left|u(t,x)-u(t,y)\right|^{p}\cdot\frac{d(x,y)}{|x-y|^{s_{1}p+d}}\cdot\frac{d(t,s)}{|t-s|^{s_{0}p+1}}\\
+ & \int_{|t-s|\le|x-y|^{2}}\left|u(t,x)-u(s,x)\right|^{p}\cdot\frac{d(x,y)}{|x-y|^{s_{1}p+d}}\cdot\frac{d(t,s)}{|t-s|^{s_{0}p+1}}\\
\les & \int_{\R}\int_{\T^{d}\times\T^{d}}|u(t,x)-u(t,y)|^{p}\cdot\frac{d(x,y)}{|x-y|^{(2s_{0}+s_{1})p+d}}\cdot dt\\
+ & \int_{\T^{d}}\int_{\R\times\R}|u(t,x)-u(s,x)|^{p}\cdot\frac{d(t,s)}{|t-s|^{(s_{0}+s_{1}/2)p+1}}\cdot dx\\
= & \int_{\R}\int_{\T^{d}\times\T^{d}}\left(\frac{|u(t,x)-u(t,y)|}{|x-y|^{2s_{0}+s_{1}}}\right)^{p}\cdot\frac{d(x,y)}{|x-y|^{d}}\cdot dt\\
+ & \int_{\R\times\R}\int_{\T^{d}}\left(\frac{|u(t,x)-u(s,x)|}{|t-s|^{s_{0}+s_{1}/2}}\right)^{p}\cdot dx\cdot\frac{d(t,s)}{|t-s|}\\
\les & \|u\|_{L^{p}B_{p,p}^{2s_{0}+s_{1}}}^{p}+\norm u_{B_{p,p}^{s_{0}+s_{1}/2}L^{p}}^{p},
\end{align*}
which finishes the proof of (\ref{eq:frac Holder tx}).
\end{proof}
\begin{rem}
The parameter $2$ in (\ref{eq:frac Holder tx}) is replaceable. Indeed,
for every $\la>0$, we can show $\|F(u)\|_{B_{p/\al,p/\al}^{s_{0}\al}B_{p/\al,p/\al}^{s_{1}\al}}\les\|u\|_{L^{p}B_{p,p}^{\la s_{0}+s_{1}}\cap B_{p,p}^{s_{0}+s_{1}/\la}L^{p}}^{\al}$
by the same argument. The choice $\la=2$ is for the scaling of the
Schr{\"o}dinger operator.
\end{rem}
\subsection{Schr{\"o}dinger operators, Strichartz estimates, and atomic spaces}

We collect Strichartz estimates for linear Schr{\"o}dinger operators on
tori, properties of the atomic spaces, and Galilean invariance properties.

\subsubsection*{Schr{\"o}dinger operators}

For a function $\phi:\T^{d}\rightarrow\C$ and $t\in\R$, we
denote by $e^{it\De}\phi$ the function such that
\[
\widehat{e^{it\De}\phi}(\xi)=e^{-it\left|\xi\right|^{2}}\widehat{\phi}(\xi).
\]

For a function $f:\R\times\T^{d}\rightarrow\C$, we denote the retarded
Schr{\"o}dinger operator $K^{+}$ by 
\begin{equation}
K^{+}f(t):=\int_{-\infty}^{t}e^{i(t-s)\De}f(s)ds.\label{eq:K+}
\end{equation}

\subsubsection*{Strichartz estimates and atomic spaces}

The following are the kernel estimate and the $L_{t,x}^{4}$-Strichartz
estimate for the Schr{\"o}dinger operator on $\T$, first shown in \cite{bourgain1993fourier}:
\begin{prop}
\cite[Lemma 3.18]{bourgain1993fourier} On $\T$, for dyadic $N\in2^{\N}$
and coprime integers $l$ and $m$ such that $1\le l<m<N$ and $\left|t-\frac{l}{m}\right|\le\frac{1}{mN}$,
we have
\begin{equation}
\left|e^{it\De}\delta_{N}(t,x)\right|\les\frac{N}{\sqrt{m}\left(1+N\left|t-\frac{l}{m}\right|^{1/2}\right)}.\label{eq:Bourgain bound}
\end{equation}
\end{prop}
\begin{prop}
\cite[(2.2)]{bourgain1993fourier} On the domain $[0,T]\times\T$,
where $T>0$, for any function $f\in L_{t,x}^{3/4}+L^{1}L^{2}$, we
have 
\begin{equation}
\norm{K^{+}f}_{L_{t,x}^{4}\cap L^{\infty}L^{2}}\les_{T}\norm f_{L_{t,x}^{3/4}+L^{1}L^{2}}.\label{eq:L^4 Strichartz}
\end{equation}
\end{prop}

The following is a scale-invariant Strichartz estimate for general
tori, which is a main ingredient of the proof of Theorem \ref{thm:LWP s<a}.
This was first shown for rational tori in \cite{bourgain2015proof}.
For general tori, a subcritical version was first shown in \cite{bourgain2015proof}
and was sharpened to the critical scale in \cite{killip2016scale}.
\begin{prop}
\cite{bourgain1993fourier,bourgain2015proof,killip2016scale} Fix $p\in(\frac{2(d+2)}{d},\infty)$.
Let $\s=\frac{d}{2}-\frac{d+2}{p}$. Fix a finite interval $I\subset\R$.
For $N\in2^{\N}$, we have
\begin{equation}
\|P_{N}e^{it\De}f\|_{L_{t,x}^{p}(I\times\T^{d})}\les_{p,I}\|f\|_{H^{\s}(\T^{d})}.\label{eq:Bourgain Strichartz}
\end{equation}
\end{prop}

Next, we recall the definition of atomic spaces $U^{p}$ and $V^{p}$.
Although we will not directly use $U^{p},V^{p}$-structures to construct
the function spaces for the well-posedness, we will still use their
embedding properties. Here, we collect facts relevant to them. For
a general theory, we refer to \cite{koch2014dispersive}, \cite{herr2011global},
and \cite{hadac2009well}.
\begin{defn}[atomic spaces, \cite{herr2011global}]
Let $H$ be a separable Hilbert space. Let $\mathcal{Z}$ be the
collection of finite non-decreasing sequences $\left\{ t_{k}\right\} _{k=0}^{K}$
in $(-\infty,\infty]$. For $1\le p<\infty$, we call $a:\R\rightarrow H$
a $U^{p}$-atom if $a$ can be expressed as $a=\sum_{k=1}^{K}\chi_{[t_{k-1},t_{k})}\phi_{k}$,
$\sum_{k=1}^{K}\|\phi_{k}\|_{H}^{p}=1$. We define $U^{p}H$ as the
space of all functions $u:\R\rightarrow H$ that can be represented
as $u=\sum_{j=1}^{\infty}\la_{j}a_{j}$, where $a_{j}$ is a $U^{p}$-atom
for each $j\in\N$ and $\left\{ \la_{j}\right\} \in\ell^{1}$ is a
complex-valued sequence, equipped with the norm
\[
\|u\|_{U^{p}H}:=\inf\left\{ \sum_{j=1}^{\infty}|\la_{j}|:u=\sum_{j=1}^{\infty}\la_{j}a_{j},\la_{j}\in\C,a_{j}:U^{p}\text{-atom}\right\} .
\]
We define $V^{p}H$ as the space of all functions $u:\R\rightarrow H$
with $\norm u_{V^{p}H}<\infty$, where the norm is defined as
\[
\|u\|_{V^{p}H}^{p}:=\sup_{\left\{ t_{k}\right\} _{k=0}^{K}\in\mc Z}\sum_{k=1}^{K}\|u(t_{k})-u(t_{k-1})\|_{H}^{p},
\]
where the convention $u(\infty)=0$ is used. Then, we define $V_{rc}^{p}H$
as the subspace of $V^{p}H$ of right-continuous function $u:\R\rightarrow H$
satisfying $\lim_{t\rightarrow-\infty}u(t)=0.$ For simplicity of
notation, we omit $H$ in $U^{p}H,V^{p}H,V_{rc}^{p}H$ when $H\simeq\C$.
Based on this, we define the spaces $U_{\De}^{p}H,V_{\De}^{p}H,V_{\De,rc}^{p}H$
as the images by the map $u\mapsto e^{it\De}u$ of $U^{p}H,V^{p}H,V_{rc}^{p}H$,
respectively.

We define $Y^{s}$ as the space of $u:\R\times\T^{d}\rightarrow\C$
such that $\widehat{u}(n)$ lies in $V_{rc}^{2}$ for each $n\in\Z^{d}$
and
\[
\|u\|_{Y^{s}}^{2}:=\sum_{n\in\Z^{d}}\jp n^{2s}\|e^{it|n|^{2}}\widehat{u(t)}(n)\|_{V^{2}}^{2}<\infty.
\]
\end{defn}
While the space $Y^{s}$ is defined on the full domain $\R\times\T^{d}$,
since Strichartz estimates such as (\ref{eq:Bourgain Strichartz})
depend on the size of a time interval, we often need to restrict the
space to a finite interval. Given a time interval $I$, the $Y^{s}$
space corresponding to $I$, $Y^{s}(I)$, is the restriction of the
$Y^{s}$ space to the domain $I\times\T^{d}$.

In particular, in the proof of Theorem \ref{thm:LWP s<a}, we will
always consider $Y^{s}$ and all the other solution spaces localized
on a short time interval containing $0$ to avoid any issue with the
interval size.

The space $Y^{s}$ is used in \cite{herr2011global}, \cite{killip2016scale},
and \cite{lee2019local}. Some well-known properties of such atomic
spaces are the following propositions:
\begin{prop}
\cite{herr2011global}\label{prop:atomic space props} Fix $s\in\R$.
Fix a finite time interval $I$ and the corresponding $Y^{s}$ space.
We have the following:
\begin{enumerate}
\item Let $A$ and $B$ be disjoint subsets of $\Z^{d}$. We have
\begin{equation}
\|P_{A\cup B}u\|_{Y^{s}}^{2}=\|P_{A}u\|_{Y^{s}}^{2}+\|P_{B}u\|_{Y^{s}}^{2}.\qquad(\ell_{\xi}^{2}-structure)\label{eq:l^2_=00005Cxistructure}
\end{equation}
\item For $q>2$, we have
\begin{equation}
U_{\De}^{2}H^{s}\hook Y^{s}\hook V_{\De,rc}^{2}H^{s}\hook U_{\De}^{q}H^{s}\hook L^{\infty}H^{s}.\label{eq:U^p V^p embed}
\end{equation}
\end{enumerate}
\end{prop}
\begin{prop}[Strichartz estimates]
Fix a finite time interval $I$ and the corresponding $Y^{s}$ space.
Fix $p\in(\frac{2(d+2)}{d},\infty)$. Let $\s=\frac{d}{2}-\frac{d+2}{p}$.
Denote the diameter of a set $S\subset\Z^{d}$ by $\textrm{diam}(S)$.
We have the following estimates:
\begin{itemize}
\item Let $C\subset\Z^{d}$ be a square cube. We have the estimate
\begin{equation}
\|\chi_{I}P_{C}u\|_{L_{t,x}^{p}}\les_{p,I}\jp{\textrm{diam}(C)}^{\s}\|P_{C}u\|_{Y^{0}}.\label{eq:cube Strichartz}
\end{equation}
\item For $s\in\R$ and a function $f:I\times\T^{d}\rightarrow\C$, we have
\begin{equation}
\norm{K^{+}f}_{Y^{s}}\les_{I}\norm f_{(Y^{-s})'}.\label{eq:U2V2}
\end{equation}
\end{itemize}
\end{prop}
(\ref{eq:cube Strichartz}) is a consequence of (\ref{eq:Bourgain Strichartz})
used, for example, in \cite{killip2016scale}. It is obtained using
the atomic structure and the Galilean invariance of the $Y^{0}$ norm.

(\ref{eq:U2V2}) is a version of the $U^{2}-V^{2}$ dual estimate; see \cite{herr2011global}.\footnote{In other literature such as \cite{herr2011global}, \cite{killip2016scale},
and \cite{lee2019local}, (\ref{eq:U2V2}) is obtained from the duality
between $X^{s}$ and $Y^{s}$. Since $X^{s}\hook Y^{s}$, we still have
(\ref{eq:U2V2}).}

\subsubsection*{Galilean transforms}

We denote the Galilean transform with a shift $\xi\in\Z^{d}$ by $I_{\xi}:\mathcal{S}'(\R\times\T^{d})\rightarrow\mathcal{S}'(\R\times\T^{d})$,
where $\mc S'$ denotes the set of tempered distributions, which maps
$u:\R\times\T^{d}\rightarrow\C$ to 
\begin{equation}
I_{\xi}u(t,x)=e^{ix\cdot\xi-it|\xi|^{2}}u(t,x-2t\xi).\label{eq:Galilean}
\end{equation}

We collect some elementary properties of the Galilean transforms.
In particular, the $Y^{0}$ norm is invariant under the Galilean transforms.
\begin{prop}
\label{prop:galilean}For $u\in\mathcal{S}'(\R\times\T^{d})$ and
$\xi,\eta\in\Z^{d}$, we have the following: 
\begin{itemize}
\item $(i\d_{t}+\De)I_{\xi}u=I_{\xi}(i\d_{t}+\De)u$.
\item $I_{\xi}I_{\eta}u=I_{\xi+\eta}u$.
\item For each set $C\subset\Z^{d}$, we have $P_{C+\xi}I_{\xi}u=I_{\xi}P_{C}u$.
\item The $Y^{0}$ norm is invariant under the Galilean transforms, i.e.,
$\norm u_{Y^{0}}=\norm{I_{\xi}u}_{Y^{0}}$.
\item The spacetime Fourier transform of $I_{\xi}u$ can be written as follows:
\[
\widetilde{I_{\xi}u}(\tau,n)=\widetilde{u}(\tau+2n\cdot\xi-|\xi|^{2},n-\xi).
\]
\end{itemize}
\end{prop}

\section{\label{sec:-Zspaces}$Z^{s}$ spaces\label{sec:Z^s-spaces}}

In this section, we introduce our main function spaces $Z^{s}=Z^{s}(\R\times\T^{d})$
for the local well-posedness theorem. We fix a sufficiently small
number $\s>0$ and the corresponding parameter $p$, depending only
on $d$ and $s$:
\begin{equation}
\s=\s(d,s)\ll1\text{ and }p:=\frac{d+2}{\frac{d}{2}-\s}.\label{eq:=00005Cs,p def}
\end{equation}
$\s$ has to approach $0$ as $d\rightarrow\infty$ or $s\rightarrow0$.
It suffices to choose, for example, $\s=10^{-10^{10^{10^{10^{d+1/s}}}}}$.

\subsection{Strichartz estimates}

In this subsection, we prove that the $Y^{0}$ norm is stronger than
certain time Besov spaces. For this, we start with a lemma that helps
estimate time Besov norms of $U^{p}$-atoms.
\begin{lem}
Let $E$ be a Banach space on $\T^{d}$. Let $q\in(1,\infty)$ and
$\al\in(0,\frac{1}{q})$. Given a finite collection of disjoint intervals
$I_{1},\ldots,I_{n}\subset\R$ and functions $f_{1},\ldots,f_{n}\in B_{q,1}^{\al}E$,
we have
\begin{equation}
\norm{\sum_{j=1}^{n}f_{j}\chi_{I_{j}}}_{B_{q,\infty}^{\al}E}\les_{\al,q}\left(\sum_{j=1}^{n}\norm{f_{j}}_{B_{q,1}^{\al}E}^{q}\right)^{1/q}.\label{eq:atom estimate}
\end{equation}
Here, $\chi_{I}$ is the sharp cutoff of $I$. We emphasize that (\ref{eq:atom estimate})
is independent of the choice of $\left\{ I_{j}\right\} $.
\end{lem}
\begin{proof}
Let $r=\frac{1-\al}{\frac{1}{q}-\al}$. Since $I_{j}$'s are disjoint,
we have
\begin{equation}
\norm{\sum_{j=1}^{n}f_{j}\chi_{I_{j}}}_{B_{r,\infty}^{0}E}\les\norm{\sum_{j=1}^{n}f_{j}\chi_{I_{j}}}_{L^{r}E}\les\left(\sum_{j=1}^{n}\norm{f_{j}}_{L^{r}E}^{r}\right)^{1/r}\les\left(\sum_{j=1}^{n}\norm{f_{j}}_{B_{r,1}^{0}E}^{r}\right)^{1/r}.\label{eq:atom L^r}
\end{equation}
For each $j$, let $\varphi_{j}^{(k)}$ be an approximation of the
sharp cutoff $\chi_{I_{j}}$, i.e., $\varphi_{j}^{(k)}=\chi_{I_{j}}*\phi_{k}$,
where we set $\phi_{k}$ as $\phi_{k}(t)=k\phi(kt)$ with $\phi\in C_{c}^{\infty}(\R)$
chosen as a fixed function such that $\int_{\R}\phi dt=1$. Then,
we have $\norm{\d_{t}\varphi_{j}^{(k)}}_{L^{1}}+\norm{\varphi_{j}^{(k)}}_{L^{\infty}}\les1$.

For each $j$, we have 
\begin{align*}
\norm{\d_{t}(f_{j}\varphi_{j}^{(k)})}_{L^{1}E} & \le\norm{\d_{t}f_{j}\cdot\varphi_{j}^{(k)}}_{L^{1}E}+\norm{f_{j}\cdot\d_{t}\varphi_{j}^{(k)}}_{L^{1}E}\\
 & \les\norm{\d_{t}f_{j}}_{L^{1}E}\norm{\varphi_{j}^{(k)}}_{L^{\infty}}+\norm{f_{j}}_{L^{\infty}E}\norm{\d_{t}\varphi_{j}^{(k)}}_{L^{1}}\\
 & \les\norm{f_{j}}_{W^{1,1}E}.
\end{align*}
Thus, by (\ref{eq:B c W c B}), we have
\[
\norm{\sum_{j=1}^{n}f_{j}\varphi_{j}^{(k)}}_{B_{1,\infty}^{1}E}\les\norm{\sum_{j=1}^{n}f_{j}\varphi_{j}^{(k)}}_{W^{1,1}E}\les\sum_{j=1}^{n}\norm{f_{j}}_{W^{1,1}E}\les\sum_{j=1}^{n}\norm{f_{j}}_{B_{1,1}^{1}E}.
\]
For each $M\in2^{\N}$, we have
\begin{align*}
\norm{P_{\le M}^{t}\sum_{j=1}^{n}f_{j}\chi_{I_{j}}}_{B_{1,\infty}^{1}E} & =\lim_{k\rightarrow\infty}\norm{P_{\le M}^{t}\sum_{j=1}^{n}f_{j}\varphi_{j}^{(k)}}_{B_{1,\infty}^{1}E}\les\sum_{j=1}^{n}\norm{f_{j}}_{B_{1,1}^{1}E}.
\end{align*}
Taking $M\rightarrow\infty$, we have $\sum_{j=1}^{n}f_{j}\chi_{I_{j}}\in B_{1,\infty}^{1}E$
and
\begin{equation}
\norm{\sum_{j=1}^{n}f_{j}\chi_{I_{j}}}_{B_{1,\infty}^{1}E}\les\sum_{j=1}^{n}\norm{f_{j}}_{B_{1,1}^{1}E}.\label{eq:atom W^1,1}
\end{equation}
Since $r=\frac{1-\al}{\frac{1}{q}-\al}$ implies $(1-\frac{1}{q})/(1-\frac{1}{r})=1-\al$,
by a complex interpolation between (\ref{eq:atom L^r}) and (\ref{eq:atom W^1,1}),
we have (\ref{eq:atom estimate}).
\end{proof}
The next lemma presents a Strichartz estimate in time Besov spaces.
In the proof of the next lemma, we will use a real interpolation technique,
which is also used in many contexts, e.g., \cite{keel1998endpoint}.
We will first show the estimate with the spacetime frequency domain restricted
to a fixed support. Such an estimate gives
information only on Besov spaces of parameter $\infty$, i.e.,
$B_{p,\infty}^{\al}$. We will improve that estimate by using (\ref{eq:Besov real interp E s0s1})
with perturbed choices of exponents. This kind of technique will be
frequently used in later proofs.

Let $\psi:\R\rightarrow[0,\infty)$ be a smooth even bump function
such that $\psi|_{[-1,1]}\equiv1$ and $\text{supp}(\psi)\subset[-\frac{11}{10},\frac{11}{10}]$.
\begin{lem}
\label{lem:Besov stri}Let $\s$ and $p$ be as in (\ref{eq:=00005Cs,p def}).
Fix $\al\in(0,\frac{1}{p})$. Let $\be=\s+2\al$. For every $N\in2^{\N}$
and $f\in L^{1}L^{2}$, we have
\begin{equation}
\norm{\psi K^{+}f_{N}}_{B_{p,1}^{\al}L^{p}}\les_{\al,\psi}N^{\be}\norm f_{L^{1}L^{2}},\label{eq:Besov Stri}
\end{equation}
where $f_{N}=P_{N}f$.
\end{lem}
\begin{proof}
Let $u=K^{+}f_{N}$. We first claim that for $\tilde{\al}\in(0,\frac{1}{p})$
and $M\in2^{\N}$, we have
\begin{equation}
\norm{P_{M}^{t}(\psi u)}_{L^{p}L^{p}}\les M^{-\tilde{\al}}N^{\s+2\tilde{\al}}\norm f_{L^{1}L^{2}}.\label{eq:Besov Stri claim}
\end{equation}
Once we have (\ref{eq:Besov Stri claim}), we plug in $\al_{0}\in(0,\al)$
and $\al_{1}\in(\al,\frac{1}{p})$ to obtain
\begin{equation}
\norm{\psi u}_{B_{p,\infty}^{\al_{j}}L^{p}}\les N^{\s+2\al_{j}}\norm f_{L^{1}L^{2}}\qquad j=0,1.\label{eq:Besov stri to interp}
\end{equation}
Applying a real interpolation of parameter $1$ to (\ref{eq:Besov stri to interp})
gives (\ref{eq:Besov Stri}) for $\al\in(0,\frac{1}{p})$.

Now, we prove the claim (\ref{eq:Besov Stri claim}). When $M\sim N^{2}$,
(\ref{eq:Besov Stri claim}) is merely (\ref{eq:Bourgain Strichartz}).
From now, we assume $M\nsim N^{2}$. With this assumption, we have
\[
\left|\tau+\left|\xi\right|^{2}\right|\sim\max\left\{ M,N^{2}\right\} .
\]
Since $(i\d_{t}+\De)u=f_{N}$, we have $(i\d_{t}+\De)P_{M}^{t}(\psi u)=P_{M}^{t}(\psi f_{N}+i\psi_{t}u)$.
Thus, we have
\begin{align*}
\norm{P_{M}^{t}(\psi u)}_{L_{t,x}^{p}} & \les M^{\frac{1}{2}-\frac{1}{p}}N^{d\left(\frac{1}{2}-\frac{1}{p}\right)}\norm{P_{M}^{t}(\psi u)}_{L_{t,x}^{2}}\\
 & =M^{\frac{1}{2}-\frac{1}{p}}N^{d\left(\frac{1}{2}-\frac{1}{p}\right)}\norm{\frac{1}{\tau+\left|\xi\right|^{2}}\F_{t,x}\left(P_{M}^{t}(\psi f_{N}+i\psi_{t}u)\right)}_{L_{\tau}^{2}\ell_{\xi}^{2}}\\
 & \les\frac{M^{\frac{1}{2}-\frac{1}{p}}N^{d\left(\frac{1}{2}-\frac{1}{p}\right)}}{\max\left\{ M,N^{2}\right\} }\norm{P_{M}^{t}(\psi f_{N}+i\psi_{t}u)}_{L_{t,x}^{2}}\\
 & \les\frac{M^{1-\frac{1}{p}}N^{d\left(\frac{1}{2}-\frac{1}{p}\right)}}{\max\left\{ M,N^{2}\right\} }\norm{\psi f_{N}+i\psi_{t}u}_{L^{1}L^{2}}\\
 & \les M^{-\tilde{\al}}N^{\s+2\tilde{\al}}\norm f_{L^{1}L^{2}}.
\end{align*}
Here, we have $\frac{M^{1-\frac{1}{p}}N^{d\left(\frac{1}{2}-\frac{1}{p}\right)}}{\max\left\{ M,N^{2}\right\} }\les M^{-\tilde{\al}}N^{\s+2\tilde{\al}}$
since the scales match as $2\left(1-\frac{1}{p}\right)+d\left(\frac{1}{2}-\frac{1}{p}\right)-2=\frac{d}{2}-\frac{d+2}{p}=\s=-2\tilde{\al}+(\s+2\tilde{\al})$
and $-\tilde{\al}\in(-\frac{1}{p},0)\subset(-\frac{1}{p},1-\frac{1}{p})$.
\end{proof}
Now, we transfer Lemma \ref{lem:Besov stri} to an estimate regarding
$Y^{0}$.
\begin{lem}
\label{lem:Besov Y^0}Let $\s,p,\al$, and $\be$ be as in Lemma \ref{lem:Besov stri}.
For $u\in Y^{0}$, we have 
\begin{equation}
\left(\sum_{N\in2^{\N}}N^{-2\be}\norm{\psi u_{N}}_{B_{p,1}^{\al}L^{p}}^{2}\right)^{1/2}=\norm{\psi u}_{\ell_{-\be}^{2}B_{p,1}^{\al}L^{p}}\les\norm u_{Y^{0}}.\label{eq:Besov Y^0}
\end{equation}
\end{lem}
\begin{proof}
Let $\al\in(0,\frac{1}{p})$. By (\ref{eq:atom estimate}) and (\ref{eq:Besov Stri}),
for $\left\{ t_{k}\right\} _{k=1}^{K}\in\mc Z$, $N\in2^{\N}$, and
$\left\{ \phi_{k}\right\} _{k=1}^{K}$ in $L^{2}(\T^{d})$, we have
\[
\norm{\psi P_{N}\sum_{j=1}^{K}e^{it\De}\phi_{j}\cdot\chi_{[t_{j-1},t_{j})}}_{B_{p,\infty}^{\al}L^{p}}\les_{p,\al}N^{\be}\norm{\norm{\phi_{j}}_{L^{2}}}_{\ell_{j}^{p}}.
\]
In other words, we have
\[
\norm{\psi e^{it\De}f_{N}}_{B_{p,\infty}^{\al}L^{p}}\les_{p,\al}N^{\be}
\]
for every $U^{p}L^{2}$-atom $f=\sum_{j=1}^{K}\phi_{j}\cdot\chi_{[t_{j-1},t_{j})}$,
i.e., $\sum_{j=1}^{K}\norm{\phi_{j}}_{L^{2}}^{p}=1$. Using the embedding
$Y^{0}\hook U_{\De}^{p}L^{2}$, we have 
\[
\norm{\psi u_{N}}_{B_{p,\infty}^{\al}L^{p}}\les N^{\be}\norm{u_{N}}_{Y^{0}}
\]
for $u\in Y^{0}$. Then, by the real interpolation argument used in
Lemma \ref{lem:Besov stri}, we arrive at $\norm{\psi u_{N}}_{B_{p,1}^{\al}L^{p}}\les N^{\be}\norm{u_{N}}_{Y^{0}}$,
which implies (\ref{eq:Besov Y^0}) due to $Y^{0}=\ell^{2}Y^{0}$.
\end{proof}

\subsection{$Z^{s}$ spaces}

We now introduce a function space for the well-posedness problem.
In the introduction, we explained a limitation of the conventional
candidate, the $Y^{s}$ space. Large data a priori bounds were obtained
in \cite{herr2011global}, \cite{killip2016scale}, and \cite{lee2019local}
using an argument that works only when the nonlinearity $\calN(u)$
is either algebraic or sufficiently regular, i.e., $a\ge2$.

Here, we overcome this difficulty by introducing a new function space
$Z^{s}$. We hope to keep bringing up the favorable embedding properties
of the $Y^{s}$ space. Still, using an $L_{t}^{p}$-structure for the
time variable, we make the $Z^{s}$ norm shrink to zero as the length of the time interval goes to zero.

Indeed, we look for $Z^{s}$ weaker than $Y^{s}$, i.e., $Z^{s}\hookleftarrow Y^{s}$.
While the embedding $Z^{s}\hookleftarrow Y^{s}$ makes the retarded
estimate $K^{+}:(Z^{-s})'\rightarrow Z^{s}$ immediate, it also makes
the bilinear estimate based on $Z^{s}$ stronger than that based on
$Y^{s}$. Thus, we focus on the structure of the norm we use in our
proof of the main bilinear estimate.

In Section \ref{subsec:Bi-linear-estimates}, we will prove our bilinear
estimate using the identity
\begin{equation}
\int_{\R\times\T^{d}}v\overline{w}Adxdt=\int_{\R\times\T^{d}}I_{\xi}v\cdot\overline{I_{\xi}w}\cdot J_{\xi}Adxdt\label{eq:IIJ identity intro}
\end{equation}
for $v,w,A:\R\times\T^{d}\rightarrow\C$ and $\xi\in\Z^{d}$, where
$J_{\xi}$ denotes the shear effect $J_{\xi}A(t,x):=A(t,x-2t\xi)$.
To prove the bilinear estimate, we will partition the spatial frequency
domain into congruent cubes $C\subset\Z^{d}$, then apply (\ref{eq:IIJ identity intro})
to each $C$ with the shift $\xi$ that translates the center of $C$ to the
origin. Combining spacetime regularities of $v,w,A:\R\times\T^d\rightarrow\C$ and the shear
effect $J_{\xi}$ will give an extra gain of decay that we mentioned
in the introduction. To use the Galilean structure, frequency
partitioning, and time Besov regularity, we are motivated to consider
a norm of the form
\[
\max_{\substack{R\in2^{\N}}
}R^{-\be}\norm{\norm{\psi P_{\le R}I_{Rk}u}_{B_{p,1}^{\al}L^{p}}}_{\ell^{2}(k\in\Z^{d})},
\]
where $\al$ and $\be$ are some scaling-critical exponents.

Inspired by this observation, we design a new space $Z^{s}$.
\begin{defn}
Let $\s$ and $p$ be as in (\ref{eq:=00005Cs,p def}). We define
$Z^{s}=Z^{s}(\R\times\T^{d})$ as a Banach space given by the norm
\begin{align}
\norm u_{Z^{s}} & =\max_{q\in\text{ \ensuremath{\left[p,\frac{1}{\s}\right]} }}\norm{\norm{\psi u_{N}}_{L^{q}L^{r}}}_{\ell_{s-\s}^{2}(N\in2^{\N})}\label{eq:Z^s def}\\
 & +\max_{\al\in\left[\s,\frac{1}{p}-\s\right]}\norm{\max_{\substack{R\in2^{\N}\\
R\le8N
}
}R^{-\be}\norm{\norm{\psi P_{\le8R}I_{Rk}u_{N}}_{B_{p,1}^{\al}L^{p}}}_{\ell^{2}(k\in\Z^{d})}}_{\ell_{s}^{2}(N\in2^{\N})},\nonumber 
\end{align}
where scaling conditions $\frac{d}{r}=\frac{d}{2}-\s-\frac{2}{q}$
and $\be=\s+2\al$ are imposed.

Similarly, $Z^{0}$ is defined as (\ref{eq:Z^s def}) with $s$ replaced
by $0$.
\end{defn}
\begin{rem}
As seen in (\ref{eq:Besov Stri}), the scaling conditions $\be=\s+2\al$
and $p=\frac{d+2}{\frac{d}{2}-\s}$ make the embedding $Y^{0}\hook\ell_{-\be}^{2}B_{p,1}^{\al}L^{p}$
scaling invariant with the scale of $C^{0}L^{2}$. For the same reason,
$Z^{s}$ and $Z^{0}$ spaces have scales identical to those of $Y^{s}$
and $Y^{0}$, respectively.

In view of complex interpolations, $\norm u_{Z^{s}}$ is equivalent
to (\ref{eq:Z^s def}) with $q$ and $\al$ at the endpoints of the
intervals.
\end{rem}
We collect elementary properties of the space $Z^{s}$.
\begin{lem}
\label{lem:summary of embeds}We have the following properties:
\begin{itemize}
\item For a finite interval $I\subset\R$, we have the embedding
\begin{equation}
\ell_{s}^{2}(Z^{0})'=(Z^{-s})'\hook(Y^{-s})'\stackrel{K^{+}}{\longrightarrow}Y^{s}\hook Z^{s}=\ell_{s}^{2}Z^{0}.\label{eq:Z^s-Y^s}
\end{equation}
\item For a finite interval $I\subset\R$, we have
\begin{equation}
\norm{u\cdot\chi_{I}}_{Z^{s}}\les\norm u_{Z^{s}}.\label{eq:step Z^s}
\end{equation}
This estimate is uniform in the choice of $I$.
\item For $u\in Z^{s}$, we have
\begin{equation}
\lim_{T\rightarrow0^{+}}\norm{u\cdot\chi_{[0,T]}}_{Z^{s}}=0.\label{eq:shrink Z^s}
\end{equation}
\item Let $q\in[p,\frac{1}{\s}]$ and $r$ be parameters such that $\frac{2}{q}+\frac{d}{r}=\frac{d}{2}-\s$.
We have
\begin{equation}
\norm{\psi u}_{L^{q}H^{s-\s,r}}\les\norm{\psi u}_{L^{q}B_{r,2}^{s-\s}}\les\norm{\psi u}_{\ell_{s-\s}^{2}L^{q}L^{r}}\les\norm u_{Z^{s}}.\label{eq:Sobolev embed Z^s}
\end{equation}
\item Let $\al\in[\s,\frac{1}{p}-\s]$ and $\be=\s+2\al$. We have
\begin{equation}
\norm{\psi u}_{B_{p,2}^{\al}B_{p,2}^{s-\be}}\les\norm{\psi u}_{\ell_{s-\be}^{2}B_{p,1}^{\al}L^{p}}\les\norm u_{Z^{s}}.\label{eq:time Besov embed Z^s}
\end{equation}
\end{itemize}
Similar properties hold with $s$ replaced by $0$.
\end{lem}

\begin{proof}
In (\ref{eq:Z^s-Y^s}), we show that $Z^{s}$ is weaker than $Y^{s}$.
In (\ref{eq:time Besov embed Z^s}), we bring the Strichartz estimate
(\ref{eq:Besov Y^0}) to $Z^{s}$. Most importantly, in (\ref{eq:shrink Z^s}),
we show that the $Z^{s}$ norm of a free evolution converges to $0$
as the time cutoff shrinks.

1. We show 
\begin{equation}
Y^{s}\hook Z^{s}.\label{eq:Z^s-Y^s claim}
\end{equation}
Once we have (\ref{eq:Z^s-Y^s claim}), by duality and (\ref{eq:U2V2}),
(\ref{eq:Z^s-Y^s}) follows.

We decompose the equation for $Z^{s}$ as follows:
\begin{align*}
\norm u_{Z^{s}} & =\max_{\substack{q\in\text{ \ensuremath{\left[p,\frac{1}{\s}\right]} }\\
\frac{d}{r}=\frac{d}{2}-\s-\frac{2}{q}
}
}\norm{\norm{\psi u_{N}}_{L^{q}L^{r}}}_{\ell_{s-\s}^{2}(N\in2^{\N})}\\
 & +\max_{\substack{\al\in\left[\s,\frac{1}{p}-\s\right]\\
\be=\s+2\al
}
}\norm{\max_{\substack{R\in2^{\N}\\
R\le8N
}
}R^{-\be}\norm{\norm{\psi P_{\le8R}I_{Rk}u_{N}}_{B_{p,1}^{\al}L^{p}}}_{\ell^{2}(k\in\Z^{d})}}_{\ell_{s}^{2}(N\in2^{\N})}\\
 & =I+II.
\end{align*}

We bound $I$ first. We fix $q\in\left[p,\frac{1}{\s}\right]$ and
$r\in(2,\infty)$ satisfying $\frac{2}{q}+\frac{d}{r}=\frac{d}{2}-\s$.
Since the spacetime norms $L^{\infty}L^{2},\ell_{-\s}^{2}L^{\infty}L^{(1/2-\s/d)^{-1}},\ell_{-\s}^{2}L^{p}L^{p},\ell_{-\s}^{2}L^{q}L^{r}$,
and $Y^{0}$ have the same scale, a complex interpolation between
$\norm{\psi u}_{\ell_{-\s}^{2}L^{\infty}L^{(1/2-\s/d)^{-1}}}\les\norm{\psi u}_{\ell^{2}L^{\infty}L^{2}}\les\norm u_{Y^{0}}$
and $\norm{\psi u}_{\ell_{-\s}^{2}L^{p}L^{p}}\les\norm u_{Y^{0}}$
gives $\norm{\psi u}_{\ell_{-\s}^{2}L^{q}L^{r}}\les\norm u_{Y^{0}}$.
Therefore, we have
\[
I=\max_{\substack{q\in\text{ \ensuremath{\left[p,\frac{1}{\s}\right]} }\\
\frac{d}{r}=\frac{d}{2}-\s-\frac{2}{q}
}
}\norm{\norm{\psi u_{N}}_{L^{q}L^{r}}}_{\ell_{s-\s}^{2}(N\in2^{\N})}\les\norm u_{\ell_{s}^{2}Y^{0}}=\norm u_{Y^{s}}.
\]

To bound $II$, we fix $\al\in\left[\s,\frac{1}{p}-\s\right]$ and
$R\in2^{\N}$. Let $\be=2\al+\s$. By (\ref{eq:Besov Y^0}) and the
Galilean invariance of the $Y^{0}$ norm, we have
\begin{align*}
R^{-\be}\norm{\norm{\psi P_{\le8R}I_{Rk}u_{N}}_{B_{p,1}^{\al}L^{p}}}_{\ell^{2}(k\in\Z^{d})} & \les\norm{\norm{P_{\le8R}I_{Rk}u_{N}}_{Y^{0}}}_{\ell^{2}(k\in\Z^{d})}\\
 & =\norm{\norm{I_{Rk}P_{-Rk+[-8R,8R]^{d}}u_{N}}_{Y^{0}}}_{\ell^{2}(k\in\Z^{d})}\\
 & =\norm{\norm{P_{-Rk+[-8R,8R]^{d}}u_{N}}_{Y^{0}}}_{\ell^{2}(k\in\Z^{d})}\\
 & \les\norm{u_{N}}_{Y^{0}},
\end{align*}
which implies $II\les\norm u_{Y^{0}}$ immediately.

2. In fact, the stability under time cutoffs holds true for any time
Besov space. We prove a more general statement: For any Banach space
$E$ on $\T^{d}$, $\al\in(0,\frac{1}{p})$, $q\in[1,\infty]$, $f\in B_{p,q}^{\al}E$,
and any interval $I\subset\R$, we have
\begin{equation}
\norm{f\chi_{I}}_{B_{p,q}^{\al}E}\les_{\al,p,q}\norm f_{B_{p,q}^{\al}E}.\label{eq:step estimate}
\end{equation}
Once we have (\ref{eq:step estimate}), (\ref{eq:step Z^s}) follows
directly.

For $\al_{0}\in(0,\al)$ and $\al_{1}\in(\al,\frac{1}{p})$, by (\ref{eq:atom estimate}),
we have
\begin{equation}
\norm{f\chi_{I}}_{B_{p,\infty}^{\al_{j}}E}\les\norm f_{B_{p,1}^{\al_{j}}E},\qquad j=0,1.\label{eq:step to interp}
\end{equation}
Thus, applying a real interpolation of parameter $q$ to (\ref{eq:step to interp})
gives (\ref{eq:step estimate}).

3. We claim that for $N,R\in2^{\N}$, $\al\in[\s,\frac{1}{p}-\s]$,
$k\in\Z^{d}$, and $u\in Z^{s}$, we have
\begin{equation}
\lim_{T\rightarrow0^{+}}\norm{P_{\le8R}I_{Rk}u_{N}\cdot\chi_{[0,T]}}_{B_{p,1}^{\al}L^{p}}=0.\label{eq:shrinkclaim}
\end{equation}
Once we have (\ref{eq:shrinkclaim}), since we can deduce by (\ref{eq:step estimate})
that
\begin{align*}
 & \max_{\substack{q\in\text{\ensuremath{\left\{  p,\frac{1}{\s}\right\} } }\\
\frac{d}{r}=\frac{d}{2}-\s-\frac{2}{q}
}
}\norm{\sup_{T\in[0,1]}\norm{\psi u_{N}\cdot\chi_{[0,T]}}_{L^{q}L^{r}}}_{\ell_{s-\s}^{2}(N\in2^{\N})}\\
 & +\max_{\substack{\al\in\left\{ \s,\frac{1}{p}-\s\right\} \\
\be=\s+2\al
}
}\norm{\max_{\substack{R\in2^{\N}\\
R\le8N
}
}R^{-\be}\norm{\sup_{T\in[0,1]}\norm{\psi P_{\le8R}I_{Rk}u_{N}\cdot\chi_{[0,T]}}_{B_{p,1}^{\al}L^{p}}}_{\ell^{2}(k\in\Z^{d})}}_{\ell_{s}^{2}(N\in2^{\N})}<\infty,
\end{align*}
by the dominated convergence theorem, it follows that
\begin{align*}
 & \lim_{T\rightarrow0^{+}}\norm{u\cdot\chi_{[0,T]}}_{Z^{s}}\\
 & =\lim_{T\rightarrow0^{+}}\max_{\substack{q\in\text{\ensuremath{\left\{  p,\frac{1}{\s}\right\} } }\\
\frac{d}{r}=\frac{d}{2}-\s-\frac{2}{q}
}
}\norm{\norm{\psi u_{N}\cdot\chi_{[0,T]}}_{L^{q}L^{r}}}_{\ell_{s-\s}^{2}(N\in2^{\N})}\\
 & +\lim_{T\rightarrow0^{+}}\max_{\substack{\al\in\left\{ \s,\frac{1}{p}-\s\right\} \\
\be=\s+2\al
}
}\norm{\max_{\substack{R\in2^{\N}\\
R\le8N
}
}R^{-\be}\norm{\norm{\psi P_{\le8R}I_{Rk}u_{N}\cdot\chi_{[0,T]}}_{B_{p,1}^{\al}L^{p}}}_{\ell^{2}(k\in\Z^{d})}}_{\ell_{s}^{2}(N\in2^{\N})}\\
 & =\max_{\substack{q\in\text{\ensuremath{\left\{  p,\frac{1}{\s}\right\} } }\\
\frac{d}{r}=\frac{d}{2}-\s-\frac{2}{q}
}
}\norm{\lim_{T\rightarrow0^{+}}\norm{\psi u_{N}\cdot\chi_{[0,T]}}_{L^{q}L^{r}}}_{\ell_{s-\s}^{2}(N\in2^{\N})}\\
 & +\max_{\substack{\al\in\left\{ \s,\frac{1}{p}-\s\right\} \\
\be=\s+2\al
}
}\norm{\max_{\substack{R\in2^{\N}\\
R\le8N
}
}R^{-\be}\norm{\lim_{T\rightarrow0^{+}}\norm{\psi P_{\le8R}I_{Rk}u_{N}\cdot\chi_{[0,T]}}_{B_{p,1}^{\al}L^{p}}}_{\ell^{2}(k\in\Z^{d})}}_{\ell_{s}^{2}(N\in2^{\N})}\\
 & =0.
\end{align*}

Now, we show the claim (\ref{eq:shrinkclaim}). Let $f=P_{\le8R}I_{Rk}u_{N}$.
Choose $\al_{+}\in(\al,\frac{1}{p})$. Since $\norm u_{Z^{s}}<\infty$,
we have $\norm f_{B_{p,1}^{\al_{+}}L^{p}}<\infty.$ By (\ref{eq:step estimate}),
for a dyadic number $M\in2^{\N}$, we have
\begin{align*}
\sup_{T\in[0,1]}\norm{P_{\ge M}^{t}\left(f\cdot\chi_{[0,T]}\right)}_{B_{p,1}^{\al}L^{p}} & \les M^{\al-\al_{+}}\sup_{T\in[0,1]}\norm{f\cdot\chi_{[0,T]}}_{B_{p,1}^{\al_{+}}L^{p}}\\
 & \les M^{\al-\al_{+}}\norm f_{B_{p,1}^{\al_{+}}L^{p}}.
\end{align*}
As a consequence, we have (\ref{eq:shrinkclaim}).

4. Since $q,r\in(2,\infty)$, applying $\ell_{s-\s}^{2}L^{q}L^{r}\hook L^{q}\ell_{s-\s}^{2}L^{r}=L^{q}B_{r,2}^{s-\s}$
to (\ref{eq:Z^s def}), we have (\ref{eq:Sobolev embed Z^s}).

5. Since $p>2$, we have $B_{p,2}^{\al}B_{p,2}^{s-\be}=\ell_{\al;\tau}^{2}L_{t}^{p}\ell_{s-\be}^{2}L_{x}^{p}\hookleftarrow\ell_{\al;\tau}^{2}\ell_{s-\be}^{2}L_{t}^{p}L_{x}^{p}=\ell_{s-\be}^{2}\ell_{\al;\tau}^{2}L_{t}^{p}L_{x}^{p}=\ell_{s-\be}^{2}B_{p,2}^{\al}L^{p}$.
Thus, we have
\[
\norm{\psi u}_{B_{p,2}^{\al}B_{p,2}^{s-\be}}\les\norm{\psi u}_{\ell_{s-\be}^{2}B_{p,1}^{\al}L^{p}}=\norm{N^{-\be}\norm{\psi P_{\le8N}u_{N}}_{B_{p,1}^{\al}L^{p}}}_{\ell_{s}^{2}(N\in2^{\N})}\les\norm u_{Z^{s}}.
\]
\end{proof}

\subsection{\label{subsec:Bi-linear-estimates}Bilinear estimates}

In this subsection, we prove the main bilinear estimate (\ref{eq:strip1 claim1-1}), as introduced in (\ref{eq:key prop}).
To obtain the decay in (\ref{eq:strip1 claim1-1}), we use the identity
for functions $f,g,A:\R\times\T^{d}\rightarrow\C$:
\begin{equation}
\int_{\R\times\T^{d}}f\overline{g}Adxdt=\int_{\R\times\T^{d}}I_{\xi}f\cdot\overline{I_{\xi}g}\cdot J_{\xi}Adxdt,\label{eq:If Ig JA}
\end{equation}
where
\[
J_{\xi}A(t,x):=A(t,x-2t\xi).
\]
We first estimate $J_{\xi}A$ in Besov spaces.
\begin{lem}
For a dyadic number $N\in2^{\N}$, an integer point $k\in\Z^{d}\setminus\left\{ 0\right\} $,
and a function $A\in B_{1,1}^{\frac{1}{2}}L^{1}$, we have
\begin{equation}
\norm{J_{Nk}A_{N}}_{B_{2,2}^{-\frac{1}{8}}L^{2}}\les N^{\frac{d}{2}-\frac{1}{4}}\left(N^{-\frac{1}{4}}+\left|k\right|^{-\frac{1}{8}}\right)\norm A_{B_{1,1}^{\frac{1}{2}}L^{1}}.\label{eq:A_k bound1}
\end{equation}
\end{lem}
\begin{proof}
We choose a coordinate vector $e_{j}$ such that $\kappa=k\cdot e_{j}\sim\left|k\right|$.
Let $k'=k-\kappa e_{j}$. Since $\widetilde{J_{Nk}A}(\tau,\xi)=\widetilde{A}(\tau+2Nk\cdot\xi,\xi)$,
for a dyadic number $M\in2^{\N}$, we have
\begin{align*}
\norm{J_{Nk}P_{M}^{t}A_{N}}_{B_{2,2}^{-\frac{1}{8}}L^{2}}^{2} & \les\sum_{\substack{\xi\in\Z^{d}}
}\int_{\R}\jp{\tau}^{-\frac{1}{4}}\left|\F_{t,x}\left(J_{Nk}P_{M}^{t}A_{N}\right)(\tau,\xi)\right|^{2}d\tau\\
 & =\sum_{\substack{\xi\in\Z^{d}}
}\int_{\R}\jp{\tau}^{-\frac{1}{4}}\left|\F_{t,x}\left(P_{M}^{t}A_{N}\right)\left(\tau+2Nk\cdot\xi,\xi\right)\right|^{2}d\tau\\
 & \les\sum_{\substack{\xi\in[-N,N]^{d}}
}\int_{-2Nk\cdot\xi+[-10M,10M]}\jp{\tau}^{-\frac{1}{4}}\left|\widetilde{A}(\tau+2Nk\cdot\xi,\xi)\right|^{2}d\tau\\
{\color{blue}{\color{blue}}} & \les\sum_{\substack{\xi\in[-N,N]^{d}}
}M\jp{Nk\cdot\xi}^{-\frac{1}{4}}\cdot\norm{\tilde A}_{L_{\tau}^{\infty}\ell_{\xi}^{\infty}}^{2}\\
{\color{blue}} & =\sum_{\xi'\in[-N,N]^{d-1}}\sum_{\substack{n\in[-N,N]}
}M\jp{Nk'\cdot\xi'+N\kappa n}^{-\frac{1}{4}}\cdot\norm{\tilde A}_{L_{\tau}^{\infty}\ell_{\xi}^{\infty}}^{2}\\
 & \les MN^{d-1}\left(1+\left|\kappa\right|^{-\frac{1}{4}}N^{\frac{1}{2}}\right)\norm A_{L_{t,x}^{1}}^{2}\\
{\color{blue}} & \les MN^{d-\frac{1}{2}}\left(N^{-\frac{1}{2}}+\left|k\right|^{-\frac{1}{4}}\right)\norm A_{L_{t,x}^{1}}^{2}.
\end{align*}
Taking a square root, summing over $M\in2^{\N}$, and applying the
triangle inequality, we obtain (\ref{eq:A_k bound1}), finishing the
proof.
\end{proof}
Meanwhile, for $1<r<q<\infty$, we have the embedding
\begin{equation}
\norm{J_{Nk}A_{N}}_{B_{q,\infty}^{0}L^{q}}\les\norm{J_{Nk}A_{N}}_{L_{t,x}^{q}}=\norm{A_{N}}_{L_{t,x}^{q}}\les N^{d\left(\frac{1}{r}-\frac{1}{q}\right)}\norm A_{B_{r,r}^{\frac{1}{r}-\frac{1}{q}}L^{r}}.\label{eq:A_k bound2}
\end{equation}
Interpolating between (\ref{eq:A_k bound1}) and (\ref{eq:A_k bound2}),
we have the following lemma.
\begin{lem}
\label{lem:A_k bound}Let $\Omega$ be the open tetrahedron whose
vertices are $\left(\frac{1}{2},1,\frac{1}{8}\right),\left(0,1,0\right),\left(1,1,0\right)$,
and $\left(0,0,0\right)$. Let $q,r$, and $\rho$ be parameters such
that $\left(\frac{1}{q},\frac{1}{r},\rho\right)$ lies in $\Omega$.
For a dyadic number $N\in2^{\N}$, an integer point $k\in\Z^{d}\setminus\left\{ 0\right\} ,$
and a function $A:\R\times\T^{d}\rightarrow\C,$ we have
\begin{equation}
\norm{J_{Nk}A_{N}}_{B_{q,\infty}^{-\rho}L^{q}}\les_{q,r,\rho}\left(N^{-2\rho}+|k|^{-\rho}\right)N^{d\left(\frac{1}{r}-\frac{1}{q}\right)-2\rho}\norm A_{B_{r,r}^{\frac{1}{r}-\frac{1}{q}}L^{r}}.\label{eq:A_k bound}
\end{equation}
\end{lem}
We choose $\s_{1},\ldots,\s_{4}>0$ satisfying 
\begin{equation}
\s\ll\s_{1}\ll\s_{2}\ll\s_{3}\ll\s_{4}\ll1.\label{eq:ss1s2<<}
\end{equation}
For example, if we choose $\s=10^{-10^{10^{10^{10^{d+1/s}}}}}$ for
(\ref{eq:=00005Cs,p def}), we can choose $\s_{j}$'s as $\s_{1}=10^{-10^{10^{10^{d+1/s}}}},\s_{2}=10^{-10^{10^{d+1/s}}},\s_{3}=10^{-10^{d+1/s}}$,
and $\s_{4}=10^{-d-1/s}$.
\begin{lem}
\label{lem:strip1 claim1-1}Let $q_{0},r_{0}$, and $\theta$ be the exponents such that $\frac{1}{q_{0}}=\frac{2+\s_{3}}{d+2},\frac{1}{r_{0}}=\frac{2+\s_{3}+\s_{2}}{d+2}$,
and $\theta=\frac{2}{q_{0}}+\frac{d}{r_{0}}-2$. For $A\in B_{r_{0},r_{0}}^{\frac{1}{r_{0}}-\frac{1}{q_{0}}}B_{r_{0},\infty}^{\theta}$,
$u,v\in Z^{0}$, and frequencies $N,R\in2^{\N}$ such that $N\ge32R$,
we have
\begin{align}
\left|\int_{\R\times\T^{d}}\psi^{2}u_{N}\overline{v}A_{R}dxdt\right| & \les\norm u_{Z^{0}}\norm v_{Z^{0}}\left(\text{\ensuremath{\jp{N/R}^{-\s_{1}}}}+R^{-2\s_{1}}\right)R^{\theta}\norm{A_{R}}_{B_{r_{0},r_{0}}^{\frac{1}{r_{0}}-\frac{1}{q_{0}}}L^{r_{0}}}.\label{eq:strip1 claim1-1}
\end{align}
\end{lem}
\begin{proof}
For simplicity of notation, we denote $r=r_{0}$ and $q=q_{0}$ in
this proof. Let $\al$ and $\be$ be the parameters satisfying $\frac{1}{q}+\s_{1}+2\left(\frac{1}{p}-\al\right)=1$
and $\be=2\al+\s$. Since $\int_{\R\times\T^{d}}\psi^{2}u_{N}\overline{v_{M}}A_{R}dxdt$
is zero whenever $M>4N$ or $M<\frac{1}{4}N$, (\ref{eq:strip1 claim1-1})
follows once we have for $M\sim N$ the estimate
\[
\left|\int_{\R\times\T^{d}}\psi^{2}u_{N}\overline{v_{M}}A_{R}dxdt\right|\les\norm u_{Z^{0}}\norm v_{Z^{0}}\left(\text{\ensuremath{\jp{N/R}^{-\s_{1}}}}+R^{-2\s_{1}}\right)R^{\theta}\norm{A_{R}}_{B_{r,r}^{\frac{1}{r}-\frac{1}{q}}L^{r}}.
\]
Due to (\ref{eq:If Ig JA}), we have 
\[
\int_{\R\times\T^{d}}v\overline{w}Adxdt=\int_{\R\times\T^{d}}I_{Rk}v\cdot\overline{I_{Rk}w}\cdot J_{Rk}Adxdt
\]
for $v,w:\R\times\T^{d}\rightarrow\C$ and $k\in\Z^{d}$. We have
\begin{align*}
 & \left|\int_{\R\times\T^{d}}\psi^{2}u_{N}\overline{v_{M}}A_{R}dxdt\right|\\
 & =\left|\sum_{k\in\Z^{d}}\int_{\R\times\T^{d}}\psi^{2}u_{N}\cdot\overline{P_{(-R,R]^{d}-2Rk}v_{M}}A_{R}dxdt\right|\\
 & =\left|\sum_{k\in\Z^{d}}\int_{\R\times\T^{d}}\psi^{2}P_{[-8R,8R]^{d}-2Rk}u_{N}\cdot\overline{P_{(-R,R]^{d}-2Rk}v_{M}}A_{R}dxdt\right|\\
 & =\left|\sum_{k\in\Z^{d}}\int_{\R\times\T^{d}}\psi I_{2Rk}P_{[-8R,8R]^{d}-2Rk}u_{N}\cdot\overline{\psi I_{2Rk}P_{(-R,R]^{d}-2Rk}v_{M}}J_{2Rk}A_{R}dxdt\right|\\
 & =\left|\sum_{k\in\Z^{d}}\int_{\R\times\T^{d}}\psi P_{\le8R}I_{2Rk}u_{N}\cdot\overline{\psi P_{(-R,R]^{d}}I_{2Rk}v_{M}}J_{2Rk}A_{R}dxdt\right|.
\end{align*}
Using $0<\s_{1}<\al$, $\frac{1}{q}+\s_{1}+2\left(\frac{1}{p}-\al\right)=1$,
and (\ref{eq:Besov product rule}), we continue to estimate
\begin{align*}
\text{} & \les\sum_{k\in\Z^{d}}\norm{\psi P_{\le8R}I_{2Rk}u_{N}}_{B_{p,2}^{\al}L^{p}}\norm{\psi P_{\le8R}I_{2Rk}v_{M}}_{B_{p,2}^{\al}L^{p}}\norm{J_{2Rk}A_{R}}_{B_{q,\infty}^{-\s_{1}}L^{(p/2)'}}.
\end{align*}
Again, using $(\frac{1}{q},\frac{1}{r},\s_{1})\in\Omega$, $d\left(\frac{1}{r}+\frac{2}{p}-1\right)-2\s_{1}=\theta-2\be$,
$(p/2)'>q$, (\ref{eq:A_k bound}), and Bernstein estimates, we estimate
\begin{align*}
 & \les\sum_{k\in\Z^{d}}\norm{\psi P_{\le8R}I_{2Rk}u_{N}}_{B_{p,2}^{\al}L^{p}}\norm{\psi P_{\le8R}I_{2Rk}v_{M}}_{B_{p,2}^{\al}L^{p}}\\
 & \cdot R^{d\left(\frac{1}{r}+\frac{2}{p}-1\right)-2\s_{1}}\left(\left|k\right|^{-\s_{1}}+R^{-2\s_{1}}\right)\norm{A_{R}}_{B_{r,r}^{\frac{1}{r}-\frac{1}{q}}L^{r}}\\
 & =\sum_{k\in\Z^{d}}R^{-\be}\norm{\psi P_{\le8R}I_{2Rk}u_{N}}_{B_{p,2}^{\al}L^{p}}R^{-\be}\norm{\psi P_{\le8R}I_{2Rk}v_{M}}_{B_{p,2}^{\al}L^{p}}\\
 & \cdot\left(\left|k\right|^{-\s_{1}}+R^{-2\s_{1}}\right)R^{\theta}\norm{A_{R}}_{B_{r,r}^{\frac{1}{r}-\frac{1}{q}}L^{r}}\\
 & \les\norm{u_{N}}_{Z^{0}}\norm{v_{M}}_{Z^{0}}\left(\jp{N/R}^{-\s_{1}}+R^{-2\s_{1}}\right)\cdot R^{\theta}\norm{A_{R}}_{B_{r,r}^{\frac{1}{r}-\frac{1}{q}}L^{r}},
\end{align*}
where the last inequality holds since $P_{\le8R}I_{2Rk}u_{N}$ is
nonzero only if $\left|k\right|\gtrsim N/R$. This finishes the proof
of (\ref{eq:strip1 claim1-1}).
\end{proof}
\begin{lem}
\label{lem:strip1}Let $q_{0},r_{0}$, and $\theta$ be defined in Lemma
\ref{lem:strip1 claim1-1}. For $A\in B_{r_{0},r_{0}}^{\frac{1}{r_{0}}-\frac{1}{q_{0}}}B_{r_{0},\infty}^{\theta}$
and $u\in Z^{0}$, we have
\begin{equation}
\left|\int_{\R\times\T^{d}}\psi^{2}\left|u\right|^{2}Adxdt\right|\les\norm u_{Z^{0}}^{2}\norm A_{B_{r_{0},r_{0}}^{\frac{1}{r_{0}}-\frac{1}{q_{0}}}B_{r_{0},\infty}^{\theta}}.\label{eq:strip 1}
\end{equation}
\end{lem}
\begin{proof}
Again, we denote $r=r_{0}$ and $q=q_{0}$ in this proof. We decompose
the left-hand side of (\ref{eq:strip 1}) into
\begin{equation}
\left|\int_{\R\times\T^{d}}\psi^{2}\left|u\right|^{2}Adxdt\right|\le\left|\int_{\R\times\T^{d}}\psi^{2}u\cdot\pi_{>}\left(\overline{u},A\right)dxdt\right|+\left|\int_{\R\times\T^{d}}\psi^{2}u\cdot\pi_{\le}\left(\overline{u},A\right)dxdt\right|.\label{eq:triangle eq}
\end{equation}
We first estimate $\left|\int_{\R\times\T^{d}}\psi^{2}u\cdot\pi_{\le}\left(\overline{u},A\right)dxdt\right|$.
This is simply estimated using the Besov product rule (\ref{eq:Besov paraproduct rule,<=00003D})
on $\T^{d}$ and the embedding property of $Z^{s}$ (\ref{eq:time Besov embed Z^s}).
Let $\tilde{\al}$ and $\tilde{\be}$ be the exponents such that $\frac{1}{q}+2\left(\frac{1}{p}-\tilde{\al}\right)=1$
and $\tilde{\be}=2\tilde{\al}+\s$. Since $\tilde{\al}\in[\s,\frac{1}{p}-\s]$
and $2\tilde{\be}<\theta$, we have
\[
\left|\int_{\R\times\T^{d}}\psi^{2}u\cdot\pi_{\le}\left(\overline{u},A\right)dxdt\right|\les\norm{\psi u}_{B_{p,2}^{\tilde{\al}}B_{p,2}^{-\tilde{\be}}}^{2}\norm A_{B_{r,r}^{\frac{1}{r}-\frac{1}{q}}B_{r,\infty}^{\theta}}\les\norm u_{Z^{0}}^{2}\norm A_{B_{r,r}^{\frac{1}{r}-\frac{1}{q}}B_{r,\infty}^{\theta}}.
\]

The main part of the proof is to estimate $\left|\int_{\R\times\T^{d}}\psi^{2}u\cdot\pi_{>}\left(\overline{u},A\right)dxdt\right|$.
Since $\int_{\R\times\T^{d}}\psi^{2}u_{N}\cdot\pi_{>}\left(\overline{u_{M}},A_{R}\right)dxdt$
is zero whenever the dyadic frequencies $N,M\gg R$ are not comparable,
by (\ref{eq:strip1 claim1-1}), we have
\begin{align}
\left|\int_{\R\times\T^{d}}\psi^{2}u\cdot\pi_{>}\left(\overline{u},A_{R}\right)dxdt\right| & \les\sum_{N\ge16R}\norm{u_{N}}_{Z^{0}}^{2}\left(\text{\ensuremath{\jp{N/R}^{-\s_{1}}}}+R^{-2\s_{1}}\right)\label{eq:pi_hl}\\
 & \cdot R^{\theta}\norm{A_{R}}_{B_{r,r}^{\frac{1}{r}-\frac{1}{q}}L^{r}}.\nonumber 
\end{align}
Writing $A=\sum_{R\in2^{\N}}A_{R}$, we have
\begin{align*}
\left|\int_{\R\times\T^{d}}\psi^{2}u\cdot\pi_{>}\left(\overline{u},A\right)dxdt\right| & \les\sum_{\substack{R\in2^{\N}\\
N\ge16R
}
}\norm{u_{N}}_{Z^{0}}^{2}\left(\text{\ensuremath{\jp{N/R}^{-\s_{1}}}}+R^{-2\s_{1}}\right)\cdot\norm A_{B_{r,r}^{\frac{1}{r}-\frac{1}{q}}B_{r,\infty}^{\theta}}\\
 & \les\norm u_{Z^{0}}^{2}\norm A_{B_{r,r}^{\frac{1}{r}-\frac{1}{q}}B_{r,\infty}^{\theta}},
\end{align*}
which yields the estimate of $\left|\int_{\R\times\T^{d}}\psi^{2}u\cdot\pi_{>}\left(\overline{u},A\right)dxdt\right|$.
\end{proof}
As a consequence, we have the main estimate of this section.
\begin{prop}
\label{prop:Holder strip decomposition}Let $m\in\Z$. Let $\psi_{1}:\R\rightarrow\R$
be a $C_{0}^{\infty}$ bump function satisfying $\psi|_{supp(\psi_{1})}\equiv1$.
Then, for $u\in Z^{s}$ and $v\in Z^{0}$, we have
\begin{equation}
\norm{v^{*}\cdot\psi_{1}|u|^{a-m}u^{m}}_{(Z^{0})'}\les\norm v_{Z^{0}}\norm u_{Z^{s}}^{a},\label{eq:Holder strip decomposition}
\end{equation}
where $v^{*}$ denotes either $v$ or $\overline{v}$.
\end{prop}
\begin{proof}
In this proof, we use exponents $\hat{r},\zeta$, and $\eta$ defined
as $\frac{1}{\hat{r}}=\frac{1+\s_{4}}{d+2},\zeta=\frac{1}{\hat{r}}-\frac{1}{2q_{0}}$,
and $\eta=\frac{d}{\hat{r}}-\frac{d}{2r_{0}}+\frac{\theta}{2}$, respectively,
where $q_{0},r_{0}$, and $\theta$ are the exponents defined in Lemma
\ref{lem:strip1 claim1-1}.

First, we claim that for $u\in Z^{0}$ and $A\in B_{\hat{r},\hat{r}}^{\zeta}B_{\hat{r},\hat{r}}^{\eta}$,
we have
\begin{equation}
\norm{\psi Au}_{L_{t,x}^{2}}\les\norm u_{Z^{0}}\norm A_{B_{\hat{r},\hat{r}}^{\zeta}B_{\hat{r},\hat{r}}^{\eta}}.\label{eq:strip 2}
\end{equation}
Since $\zeta>\frac{1}{r_{0}}-\frac{1}{q_{0}}$, $\eta>\theta$, and
$\frac{1}{r_{0}}<\frac{2}{\hat{r}}$, by (\ref{eq:Besov product rule}),
we have 
\begin{equation}
\norm{|A|^{2}}_{B_{r_{0},r_{0}}^{\frac{1}{r_{0}}-\frac{1}{q_{0}}}B_{r_{0},\infty}^{\theta}}\les\norm A_{B_{\hat{r},\hat{r}}^{\zeta}B_{\hat{r},\hat{r}}^{\eta}}^{2}.\label{eq:|A|^2}
\end{equation}
By (\ref{eq:strip 1}) and (\ref{eq:|A|^2}), we have $\int\psi^{2}u\overline{u}\left|A\right|^{2}dxdt\les\norm u_{Z^{0}}^{2}\norm A_{B_{\hat{r},\hat{r}}^{\zeta}B_{\hat{r},\hat{r}}^{\eta}}^{2}$,
which implies (\ref{eq:strip 2}).

Next, we claim that for $m\in\Z$, we have
\begin{equation}
\norm{|\psi u|^{a/2-m}(\psi u)^{m}}_{B_{\hat{r},\hat{r}}^{\zeta}B_{\hat{r},\hat{r}}^{\eta}}\les\norm u_{Z^{s}}^{a/2}.\label{eq:Holder strip'}
\end{equation}
Once we have (\ref{eq:Holder strip'}), by (\ref{eq:strip 2}), we
obtain (\ref{eq:Holder strip decomposition}) as the following dual form:
\begin{align*}
 & \left|\int_{\R\times\T^{d}}\psi_{1}|u|^{a-m}u^{m}v^{*}wdxdt\right|\\
 & \les\norm{\psi|\psi u|^{a/2-m}(\psi u)^{m}v^{*}}_{L_{t,x}^{2}}\cdot\norm{\psi|\psi u|^{a/2}w}_{L_{t,x}^{2}}\\
 & \les\norm{|\psi u|^{a/2-m}(\psi u)^{m}}_{B_{\hat{r},\hat{r}}^{\zeta}B_{\hat{r},\hat{r}}^{\eta}}\norm{|\psi u|^{a/2}}_{B_{\hat{r},\hat{r}}^{\zeta}B_{\hat{r},\hat{r}}^{\eta}}\norm v_{Z^{0}}\norm w_{Z^{0}}\\
 & \les\norm u_{Z^{s}}^{a}\norm v_{Z^{0}}\norm w_{Z^{0}}.
\end{align*}
Since (\ref{eq:frac Holder tx}) holds only for functions of H{\"o}lder
regularities $C^{0,\al},\al\in(0,1)$, we first decompose the exponent
$a/2$ into $a_{1},\ldots,a_{k}\in(0,1)$ satisfying 
\begin{equation}
a_{1}+\ldots+a_{k}=a/2\label{eq:a1+...+ak}
\end{equation}
and $a_{j}\gg\s_{4}$ for $j=1,\ldots,k$. We partition $|\psi u|^{a/2-m}(\psi u)^{m}$
into a product of $k$ terms of the form
\begin{equation}
|\psi u|^{a/2-m}(\psi u)^{m}=\prod_{j=1}^{k}|\psi u|^{a_{j}-m_{j}}(\psi u)^{m_{j}},\label{eq:product form}
\end{equation}
where $m_{j}\in\Z$'s are certain integers, then estimate each term
using (\ref{eq:frac Holder tx}).

Choose $s_{0}>0$ satisfying $\s_{4}\ll_{d,a_{j},m_{j}}s_{0}\ll_{d,a_{j},m_{j}}1$.
Let $r$ and $s_{1}$ be the exponents satisfying $\frac{a}{2}\left(\frac{1}{r}-s_{0}\right)=\frac{1}{\hat{r}}-\zeta$
and $\frac{a}{2}\left(\frac{1}{r}-\frac{s_{1}}{d}\right)=\frac{1}{\hat{r}}-\frac{\eta}{d}$.

Since $s-\s>2s_{0}+s_{1}$ and $r\in[p,\frac{1}{\s}]$, by (\ref{eq:Sobolev embed Z^s})
and (\ref{eq:Besov embed}), we have
\begin{equation}
\norm{\psi u}_{L^{r}B_{r,2}^{2s_{0}+s_{1}}}\les\norm u_{Z^{s}}.\label{eq:Holder strip claim1}
\end{equation}
Choose $\al=\frac{1}{p}-\frac{1}{r}+s_{0}+s_{1}/2$ and $\be=2\al+\s$.
Since $\al\in[\s,\frac{1}{p}-\s]$, $\al>s_{0}+s_{1}/2$, and $s-\be>0$,
by (\ref{eq:time Besov embed Z^s}), (\ref{eq:Besov-Lorentz embed}),
and (\ref{eq:Besov embed}), we have
\begin{equation}
\norm{\psi u}_{B_{r,2}^{s_{0}+s_{1}/2}L^{r}}\les\norm{\psi u}_{B_{p,2}^{\al}B_{p,2}^{s-\be}}\les\norm u_{Z^{s}}.\label{eq:Holder strip claim2}
\end{equation}
Applying (\ref{eq:Besov product rule}) and (\ref{eq:frac Holder tx})
to (\ref{eq:product form}), since $r\ge2$, $s_{0}\cdot a_{j}>\zeta>0$,
and $s_{1}\cdot a_{j}>\eta>0$, by (\ref{eq:Holder strip claim1})
and (\ref{eq:Holder strip claim2}), we have
\begin{align*}
\norm{|\psi u|^{a/2-m}(\psi u)^{m}}_{B_{\hat{r},\hat{r}}^{\zeta}B_{\hat{r},\hat{r}}^{\eta}} & \les\prod_{j=1}^{k}\norm{|\psi u|^{a_{j}-m_{j}}(\psi u)^{m_{j}}}_{B_{r/a_{j},r/a_{j}}^{s_{0}\cdot a_{j}}B_{r/a_{j},r/a_{j}}^{s_{1}\cdot a_{j}}}\\
 & \les\norm{\psi u}_{L^{r}B_{r,2}^{2s_{0}+s_{1}}\cap B_{r,2}^{s_{0}+s_{1}/2}L^{r}}^{a/2}\\
 & \les\norm u_{Z^{s}}^{a/2},
\end{align*}
which is just (\ref{eq:Holder strip'}) and finishes the proof.
\end{proof}

\section{\label{sec:Proof of LWP}Local well-posedness of (\ref{eq:NLS})}

In this section, we prove Theorem \ref{thm:LWP s<a}, the local well-posedness
of (\ref{eq:NLS}). We construct solutions in the $Z^{s}$ space introduced
in Section \ref{sec:Z^s-spaces}. Based on the key estimate, Proposition
\ref{prop:Holder strip decomposition}, we first prove the main nonlinear
estimate (Lemma \ref{lem:Nu bound =0003C6}). On the way, we handle
non-algebraic nonlinear terms using the Bony linearization. Since
we construct solutions by weak limits, we provide a separate argument
for the continuous dependence. For this purpose, we use an enhanced
form of nonlinear estimate (see (\ref{eq:Nu bound =0003C6})).

Fix a $C_{0}^{\infty}$-function $\psi_{1}$ such that $\psi|_{\text{supp}(\psi_{1})}\equiv1$
and $\psi_{1}\equiv1$ on some open interval containing $0$.

\subsection{Nonlinear estimates}

In this subsection, we propose some estimates on $\psi_{1}\calN(u)$,
where $u:\R\times\T^{d}\rightarrow\C$. To handle the non-algebraicity
of $\calN(u)$, we use a paraproduct technique known as the Bony linearization
method. We decompose $\calN(u)$ as a sum of
\[
F^{N}:=\calN(u_{\le N})-\calN(u_{\le N/2})
\]
over $N\in2^{\N}$, then estimate $F_{K}^{N}=P_{K}F^{N}$ in terms
of $N,K\in2^{\N}$. (We use the conventional notation $u_{\le1/2}=0$.)
Here, the following Bony linearization formula given in \cite{lee2019local}
is used:
\begin{equation}
F^{N}=u_{N}\int_{0}^{1}\d_{z}\calN(u_{\le N/2}+\theta u_{N})d\theta+\overline{u}_{N}\int_{0}^{1}\d_{\overline{z}}\calN(u_{\le N/2}+\theta u_{N})d\theta.\label{eq:Bony}
\end{equation}
For simplicity of notation, we denote by $A^{N}:\R\times\T^{d}\rightarrow\C^{2},N\in2^{\N}$ the function
\[
A^{N}:=\left(\int_{0}^{1}\d_{z}\calN(u_{\le N/2}+\theta u_{N})d\theta,\int_{0}^{1}\d_{\overline{z}}\calN(u_{\le N/2}+\theta u_{N})d\theta\right)
\]
and denote by $u\times A$ for $u:\R\times\T^{d}\rightarrow\C$ and $A=(A_{1},A_{2}):\R\times\T^{d}\rightarrow\C^{2}$ the function
\[
u\times A:=uA_{1}+\overline{u}A_{2}.
\]
With this notation, (\ref{eq:Bony}) can be rewritten as $F^{N}=u_{N}\times A^{N}$.

The bilinear estimate (\ref{eq:Holder strip decomposition}) is transferred
to the following estimate of $F_{K}^{N}$:
\begin{lem}
For $N,K\in2^{\N}$, we have
\begin{equation}
\norm{\psi_{1}F_{K}^{N}}_{(Z^{-s})'}\les\norm u_{Z^{s}}^{a}(K/N)^{s}\cdot\norm{u_{N}}_{Z^{s}}.\label{eq:N,K}
\end{equation}
\end{lem}
\begin{proof}
It suffices to show $\norm{\psi_{1}F^{N}}_{(Z^{0})'}\les\norm{u_{N}}_{Z^{0}}\norm u_{Z^{s}}^{a}$
by frequency localization. By (\ref{eq:Holder strip decomposition}),
we have
\begin{align*}
\norm{\psi_{1}F^{N}}_{(Z^{0})'} & =\norm{u_{N}\times\psi_{1}A^{N}}_{(Z^{0})'}\\
 & \les\int_{0}^{1}\norm{u_{N}\times\psi_{1}\left(\d_{z}\calN(u_{\le N/2}+\theta u_{N}),\d_{\overline{z}}\calN(u_{\le N/2}+\theta u_{N})\right)}_{(Z^{0})'}d\theta\\
 & \les\int_{0}^{1}\norm{u_{N}}_{Z^{0}}\norm{u_{\le N/2}+\theta u_{N}}_{Z^{s}}^{a}d\theta\\
 & \les\norm{u_{N}}_{Z^{0}}\norm u_{Z^{s}}^{a},
\end{align*}
which finishes the proof.
\end{proof}
On the other hand, the fractional H{\"o}lder inequalities give the following
estimate:
\begin{lem}
\label{lem:N<<K}Let $\mu>0$ and $\nu>0$ be the exponents 
$\nu=1+a-s-\s_{1}$ and $\mu=\left(1-\frac{2\s}{\s_{1}}\right)(\nu+\s)+\frac{2\s}{\s_{1}}(-s+\s)$.
For $N,K\in2^{\N}$ satisfying $4N\le K$, we have
\begin{equation}
\|\psi_{1}F_{K}^{N}\|_{(Z^{-s})'}\les(N/K)^{\mu}N^{-\nu}\|u_{\le N}\|_{Z^{s+\nu}}\|u\|_{Z^{s}}^{a}.\label{eq:N<<K}
\end{equation}
\end{lem}
\begin{proof}
We will estimate $P_{K}\left(u_{N}\times A^{N}\right)=F_{K}^{N}$
in two ways: (i) applying (\ref{eq:low fractional Holder}) to $A_{M}^{N},M\in2^{\N}$
to estimate $F_{K}^{N}$; (ii) applying (\ref{eq:high fractional Holder})
to $F_{K}^{N}$ directly. Then we will interpolate between (i) and
(ii) to obtain (\ref{eq:N<<K}).

Let $q$ and $r$ be the parameters satisfying $\frac{1+a}{q}=\frac{1}{p'}$
and $\frac{d}{r}=\frac{d}{2}-\s-\frac{2}{q}$. Let $p_{0}$ and $p_{1}$
be the parameters satisfying $\frac{1}{p_{1}'}-\frac{1}{r}-\frac{\s_{1}}{d}=a\left(\frac{1}{r}-\frac{s-\s}{d}\right)$
and $\frac{1}{p_{0}'}-\frac{1}{r}=a\left(\frac{1}{r}-\frac{s-\s}{d}\right)$.

(i) Since $\s_{1}<a(s-\s)$ and $\s_{1}<s-\s$, by either (\ref{eq:low fractional Holder})
or (\ref{eq:high fractional Holder}), for $M\in2^{\N}$, we have
\begin{align*}
\|u_{N}\times A_{M}^{N}\|_{L^{p_{1}'}} & \les\|u_{N}\|_{L^{r}}\|A_{M}^{N}\|_{L^{\left(\frac{1}{p_{1}'}-\frac{1}{r}\right)^{-1}}}\\
 & \les N^{-s+\s}\|u_{N}\|_{H^{s-\s,r}}\cdot M^{-\s_{1}}\|A^{N}\|_{H^{\s_{1},\left(\frac{1}{p_{1}'}-\frac{1}{r}\right)^{-1}}}\\
 & \les N^{-s+\s}M^{-\s_{1}}\|u_{N}\|_{H^{s-\s,r}}\|u\|_{H^{s-\s,r}}^{a}.
\end{align*}
Thus, we have
\begin{align}
\|F_{K}^{N}\|_{L^{p_{1}'}} & \les\sum_{M\in2^{\N}}\|P_{K}(u_{N}\times A_{M}^{N})\|_{L^{p_{1}'}}\label{eq:H}\\
 & \les\sum_{M\sim K}\|u_{N}\times A_{M}^{N}\|_{L^{p_{1}'}}\nonumber \\
 & \les(N/K)^{-s+\s}K^{-s-\s_{1}+\s}\|u_{N}\|_{H^{s-\s,r}}\|u\|_{H^{s-\s,r}}^{a}\nonumber \\
 & \les(N/K)^{-s+\s}K^{-s-\s_{1}+\s}N^{-\nu}\|u_{\le N}\|_{H^{s-\s+\nu,r}}\|u\|_{H^{s-\s,r}}^{a}.\nonumber 
\end{align}

(ii) Since $s+\nu<1+a$, by (\ref{eq:high fractional Holder}), we
have
\begin{align}
\|F_{K}^{N}\|_{L^{p_{0}'}} & \les K^{-s-\nu}\|F_{K}^{N}\|_{H^{s+\nu,p_{0}'}}\label{eq:L}\\
 & \les K^{-s-\nu}\|u_{\le N}\|_{H^{s+\nu,r}}\|u\|_{L^{a\left(\frac{1}{p_{0}'}-\frac{1}{r}\right)^{-1}}}^{a}\nonumber \\
 & \les K^{-s-\nu}N^{\s}\|u_{\le N}\|_{H^{s-\s+\nu,r}}\|u\|_{H^{s-\s,r}}^{a}\nonumber \\
 & =(N/K)^{\nu+\s}K^{-s+\s}N^{-\nu}\|u_{\le N}\|_{H^{s-\s+\nu,r}}\|u\|_{H^{s-\s,r}}^{a}.\nonumber 
\end{align}
We interpolate between (\ref{eq:H}) and (\ref{eq:L}). Interpolating
between (\ref{eq:H}) and (\ref{eq:L}) with exponent $2\s/\s_{1}$,
we have
\[
\|F_{K}^{N}\|_{L^{p'}}\les(N/K)^{\mu}K^{-s-\s}N^{-\nu}\|u_{\le N}\|_{H^{s-\s+\nu,r}}\|u\|_{H^{s-\s,r}}^{a}.
\]
Multiplying cutoffs and applying $L_{t}^{p'}$ norms to both sides,
by $\frac{1+a}{q}=\frac{1}{p'}$, $\norm{\psi u_{\le N}}_{L^{q}H^{s-\s+\nu,r}}\les\norm{u_{\le N}}_{Z^{s+\nu}}$,
and $\norm{\psi u_{\le N}}_{L^{q}H^{s-\s,r}}\les\norm{u_{\le N}}_{Z^{s}}$,
we have
\begin{equation}
\|\psi_{1}F_{K}^{N}\|_{L_{t,x}^{p'}}\les(N/K)^{\mu}K^{-s-\s}N^{-\nu}\|u_{\le N}\|_{Z^{s+\nu}}\|u\|_{Z^{s}}^{a},\label{eq:HL}
\end{equation}
which implies (\ref{eq:N<<K}) due to (\ref{eq:Sobolev embed Z^s}).
\end{proof}
In the rest of this section, for $K,M,N\in2^{\N}$, we denote
\begin{align*}
\be_{K}^{N} & :=\norm{\psi_{1}F_{K}^{N}}_{(Z^{-s})'}^{2},\\
\al_{M} & :=\norm{u_{M}}_{Z^{s}}^{2},\\
\al_{\Omega} & :=\norm u_{Z^{s}}^{2}=\sum_{M\in2^{\N}}\al_{M}.
\end{align*}
In terms of $\al_{N}$ and $\be_{K}^{N}$, (\ref{eq:N,K}) and (\ref{eq:N<<K})
can be rewritten as
\[
\be_{K}^{N}\les\al_{\Omega}^{a}(K/N)^{2s}\cdot\al_{N}
\]
and
\[
\be_{K}^{N}\les\al_{\Omega}^{a}(N/K)^{2\mu}\cdot\sum_{L\le N}\left(L/N\right)^{2\nu}\al_{L},\qquad4N\le K,
\]
respectively.

Combining (\ref{eq:N,K}) and (\ref{eq:N<<K}), we obtain the following
inequality:
\begin{cor}
\label{cor:beta bound}There exists $\epsilon_{0}=\epsilon_{0}(d,s,a)>0$
(e.g., $\ep_{0}=\min\left\{ s,\mu,2\nu\right\} /2$) such that, for
each $K\in2^{\N}$,
\begin{equation}
\left(\sum_{N\in2^{\N}}\sqrt{\be_{K}^{N}}\right)^{2}\les\al_{\Omega}^{a}\cdot\sum_{N\in2^{\N}}\max\left\{ \frac{N}{K},\frac{K}{N}\right\} ^{-\epsilon_{0}}\al_{N}.\label{eq:beta bound}
\end{equation}
\end{cor}
\begin{proof}
Let $\mu$ and $\nu$ be the exponents defined in Lemma \ref{lem:N<<K}.
Since $\sum_{N>K/4}(N/K)^{-s}+\sum_{N\le K/4}(K/N)^{-\mu}\les1$,
by the Cauchy-Schwartz inequality, we have
\begin{align*}
\left(\sum_{N\in2^{\N}}\sqrt{\be_{K}^{N}}\right)^{2} & \les\sum_{N>K/4}(N/K)^{s}\be_{K}^{N}+\sum_{N\le K/4}(K/N)^{\mu}\be_{K}^{N}\\
 & \les\al_{\Omega}^{a}\cdot\left(\sum_{N>K/4}(K/N)^{s}\al_{N}+\sum_{N\le K/4}(N/K)^{\mu}\sum_{L\le N}\left(L/N\right)^{2\nu}\al_{L}\right)\\
 & \les\al_{\Omega}^{a}\cdot\left(\sum_{N>K/4}(K/N)^{s}\al_{N}+\sum_{L\le K/4}(L/K)^{\min\left\{ \mu,2\nu\right\} -}\al_{L}\right)\\
 & \les\al_{\Omega}^{a}\cdot\sum_{N\in2^{\N}}\max\left\{ \frac{N}{K},\frac{K}{N}\right\} ^{-\epsilon_{0}}\al_{N},
\end{align*}
which is (\ref{eq:beta bound}).

Here, it suffices to choose $\epsilon_{0}=\min\left\{ s,\mu,2\nu\right\} /2$.
\end{proof}
Now, we present the main inequality for the nonlinear term $\calN(u)$.
\begin{lem}
\label{lem:Nu bound =0003C6}Let $\epsilon_{0}$ be as in Corollary
\ref{cor:beta bound}. Fix $\epsilon_{1}\ll\epsilon_{0}$. Let $\left\{ \gamma_{N}\right\} _{N\in2^{\N}}$
be a sequence of positive numbers such that $1\le\frac{\ga_{2N}}{\ga_{N}}\le1+\epsilon_{1}$.
We have
\begin{equation}
\sum_{N\in2^{\N}}\ga_{N}\norm{\psi_{1}P_{N}\calN(u)}_{(Z^{-s})'}^{2}\les\norm u_{Z^{s}}^{2a}\sum_{N\in2^{\N}}\ga_{N}\norm{u_{N}}_{Z^{s}}^{2}.\label{eq:Nu bound =0003C6}
\end{equation}
In particular, plugging in $\ga_{N}\equiv1$ gives
\begin{equation}
\norm{\psi_{1}\calN(u)}_{(Z^{-s})'}\les\norm u_{Z^{s}}^{1+a}.\label{eq:Nu bound}
\end{equation}
\end{lem}
A standard nonlinear estimate for constructing a solution is (\ref{eq:Nu bound}).
We need an enhanced form, (\ref{eq:Nu bound =0003C6}), for showing
the continuous dependence of the solution map.
\begin{proof}
The proof follows easily from (\ref{eq:beta bound}). We have
\begin{equation}
\sum_{K\in2^{\N}}\ga_{K}\left(\sum_{N\in2^{\N}}\sqrt{\be_{K}^{N}}\right)^{2}\les\al_{\Omega}^{a}\cdot\sum_{K,N\in2^{\N}}\max\left\{ \frac{N}{K},\frac{K}{N}\right\} ^{-\epsilon_{0}}\ga_{K}\al_{N}\les\al_{\Omega}^{a}\cdot\sum_{N\in2^{\N}}\ga_{N}\al_{N}.\label{eq:Nu bound =0003C61}
\end{equation}
Since $\norm{\psi_{1}P_{K}\calN u}_{(Z^{-s})'}\les\sum_{N\in2^{\N}}\sqrt{\be_{K}^{N}}$,
(\ref{eq:Nu bound =0003C61}) implies (\ref{eq:Nu bound =0003C6}).
\end{proof}
\begin{rem}
Due to the embedding between $Y^{s}$ and $Z^{s}$, (\ref{eq:Z^s-Y^s}),
(\ref{eq:Nu bound}) also implies the nonlinear estimate in $Y^{s}$
space:
\begin{equation}
\norm{\psi_{1}\calN(u)}_{(Y^{-s})'}\les\norm u_{Y^{s}}^{1+a}.\label{eq:beta bound Y^s}
\end{equation}
Although $Y^{s}\hook Z^{s}$, we expect the proof of (\ref{eq:Nu bound})
to be almost at the same level as that of (\ref{eq:beta bound Y^s}).
This is because, in designing $Z^{s}$, we incorporated structures
of $Y^{s}$ required for the nonlinear estimate, such as the Galilean
invariance, the frequency $\ell^{2}$-basedness, and the Besov Strichartz
estimates. We recall that the main benefit of working in the $Z^{s}$
space is that $\norm{u\cdot\chi_{[0,T]}}_{Z^{s}}$ shrinks as $T\rightarrow0$, 
(\ref{eq:shrink Z^s}).
\end{rem}

\subsection{Proof of Theorem \ref{thm:LWP s<a}}

In this subsection, we finish the proof of Theorem \ref{thm:LWP s<a}.
We restate it in a more precise technical form in the following proposition:
\begin{prop}
\label{prop:Y^s LWP}Let $a>\frac{4}{d}$ and $s<1+a$. Fix $\underline{u_{0}}\in H^{s}$.
There exist an open interval $I=I(\underline{u_{0}})\ni0$ and $R=R(\underline{u_{0}})>0$
such that for every $u_{0}\in H^{s}$ with $\norm{u_{0}-\underline{u_{0}}}_{H^{s}}<R$,
there exists a unique solution $\Phi(u_{0}):=u\in C^{0}(I;H^{s})\cap Y^{s}$
to 
\begin{equation}
\begin{cases}
iu_{t}+\De u=\calN(u)\cdot\chi_{I}\\
u(0)=u_{0}
\end{cases}\label{eq:NLS*I}
\end{equation}
and the solution map $\Phi:\left\{ u_{0}\in H^{s}\mid\norm{u_{0}-\underline{u_{0}}}_{H^{s}}<R\right\} \rightarrow C^{0}(I;H^{s})\cap Y^{s}$
is continuous.
\end{prop}

In Proposition \ref{prop:Y^s LWP}, $I$ depends on $\underline{u_{0}}$
to guarantee the smallness of $\norm{e^{it\De}\underline{u_{0}}\cdot\chi_{I}}_{Z^{s}}$,
for which our $Z^{s}$ space plays a crucial role.
\begin{proof}
We use nonlinear estimates based on $Z^{s}$. Using the embedding
$Y^{s}\hook Z^{s}$ we can show that the constructed solution lies
in $Y^{s}$ as well.

\emph{Construction}

First, we show that such a map $\Phi$ exists. We construct such a
solution $u=\Phi(u_{0})$ by taking a weak limit of approximate solutions
bounded in $Y^{s}$.

By (\ref{eq:shrink Z^s}), there exists an interval $I\ni0$ such
that $\norm{e^{it\De}\underline{u_{0}}\cdot\chi_{I}}_{Z^{s}}$ is
sufficiently small and $\psi_{1}|_{I}\equiv1$. Choose $R\ll_{\underline{u_{0}},I}1$.
For $\la\ge0$, we denote
\[
P_{<2^{\la}}:=\begin{cases}
P_{\le2^{\left\lfloor \la\right\rfloor -1}}+(\la-\left\lfloor \la\right\rfloor )P_{2^{\left\lfloor \la\right\rfloor }} & ,\la\ge1\\
\la P_{\le1} & ,0\le\la<1
\end{cases},
\]
where $\left\lfloor \la\right\rfloor $ denotes the greatest integer
$n\le\la$.

Fix $\la\ge0$. For $u_{0}\in H^{s}$, let $u=u^{(\la)}$ be the strong
solution of the following equation:
\begin{equation}
\begin{cases}
iu_{t}+\De u=P_{<2^{\la}}\left(\calN(u)\cdot\chi_{I}\right)\\
u(0)=u_{0}
\end{cases}.\label{eq:NLS'}
\end{equation}
Let $v=u-e^{it\De}u_{0}$. By (\ref{eq:Nu bound}), (\ref{eq:step Z^s}),
and (\ref{eq:Z^s-Y^s}), we have a bootstrap bound on $\norm v_{Z^{s}}$:
\begin{align}
\norm v_{Z^{s}} & \les\norm v{}_{Y^{s}}\les\norm{\psi_{1}\calN(u\cdot\chi_{I})}_{(Y^{-s})'}\les\norm{\psi_{1}\calN(u\cdot\chi_{I})}_{(Z^{-s})'}\label{eq:bootstrap}\\
 & \les\norm{u\cdot\chi_{I}}_{Z^{s}}^{1+a}\les\left(\norm v_{Z^{s}}+\norm{e^{it\De}u_{0}\cdot\chi_{I}}_{Z^{s}}\right)^{1+a}.\nonumber 
\end{align}
Assume $\norm{u_{0}-\underline{u_{0}}}_{H^{s}}<R$. Since $\norm{v^{(\la)}}_{Z^{s}}$
is continuous on $\la$ with $\norm{v^{(0)}}_{Z^{s}}=0$ and 
\[
\norm{e^{it\De}u_{0}\cdot\chi_{I}}_{Z^{s}}\les\norm{e^{it\De}\underline{u_{0}}\cdot\chi_{I}}_{Z^{s}}+\norm{u_{0}-\underline{u_{0}}}_{H^{s}}\ll1,
\]
we have $\sup_{\la\ge0}\norm{v^{(\la)}}_{Z^{s}}\ll1$ and so $\sup_{\la\ge0}\norm{v^{(\la)}}_{Y^{s}}\ll1.$
By (\ref{eq:Sobolev embed Z^s}), (\ref{eq:time Besov embed Z^s}),
and (\ref{eq:bootstrap}), we have
\[
\sup_{\la\ge0}\norm{\psi u^{(\la)}}_{B_{p,2}^{\s}B_{p,2}^{s-3\s}\cap L_{t,x}^{(2d+4)/(d-2s)}}\les\sup_{\la\ge0}\norm{u^{(\la)}}_{Y^{s}}<\infty.
\]
Since $\s>0$, $s-3\s>0$, and $(2d+4)/(d-2s)>1+a$, the embedding
$B_{p,2}^{\s}B_{p,2}^{s-3\s}\cap L_{t,x}^{(2d+4)/(d-2s)}\hook L_{t,x}^{1+a}$
is compact on the space of functions supported on $I\times\T^{d}$.
Thus, we have a sequence $\left\{ \la_{n}\right\} $ increasing to
$\infty$ such that $u^{(\la_{n})}$ converges to a function $u^{(\infty)}$
in $L_{t,x}^{1+a}(I\times\T^{d})$. Let $u\in Y^{s}$ be the Duhamel
solution of $iu_{t}+\De u=\calN(u^{(\infty)})\cdot\chi_{I}\in(Z^{-s})',\,u(0)=u_{0}$.
Since $u^{(\infty)}$ is a weak solution to (\ref{eq:NLS*I}), we
have $P_{\le N}u=P_{\le N}u^{(\infty)}$ almost everywhere for every
$N\in2^{\N}$. Hence $u=u^{(\infty)}\in L_{t,x}^{1+a}(I\times\T^{d})$
holds almost everywhere and so $u\in Y^{s}$ is a solution to (\ref{eq:NLS*I}),
$u(0)=u_{0}$. Furthermore, for each $t_{0}\in I$, by (\ref{eq:shrink Z^s}),
we have the continuity in time:
\begin{align*}
\limsup_{t_{1}\rightarrow t_{0}}\norm{u(t_{1})-u(t_{0})}_{H^{s}} & \les\limsup_{t_{1}\rightarrow t_{0}}\norm{\calN(u(t+t_{0}))\cdot\chi_{[0,t_{1}-t_{0}]}}_{(Z^{-s})'}\\
 & \les\limsup_{t_{1}\rightarrow t_{0}}\norm{u(t+t_{0})\cdot\chi_{[0,t_{1}-t_{0}]}}_{Z^{s}}^{1+a}=0.
\end{align*}
Therefore, we obtained a strong solution $u\in C^{0}H^{s}\cap Y^{s}$
to (\ref{eq:NLS*I}). We emphasize the estimate 
\begin{equation}
\norm{u\cdot\chi_{I}}_{Z^{s}}\les\norm{e^{it\De}\underline{u_{0}}\cdot\chi_{I}}_{Z^{s}}+\norm v_{Z^{s}}\ll1.\label{eq:smallness}
\end{equation}
(\ref{eq:smallness}) will be used to show the continuity of the solution
map $\Phi$.

\emph{Uniqueness}

Next, we check that such a solution $u\in C^{0}H^{s}\cap Y^{s}$ is
unique. Let $u,v\in C^{0}H^{s}\cap Y^{s}$ be two solutions to (\ref{eq:NLS})
with $u(0)=v(0)\in H^{s}$. We show that $u=v$ holds on a sufficiently
small interval $I_{1}\subset I$. Using (\ref{eq:shrink Z^s}), we freely
shrink an open interval $I_{1}\subset I$ containing $0$ such that
$\norm{u\cdot\chi_{I_{1}}}_{Z^{s}},\norm{v\cdot\chi_{I_{1}}}_{Z^{s}}\ll1$.

Let $w=\left(v-u\right)\cdot\chi_{I_{1}}$. By (\ref{eq:Holder strip decomposition}),
we have
\begin{align*}
\norm w_{Z^{0}} & \les\norm{\left(\calN(v)-\calN(u)\right)\chi_{I_{1}}}_{(Z^{0})'}\\
 & =\|\psi_{1}\cdot\left(w\int_{0}^{1}\d_{z}\calN(u\cdot\chi_{I_{1}}+\theta w\cdot\chi_{I_{1}})d\theta\right.\\
 & +\left.\overline{w}\int_{0}^{1}\d_{\overline{z}}\calN(u\cdot\chi_{I_{1}}+\theta w\cdot\chi_{I_{1}})d\theta\right)\|_{(Z^{0})'}\\
 & \les\norm w_{Z^{0}}\int_{0}^{1}\norm{u\cdot\chi_{I_{1}}+\theta w\cdot\chi_{I_{1}}}_{Z^{s}}^{a}d\theta\\
 & \les\norm w_{Z^{0}}\left(\norm{u\cdot\chi_{I_{1}}}_{Z^{s}}^{a}+\norm{v\cdot\chi_{I_{1}}}_{Z^{s}}^{a}\right).
\end{align*}
Thus, we have $\norm w_{Z^{0}}=0$, which implies $u=v$ on $I_{1}$.

\emph{Continuous dependence}

So far we have constructed the solution map $\Phi:\left\{ u_{0}\in H^{s}\mid\norm{u_{0}-\underline{u_{0}}}_{H^{s}}<R\right\} \rightarrow C^{0}H^{s}\cap Y^{s}$,
whose image is in a small neighborhood of $0$ in $Z^{s}$ when localized
on $I$. Now, we prove the continuity of the solution map $\Phi$.
For simplicity of notation, we show the continuity only at $\underline{u_{0}}$.
The only information we use about $\underline{u_{0}}$ is (\ref{eq:smallness}),
thus the continuity of $\Phi$ on a small neighborhood of $\underline{u_{0}}$
can be shown by the same argument. For this purpose, we show that
\[
\limsup_{N_{0}\rightarrow\infty}\limsup_{\delta\rightarrow0}\sup_{\norm{u_{0}-\underline{u_{0}}}_{H^{s}}<\delta}\norm{P_{\ge N_{0}}\Phi(u_{0})}_{Y^{s}}=0
\]
and
\[
P_{\le N}\Phi\text{ is Lipschitz for any }N\in2^{\N}.
\]
Fix $\epsilon>0$. There exist $N_{0}\in2^{\N}$ and a positive sequence
$\left\{ \gamma_{N}\right\} $ satisfying the condition of (\ref{eq:Nu bound =0003C6}),
$\sum_{N\in2^{\N}}\gamma_{N}\norm{\underline{u_{0}}_{N}}_{H^{s}}^{2}\le\epsilon^{2}$,
and $\gamma_{N}=1$ for every $N\ge N_{0}$. (Such a choice is possible since
one can choose $\gamma_{N}$ close to $0$ for arbitrarily many $N$'s.)
Fix $\delta\ll\epsilon N_{0}^{-s}$. For $\norm{u_{0}-\underline{u_{0}}}_{H^{s}}<\delta$,
by (\ref{eq:Nu bound =0003C6}), $u=\Phi(u_{0})$ satisfies 
\begin{align}
\sum_{N\in2^{\N}}\gamma_{N}\norm{u_{N}}_{Y^{s}}^{2} & \les\sum_{N\in2^{\N}}\gamma_{N}\norm{u_{0N}}_{H^{s}}^{2}+\sum_{N\in2^{\N}}\gamma_{N}\norm{P_{N}\calN(u)\cdot\chi_{I}}_{(Z^{-s})'}^{2}\label{eq:tight0}\\
 & \les\norm{u_{0}-\underline{u_{0}}}_{H^{s}}^{2}+\sum_{N\in2^{\N}}\gamma_{N}\norm{\underline{u_{0}}_{N}}_{H^{s}}^{2}\nonumber \\
 & +\norm{u\cdot\chi_{I}}_{Z^{s}}^{2a}\cdot\sum_{N\in2^{\N}}\gamma_{N}\norm{u_{N}\cdot\chi_{I}}_{Z^{s}}^{2}.\nonumber 
\end{align}
Since $\norm{u_{N}\cdot\chi_{I}}_{Z^{s}}^{2}\les\norm{u_{N}}_{Y^{s}}^{2}$,
by (\ref{eq:tight0}) and (\ref{eq:smallness}), we have
\begin{equation}
\norm{u_{\ge N_{0}}}_{Y^{s}}^{2}\le\sum_{N\in2^{\N}}\gamma_{N}\norm{u_{N}}_{Y^{s}}^{2}\les\delta^{2}+\epsilon^{2}.\label{eq:tight}
\end{equation}
Let $\underline{u}=\Phi(\underline{u_{0}})$ and $w=u-\underline{u}$.
By (\ref{eq:Holder strip decomposition}), we have
\begin{align*}
 & \norm{w-e^{it\De}(u_{0}-\underline{u_{0}})}_{Y^{0}}\\
 & \les\norm{\left(\calN(u)-\calN(\underline{u})\right)\chi_{I}}_{(Z^{0})'}\\
 & =\norm{\psi_{1}\cdot\left(w\int_{0}^{1}\d_{z}\calN(\underline{u}\cdot\chi_{I}+\theta w\cdot\chi_{I})d\theta+\overline{w}\int_{0}^{1}\d_{\overline{z}}\calN(\underline{u}\cdot\chi_{I}+\theta w\cdot\chi_{I})d\theta\right)}_{(Z^{0})'}\\
 & \les\norm w_{Z^{0}}\int_{0}^{1}\norm{\underline{u}\cdot\chi_{I}+\theta w\cdot\chi_{I}}_{Z^{s}}^{a}d\theta\\
 & \les\norm w_{Y^{0}}\left(\norm{u\cdot\chi_{I}}_{Z^{s}}^{a}+\norm{\underline{u}\cdot\chi_{I}}_{Z^{s}}^{a}\right).
\end{align*}
Since $\norm{u\cdot\chi_{I}}_{Z^{s}}^{a},\norm{\underline{u}\cdot\chi_{I}}_{Z^{s}}^{a}\ll1$,
we have
\[
\norm w_{Y^{0}}\les\norm{u_{0}-\underline{u_{0}}}_{L^{2}},
\]
which implies
\begin{equation}
\norm{P_{\le N_{0}}\left(u-\underline{u}\right)}_{Y^{s}}^{2}\les N_{0}^{2s}\norm{u_{0}-\underline{u_{0}}}_{L^{2}}^{2}\les N_{0}^{2s}\delta^{2}.\label{eq:low lipschitz}
\end{equation}
Combining (\ref{eq:tight}) and (\ref{eq:low lipschitz}), we have
\[
\norm{\Phi(u_{0})-\Phi(\underline{u_{0}})}_{Y^{s}}=\norm{u-\underline{u}}_{Y^{s}}\les\epsilon,
\]
finishing the proof of Proposition \ref{prop:Y^s LWP}.
\end{proof}

\section{\label{sec:Lipschitz}Proof of Theorem \ref{thm:Lipschitz}}

In this section, we prove Theorem \ref{thm:Lipschitz}. Throughout
this section, we fix exponents $d,a$, and $s$ such that 
\begin{equation}
0<s<a\text{ and }1<a.\label{eq:0<s<a}
\end{equation}
Note that $\frac{4}{d}<a$ is equivalent to $0<s$.

The proof of Theorem \ref{thm:Lipschitz} follows from a nonlinear
estimate of the difference between solutions:
\begin{lem}
\label{lem:Lipschitz lemma}For $u,v\in Z^{s}$, we have
\begin{equation}
\norm{\psi_{1}\d_{z}\calN(u)\cdot v+\psi_{1}\d_{\overline{z}}\calN(u)\cdot\overline{v}}_{(Z^{-s})'}\les\norm u_{Z^{s}}^{a}\norm v_{Z^{s}}.\label{eq:Lipschitz bound}
\end{equation}
\end{lem}

We first show that Lemma \ref{lem:Lipschitz lemma} implies Theorem
\ref{thm:Lipschitz}. Fix $\underline{u_{0}}\in H^{s}$. Let $I$
and $R>0$ be defined in Proposition \ref{prop:Y^s LWP}, chosen to
be small enough so that (\ref{eq:smallness}) also holds. For $u_{0},v_{0}\in H^{s}$
satisfying $\norm{u_{0}-\underline{u_{0}}}_{H^{s}},\norm{v_{0}-\underline{u_{0}}}_{H^{s}}<R$,
let $w=v-u$ and $w_{0}=v_{0}-u_{0}$. We have
\begin{align*}
\norm w_{Y^{s}} & \les\norm{w_{0}}_{H^{s}}+\int_{0}^{1}\norm{\left(\d_{z}\calN(u+\theta w)\cdot w+\d_{\overline{z}}\calN(u+\theta w)\cdot\overline{w}\right)\cdot\chi_{I}}_{(Z^{-s})'}d\theta\\
 & \les\norm{w_{0}}_{H^{s}}+\sup_{\theta\in[0,1]}\norm{\left(u+\theta w\right)\cdot\chi_{I}}_{Z^{s}}^{a}\norm{w\cdot\chi_{I}}_{Z^{s}}.
\end{align*}
This implies $\norm w_{Y^{s}}\les\norm{w_{0}}_{H^{s}}$, which completes
the proof of Theorem \ref{thm:Lipschitz}.

Lemma \ref{lem:Lipschitz lemma} is based on modifications of (\ref{eq:Holder strip decomposition}),
(\ref{eq:N,K}), (\ref{eq:N<<K}), and (\ref{eq:beta bound}), allowing one linear term of $u$ to be replaced by that of another function $v$. To handle the difference form, we apply fractional H{\"o}lder inequalities
to the linearization form $\d_{z}\calN u\cdot v$. Here, we require
$s<a$ instead of $s<1+a$. To mimick the proof of (\ref{eq:Holder strip decomposition}),
we require that at least one factor in (\ref{eq:a1+...+ak}), say $a_k$, is equal
to $1/2$. For this, we need $a>1$.
\begin{proof}
We show the dual form of (\ref{eq:Lipschitz bound}) regarding $\d_{z}\calN(u)\cdot v$:
\begin{equation}
\left|\int\psi_{1}\d_{z}\calN(u)\cdot v\cdot\overline{w}dxdt\right|\les\norm u_{Z^{s}}^{a}\norm v_{Z^{s}}\norm w_{Z^{-s}}.\label{eq:dual}
\end{equation}
The conjugate term $\d_{\overline{z}}\calN(u)\cdot\overline{v}$ can
be dealt with similarly.

For $N\in2^{\N}$, we denote by $G^{N}:\R\times\T^{d}\rightarrow\C$
and $B^{N}:\R\times\T^{d}\rightarrow\C^{2}$ the functions
\[
G^{N}:=\d_{z}\calN(u_{\le N})-\d_{z}\calN(u_{\le N/2})
\]
and
\[
B^{N}:=\left(\int_{0}^{1}\d_{zz}\calN(u_{\le N/2}+\theta u_{N})d\theta,\int_{0}^{1}\d_{z\overline{z}}\calN(u_{\le N/2}+\theta u_{N})d\theta\right).
\]
With these notations, we have $G^{N}=u_{N}\times B^{N}$. For $N=1$,
we use the conventional notation $u_{\le1/2}=0$.

We estimate the terms
\begin{align*}
\left|\int\psi_{1}\d_{z}\calN(u)\cdot v\cdot\overline{w}dxdt\right| & \le\left|\sum_{L\in2^{\N}}\int\psi_{1}\d_{z}\calN(u)\cdot v_{\ge L/8}\cdot\overline{w_{L}}dxdt\right|\\
 & +\left|\sum_{L\in2^{\N}}\int\psi_{1}\left(\d_{z}\calN(u)-\d_{z}\calN(u_{\le L/16})\right)\cdot v_{\le L/16}\cdot\overline{w_{L}}dxdt\right|\\
 & +\left|\sum_{L\in2^{\N}}\int\psi_{1}\d_{z}\calN(u_{\le L/16})\cdot v_{\le L/16}\cdot\overline{w_{L}}dxdt\right|\\
 & =I+II+III
\end{align*}
separately.

(1) \emph{Estimate of $I$}

Denote by $K,L,M,N\in2^\N$ dyadic numbers. By (\ref{eq:Holder strip decomposition}),
we have
\[
\left|\int\psi_{1}\d_{z}\calN(u)\cdot v_{M}\cdot\overline{w_{L}}\right|\les\norm u_{Z^{s}}^{a}\norm{v_{M}}_{Z^{0}}\norm{w_{L}}_{Z^{0}},
\]
which implies
\begin{align}
I & \le\sum_{\substack{L,M\in2^{\N}\\
L\le8M
}
}\left|\int\psi_{1}\d_{z}\calN(u)\cdot v_{M}\cdot\overline{w_{L}}dxdt\right|\label{eq:Lip esti 1}\\
 & \les\norm u_{Z^{s}}^{a}\cdot\sum_{\substack{L,M\in2^{\N}\\
L\le8M
}
}(L/M)^{s}\norm{v_{M}}_{Z^{s}}\norm{w_{L}}_{Z^{-s}}\nonumber \\
 & \les\norm u_{Z^{s}}^{a}\norm v_{Z^{s}}\norm w_{Z^{-s}}.\nonumber 
\end{align}
(2) \emph{Estimate of $II$}

Since $a/2>1/2$, mimicking the proof of (\ref{eq:Holder strip decomposition})
with the choice $a_{k}=1/2$ gives the estimate
\[
\left|\int\psi_{1}\left(u_{N}\times B^{N}\right)\cdot v_{\le L/16}\cdot\overline{w_{L}}dxdt\right|\les\norm u_{Z^{s}}^{a-1}\norm{u_{N}}_{Z^{0}}\norm{v_{\le L/16}}_{Z^{s}}\norm{w_{L}}_{Z^{0}}.
\]
This implies
\begin{align}
II & \le\sum_{\substack{L,N\in2^{\N}\\
L\le8N
}
}\left|\int\psi_{1}\left(u_{N}\times B^{N}\right)\cdot v_{\le L/16}\cdot\overline{w_{L}}dxdt\right|\label{eq:Lip esti 2}\\
 & \les\norm u_{Z^{s}}^{a-1}\norm v_{Z^{s}}\cdot\sum_{\substack{L,N\in2^{\N}\\
L\le8N
}
}(L/N)^{s}\norm{u_{N}}_{Z^{s}}\norm{w_{L}}_{Z^{-s}}\nonumber \\
 & \les\norm u_{Z^{s}}^{a}\norm v_{Z^{s}}\norm w_{Z^{-s}}.\nonumber 
\end{align}
(3) \emph{Estimate of $III$}

To estimate $III$, we estimate $\left|\int\psi_{1}G^{N}\cdot v_{\le L/16}\cdot\overline{w_{L}}dxdt\right|$
assuming $L\ge16N$. Let $q,r,p_{0}$, and $p_{1}$ be the parameters
defined in the proof of (\ref{eq:N<<K}). Choose $\mu>0$ and $\nu>0$
as the exponents $\nu=a-s-\s_{1}$ and $\mu=\left(1-\frac{2\s}{\s_{1}}\right)(\nu+\s)+\frac{2\s}{\s_{1}}(-s+\s)$
(recall that $\s\ll\s_{1}\ll1$ in (\ref{eq:ss1s2<<})).

(i) Since $\s_{1}<\min\left\{ (a-1)(s-\s),s-\s\right\} $, by either
(\ref{eq:low fractional Holder}) or (\ref{eq:high fractional Holder}),
for $M\in2^{\N}$ we have
\begin{align*}
 & \|\left(u_{N}\times B_{M}^{N}\right)\cdot v_{\le L/16}\|_{L^{p_{1}'}}\\
 & \les\|u_{N}\|_{L^{r}}\cdot\norm{v_{\le L/16}}_{H^{s-\s,r}}\cdot\|B_{M}^{N}\|_{L^{\left(\frac{1}{p_{1}'}-\frac{2}{r}+\frac{s-\s}{d}\right)^{-1}}}\\
 & \les N^{-s+\s}\|u_{N}\|_{H^{s-\s,r}}\cdot\norm{v_{\le L/16}}_{H^{s-\s,r}}\cdot M^{-\s_{1}}\|B_{M}^{N}\|_{H^{\s_{1},\left(\frac{1}{p_{1}'}-\frac{2}{r}+\frac{s-\s}{d}\right)^{-1}}}\\
 & \les N^{-s+\s}M^{-\s_{1}}\|u_{N}\|_{H^{s-\s,r}}\cdot\norm{v_{\le L/16}}_{H^{s-\s,r}}\cdot\|u\|_{H^{s-\s,r}}^{a-1}.
\end{align*}
Since $N\le L/16$, we have
\begin{align}
 & \|P_{L}(G^{N}\cdot v_{\le L/16})\|_{L^{p_{1}'}}\label{eq:H-1}\\
 & \les\sum_{M\in2^{\N}}\|P_{L}\left(\left(u_{N}\times B_{M}^{N}\right)\cdot v_{\le L/16}\right)\|_{L^{p_{1}'}}\nonumber \\
 & =\sum_{\substack{M\sim L}
}\|P_{L}\left(\left(u_{N}\times B_{M}^{N}\right)\cdot v_{\le L/16}\right)\|_{L^{p_{1}'}}\nonumber \\
 & \les(N/L)^{-s+\s}L^{-s-\s_{1}+\s}\|u_{N}\|_{H^{s-\s,r}}\cdot\norm{v_{\le L/16}}_{H^{s-\s,r}}\cdot\|u\|_{H^{s-\s,r}}^{a-1}\nonumber \\
 & \les(N/L)^{-s+\s}L^{-s-\s_{1}+\s}N^{-\nu}\|u_{\le N}\|_{H^{s-\s+\nu,r}}\cdot\norm{v_{\le L/16}}_{H^{s-\s,r}}\cdot\|u\|_{H^{s-\s,r}}^{a-1}.\nonumber 
\end{align}
(ii) Since $s+\nu<a$, by (\ref{eq:high fractional Holder}), we have
\begin{align}
 & \|P_{L}(G^{N}\cdot v_{\le L/16})\|_{L^{p_{0}'}}\label{eq:L-1}\\
 & \les L^{-s-\nu}\|G^{N}\|_{H^{s+\nu,\left(\frac{1}{p_{0}'}-\frac{1}{r}+\frac{s-\s}{d}\right)^{-1}}}\cdot\norm{v_{\le L/16}}_{H^{s-\s,r}}\nonumber \\
 & \les L^{-s-\nu}\|u_{\le N}\|_{H^{s+\nu,r}}\cdot\|u\|_{L^{a\left(\frac{1}{p_{0}'}-\frac{1}{r}\right)^{-1}}}^{a-1}\cdot\norm{v_{\le L/16}}_{H^{s-\s,r}}\nonumber \\
 & \les L^{-s-\nu}N^{\s}\|u_{\le N}\|_{H^{s-\s+\nu,r}}\cdot\|u\|_{H^{s-\s,r}}^{a-1}\cdot\norm{v_{\le L/16}}_{H^{s-\s,r}}\nonumber \\
 & =(N/L)^{\nu+\s}L^{-s+\s}N^{-\nu}\|u_{\le N}\|_{H^{s-\s+\nu,r}}\cdot\|u\|_{H^{s-\s,r}}^{a-1}\cdot\norm{v_{\le L/16}}_{H^{s-\s,r}}.\nonumber 
\end{align}
Interpolating between (\ref{eq:H-1}) and (\ref{eq:L-1}) with the
exponent $2\s/\s_{1}$, we have
\[
\|P_{L}(G^{N}\cdot v_{\le L/16})\|_{L^{p'}}\les(N/L)^{\mu}L^{-s-\s}N^{-\nu}\|u_{\le N}\|_{H^{s-\s+\nu,r}}\cdot\|u\|_{H^{s-\s,r}}^{a-1}\cdot\norm{v_{\le L/16}}_{H^{s-\s,r}}.
\]
Multiplying cutoffs and applying $L_{t}^{p'}$ norms to both sides,
then using (\ref{eq:Sobolev embed Z^s}), we have
\[
\norm{\psi_{1}P_{L}(G^{N}\cdot v_{\le L/16})}_{\left(Z^{-s}\right)'}\les(N/L)^{\mu}N^{-\nu}\norm{u_{\le N}}_{Z^{s+\nu}}\norm u_{Z^{s}}^{a-1}\norm v_{Z^{s}}.
\]
Since $\sum_{N\le L/16}(N/L)^{\mu}\les1$, by the Cauchy-Schwartz
inequality, we have
\begin{align*}
 & \norm{\sum_{\substack{L,N\in2^{\N}\\
L\ge16N
}
}\int P_{L}\left(\psi_{1}G^{N}\cdot v_{\le L/16}\right)}_{(Z^{-s})'}^{2}\\
 & \les\sum_{\substack{L\in2^{\N}}
}\left(\sum_{N\le L/16}\norm{\psi_{1}P_{L}(G^{N}\cdot v_{\le L/16})}_{\left(Z^{-s}\right)'}\right)^{2}\\
 & \les\norm u_{Z^{s}}^{2a-2}\norm v_{Z^{s}}^{2}\sum_{\substack{L\in2^{\N}}
}\left(\sum_{N\le L/16}(N/L)^{\mu}N^{-\nu}\norm{u_{\le N}}_{Z^{s+\nu}}\right)^{2}\\
 & \les\norm u_{Z^{s}}^{2a-2}\norm v_{Z^{s}}^{2}\sum_{\substack{L\in2^{\N}}
}\sum_{N\le L/16}(N/L)^{\mu}N^{-2\nu}\norm{u_{\le N}}_{Z^{s+\nu}}^{2}\\
 & =\norm u_{Z^{s}}^{2a-2}\norm v_{Z^{s}}^{2}\sum_{\substack{L\in2^{\N}}
}\sum_{N\le L/16}\sum_{K\le N}(N/L)^{\mu}(K/N)^{2\nu}\norm{u_{K}}_{Z^{s}}^{2}\\
 & \les\norm u_{Z^{s}}^{2a-2}\norm v_{Z^{s}}^{2}\sum_{\substack{K,L\in2^{\N}\\
L\ge16K
}
}(K/L)^{\min\left\{ \mu,2\nu\right\} -}\norm{u_{K}}_{Z^{s}}^{2}\\
 & \les\norm u_{Z^{s}}^{2a}\norm v_{Z^{s}}^{2},
\end{align*}
which implies 
\begin{equation}
III\le\left|\sum_{\substack{L,N\in2^{\N}\\
L\ge16N
}
}\int\psi_{1}G^{N}\cdot v_{\le L/16}\cdot\overline{w_{L}}dxdt\right|\les\norm u_{Z^{s}}^{a}\norm v_{Z^{s}}\norm w_{Z^{-s}}.\label{eq:Lip esti 3}
\end{equation}
Combining (\ref{eq:Lip esti 1}), (\ref{eq:Lip esti 2}), and (\ref{eq:Lip esti 3}),
we have (\ref{eq:dual}), finishing the proof of Lemma \ref{lem:Lipschitz lemma}.
\end{proof}

\section{\label{sec:Proof-of-Theorem a<1}Proof of Theorem \ref{thm:a<1}}

In this section, we prove the H{\"o}lder ill-posedness of (\ref{eq:NLS}),
Theorem \ref{thm:a<1}. For the proof, we construct an explicit counterexample.
Although Theorem \ref{thm:a<1} is stated on $\T^{d}$, the counterexample
we construct relies on one spatial variable $x_{1}$. Hence, we perform
the construction on $\T$. Indeed, by considering initial data $u_{0}\in H^{s}(\T^{d})$
of the form $u_{0}(x_{1},\ldots,x_{d})=u_{01}(x_{1}),u_{01}\in H^{s}(\T)$,
one can deduce the result in Theorem \ref{thm:a<1} from that with
the domain replaced by $\T$. Here, $s$ and $a$ are no longer related,
i.e., we do not require (\ref{eq:s_c}) in the rest of this section.

Reducing to the one-dimensional problem, we show the following statement:
\begin{prop}
\label{prop:a<1}Assume $0<a<1$ and $0<s<1+\frac{1}{a}$. The solution
map to 
\begin{equation}
\begin{cases}
iu_{t}+u_{xx}=\calN(u)\\
u(0)=u_{0}\in H^{s}(\T)
\end{cases}\label{eq:NLS(T)}
\end{equation}
is not locally $\al$-H{\"o}lder continuous in $H^{s}(\T)$ for each $\al>a$.
\end{prop}

The proof of Proposition \ref{prop:a<1} does not depend on the sign
of the nonlinearity. For simplicity, we assume that $\calN(u)$ is
defocusing, i.e., $\calN(u)=\left|u\right|^{a}u$.

We construct an explicit pair of initial data $u_{\la,0},u_{2\la,0}\in H^{s}(\T)$
which contradicts the H{\"o}lder continuity in Proposition \ref{prop:a<1}.
Denoting by $u_{\la},u_{2\la}\in C^{0}H^{s}\left([0,T]\times\T\right)$
the solutions to (\ref{eq:NLS(T)}) and letting $w=u_{2\la}-u_{\la}$,
we expand $\P_{N}\d_{x}w(T,0),N\gg1$ in a Duhamel form, expecting
that $\left|\P_{N}\d_{x}w(T,0)\right|$ exceeds the bound required
by such a H{\"o}lder assumption.

In the Duhamel expansion, we have a term of $u,\d_{x}u$, and $w$
(\ref{eq:claimA;1}), which is non-oscillating if we plug $\d_{x}u_{\la}=\d_{x}u_{2\la}=Ce^{i(t-T)\De}\delta_{N}$
and positive reals $u_{\la},u_{2\la},w\gtrsim\la$. Inspired by this
observation, we choose the initial data $u_{\la,0}$ and $u_{2\la,0}$
as functions such that $\frac{1}{2\pi}\int_{\T}u_{\la,0}dx=\la$,
$\frac{1}{2\pi}\int_{\T}u_{2\la,0}dx=2\la$, and $\d_{x}u_{\la,0}=\d_{x}u_{2\la,0}=\theta e^{-iT\De}\delta_{N}$.
The parameters $\la$ and $\theta$ are chosen as constants such that
$u_{\la}$ and $u_{2\la}$ are small perturbations of free evolutions
and (\ref{eq:claimA;1}) dominates the rest in the expansion of $\P_{N}\d_{x}w(T,0)$.

To avoid a small set of irregular high peaks of $u$ generating a
large error, we partition the set $[0,T]\times\T^{d}$ into $E\cup E^{c}$,
where $E$ is the set of points $(t,x)$ at which $u_{\rho}-e^{it\De}u_{\rho,0}$
and $\d_{x}\left(u_{\rho}-e^{it\De}u_{\rho,0}\right)$ are both small
for $\rho=\la,2\la$. The Duhamel integral is estimated on $E$ and
$E^{c}$ separately.

Before proving Proposition \ref{prop:a<1}, we show a preliminary
fact on the Schr{\"o}dinger kernel $e^{it\De}\delta_{N},N\in2^{\N}$ on
$\T$. The following lemma states that for each dyadic $N\in2^{\N}$,
the $L^{2}$ norm and the $L^{\infty}$ norm of $e^{it\De}\delta_{N}$
are comparable on a large set of times $t$:
\begin{lem}
For $T>0$, there exists a constant $C$ depending only on $T$ such
that for every dyadic $N\in2^{\N}$, we have the estimate
\begin{equation}
m\left(\left\{ t\in[0,T]\mid\norm{e^{it\De}\delta_{N}}_{L^{\infty}}\le C\sqrt{N}\right\} \right)\gtrsim_{T}1,\label{eq:L^inftyL^2}
\end{equation}
where $m(\cdot)$ denotes the (Lebesgue) measure of a set.
\end{lem}
\begin{proof}
Let $\mc A$ be the collection of coprime integer pairs $\left(l,m\right)$
such that $\frac{N}{10}\le m\le N$ and $0<\frac{l}{m}<T$. Let $\mc I$
be the collection of intervals $\left(\frac{l}{m}-\frac{1}{mN},\frac{l}{m}+\frac{1}{mN}\right),(l,m)\in\mc A$.
Denote by $\mc T$ the union $\mc T=\cup_{I\in\mc I}I$.

Once we have $m\left(\mc T\right)\gtrsim_{T}1$, we have (\ref{eq:L^inftyL^2}).
For $t\in\mc T$, by the kernel estimate (\ref{eq:Bourgain bound}),
we have $\norm{e^{it\De}\delta_{N}}_{L^{\infty}}\le C\sqrt{N}$ for
some universal constant $C=C(T)$, which implies (\ref{eq:L^inftyL^2}).

We show
the estimate $m\left(\mc T\right)\gtrsim_{T}1$. For two pairs of coprime integers $\left(l_{1},m_{1}\right)$ and
$\left(l_{2},m_{2}\right)$, we have 
\[
\left|\frac{l_{1}}{m_{1}}-\frac{l_{2}}{m_{2}}\right|=\left|\frac{l_{1}m_{2}-l_{2}m_{1}}{m_{1}m_{2}}\right|\ge\frac{1}{m_{1}m_{2}}.
\]
Thus, no more than $100$ members of $\mc I$ can intersect at a common
point. Since the length of each interval in $\mc I$ is comparable
to $\frac{1}{N^{2}}$, we are done once we have $\#\mc A\gtrsim N^{2}.$
Since $\sum_{k=2}^{\infty}\frac{1}{k^{2}}<1$, the number of such
points is indeed comparable to $N^{2}$, finishing the proof of (\ref{eq:L^inftyL^2}).
\end{proof}
Since $e^{it\De}\delta_{N}$ is localized on frequencies comparable
to $N$, (\ref{eq:L^inftyL^2}) indicates an oscillating behavior
of $e^{it\De}\delta_{N}$ for such times $t$. As a particular consequence,
we have $\left|e^{it\De}\delta_{N}(t,x)\right|\gtrsim_{T}\sqrt{N}$
on a large set of points $(t,x)\in[0,T]\times\T$, which is stated in the
following corollary:
\begin{cor}
For $T>0$, there exists a constant $\epsilon>0$ depending only on
$T$ such that for every dyadic $N\in2^{\N}$, we have the estimate
\begin{equation}
m\left(\left\{ (t,x)\in[0,T]\times\T\mid\left|e^{it\De}\delta_{N}(x)\right|\ge\epsilon\sqrt{N}\right\} \right)\gtrsim_{T}1.\label{eq:m(kernel>sqrtN)}
\end{equation}
\end{cor}

\begin{proof}
For each $t\in[0,T]$ such that $\norm{e^{it\De}\delta_{N}}_{L^{\infty}}\sim\sqrt{N}\sim\norm{e^{it\De}\delta_{N}}_{L^{2}}$,
$\left|e^{it\De}\delta_{N}(x)\right|\gtrsim\sqrt{N}$ is attained
on a set of points $x\in\T$ of measure comparable to $1$. Thus, (\ref{eq:L^inftyL^2})
implies (\ref{eq:m(kernel>sqrtN)}).
\end{proof}
Using (\ref{eq:m(kernel>sqrtN)}), we prove Proposition \ref{prop:a<1}
as follows:
\begin{proof}
Since Proposition \ref{prop:a<1} is stronger with lower $\al$, we
may assume $0<\al-a\ll_{a,s}1$ in advance. 

Fix $T>0$. We use the Wirtinger derivatives with the convention $u^{\gamma}\overline{u}^{\gamma}:=\left|u\right|^{2\gamma}$
for $\gamma\in\R$. 

Fix a number $\epsilon_{0}\ll_{a,s,T}1$. Let $N\in2^{\N}$ be a large
dyadic number. Let $\la,\kappa,\iota\ll1$ be the numbers
\begin{align}
\la & =N^{-\frac{s}{1-\al+2a}+\s_{1}},\label{eq:par def of prop a<1}\\
\kappa & =\la^{\al-a}\cdot N^{\s},\nonumber \\
\iota & =N^{-\s},\nonumber 
\end{align}
where $\s$ and $\s_{1}$ are the numbers chosen in (\ref{eq:ss1s2<<}).

For $\rho\in[\la,2\la]$, let $u_{\rho,0}\in H^{\infty}(\T)$ be the
unique function such that
\[
\d_{x}u_{\rho,0}=\kappa N^{\frac{1}{2}-s}e^{-iT\De}\delta_{N}
\]
and
\[
\frac{1}{2\pi}\int_{\T}u_{\rho,0}dx=\rho.
\]
Denote by $u_{\la}$ and $u_{2\la}$ the solutions to (\ref{eq:NLS(T)})
on $[0,T]$ with initial data $u_{\la,0}$ and $u_{2\la,0}$, respectively.
For $\rho\in(\la,2\la)$, denote by $u_{\rho}:[0,T]\times\T\rightarrow\C$ the function
$u_{\rho}=\frac{2\la-\rho}{\la}u_{\la}+\frac{\rho-\la}{\la}u_{2\la}$.
We also denote $w=u_{2\la}-u_{\la}$.

First, we collect some estimates on $u_{\la}$ and $u_{2\la}$. For
$\rho=\la$ and $\rho=2\la$, by $0<s<1+\frac{1}{a}$ and $a<\al$,
we have $\norm{u_{\rho,0}}_{L^{2}}\les\kappa N^{-s}+\la\sim\la$ and
$\norm{u_{\rho,0}}_{H^{1}}\les\kappa N^{1-s}+\la\sim\kappa N^{1-s}$.
Thus, a contraction mapping argument using (\ref{eq:L^4 Strichartz})
gives
\begin{equation}
\norm{u_{\rho}}_{C^{0}L^{2}\cap L^{4}L^{4}}\les\norm{u_{\rho,0}}_{L^{2}}\les\la,\label{eq:u L^2}
\end{equation}
\begin{equation}
\norm{u_{\rho}-e^{it\De}u_{\rho,0}}_{C^{0}L^{2}\cap L^{4}L^{4}}\les\norm{u_{\rho}}_{L^{4}L^{4}}^{1+a}\les\la^{1+a},\label{eq:u-e L^2}
\end{equation}
\begin{equation}
\norm{u_{\rho}}_{C^{0}H^{1}\cap L^{4}W^{1,4}}\les\norm{u_{\rho,0}}_{H^{1}}\les\kappa N^{1-s},\label{eq:u H^1}
\end{equation}
and
\begin{equation}
\norm{u_{\rho}-e^{it\De}u_{\rho,0}}_{C^{0}H^{1}\cap L^{4}W^{1,4}}\les\norm{u_{\rho}}_{L^{4}L^{4}}^{a}\norm{u_{\rho}}_{L^{4}W^{1,4}}\les\la^{a}\kappa N^{1-s}.\label{eq:u-e H^1}
\end{equation}
If Proposition \ref{prop:a<1} were false on $[0,T]$, since
\[
\norm{u_{\la,0}}_{H^{s}},\norm{u_{2\la,0}}_{H^{s}}\les\kappa+\la\ll1,
\]
we would have
\[
\left|\P_{N}\d_{x}w(T,0)\right|\les N^{\frac{3}{2}-s}\norm w_{C^{0}H^{s}}\les N^{\frac{3}{2}-s}\norm{u_{2\la,0}-u_{\la,0}}_{H^{s}}^{\al}\les\la^{\al}N^{\frac{3}{2}-s}.
\]
Thus, showing the following leads to a contradiction:
\begin{equation}
\left|\P_{N}\d_{x}w(T,0)\right|\gg\la^{\al}N^{\frac{3}{2}-s}.\label{eq:claim}
\end{equation}
Since $w(0)=\la$ and
\[
(i\d_{t}+\d_{xx})w=\calN(u_{2\la})-\calN(u_{\la}),
\]
we have
\begin{align}
i\P_{N}\d_{x}w(T,0) & =\int_{[0,T]\times\T}\overline{e^{i(t-T)\De}\delta_{N}}\d_{x}\left(\calN(u_{2\la})-\calN(u_{\la})\right)dxdt\label{eq:dw}\\
 & =\int_{E}\overline{e^{i(t-T)\De}\delta_{N}}\left(A+B\right)dxdt\nonumber \\
 & +\int_{E^{c}}\overline{e^{i(t-T)\De}\delta_{N}}\left(\d_{x}\calN(u_{2\la})-\d_{x}\calN(u_{\la})\right)dxdt,\nonumber 
\end{align}
where $E\subset[0,T]\times\T$ is the set of $(t,x)\in[0,T]\times\T$
such that for $\rho=\la,2\la$,
\begin{align}
\left|\frac{u_{\rho}}{\rho}-1\right| & \le\epsilon_{0},\label{eq:up/p}\\
\left|\d_{x}u_{\rho}-\kappa N^{\frac{1}{2}-s}e^{i(t-T)\De}\delta_{N}\right| & \le\iota\kappa N^{1-s}\label{eq:dxup-dN}
\end{align}
are satisfied, and the terms $A$ and $B$ denote
\begin{align*}
A & =\frac{1}{\la}\int_{\la}^{2\la}\frac{a}{2}\left(\frac{a}{2}+1\right)\left(u_{\rho}^{\frac{a}{2}-1}\overline{u}_{\rho}^{\frac{a}{2}}w+u_{\rho}^{\frac{a}{2}}\overline{u}_{\rho}^{\frac{a}{2}-1}\overline{w}\right)\d_{x}u_{\rho}d\rho\\
 & +\frac{1}{\la}\int_{\la}^{2\la}\left(\frac{a}{2}\left(\frac{a}{2}+1\right)u_{\rho}^{\frac{a}{2}}\overline{u}_{\rho}^{\frac{a}{2}-1}w+\frac{a}{2}\left(\frac{a}{2}-1\right)u_{\rho}^{\frac{a}{2}+1}\overline{u}_{\rho}^{\frac{a}{2}-2}\overline{w}\right)\overline{\d_{x}u_{\rho}}d\rho
\end{align*}
and
\[
B=\frac{1}{\la}\int_{\la}^{2\la}\left(\frac{a}{2}+1\right)u_{\rho}^{\frac{a}{2}}\overline{u}_{\rho}^{\frac{a}{2}}\d_{x}w+\frac{a}{2}u_{\rho}^{\frac{a}{2}+1}\overline{u}_{\rho}^{\frac{a}{2}-1}\overline{\d_{x}w}d\rho.
\]
We show the following claims:
\begin{equation}
\left|\int_{E}\overline{e^{i(t-T)\De}\delta_{N}}Adxdt\right|\gg\la^{\al}N^{\frac{3}{2}-s},\label{eq:claimA;1}
\end{equation}
\begin{equation}
\left|\int_{E}\overline{e^{i(t-T)\De}\delta_{N}}Bdxdt\right|\ll\la^{\al}N^{\frac{3}{2}-s},\label{eq:claimB;2}
\end{equation}
and
\begin{equation}
\left|\int_{E^{c}}\overline{e^{i(t-T)\De}\delta_{N}}\left(\d_{x}\calN(u_{2\la})-\d_{x}\calN(u_{\la})\right)dxdt\right|\ll\la^{\al}N^{\frac{3}{2}-s}.\label{eq:claimN;3}
\end{equation}
Once we have (\ref{eq:claimA;1}), (\ref{eq:claimB;2}), and (\ref{eq:claimN;3}),
we have (\ref{eq:claim}) immediately, finishing the proof.

1. Proof of (\ref{eq:claimA;1}))\emph{ }For $\rho=\la$ and $\rho=2\la$,
by (\ref{eq:u-e L^2}), we have
\begin{equation}
\norm{u_{\rho}-\rho}_{L^{4}L^{4}}\les\norm{e^{it\De}u_{\rho,0}-\rho}_{L^{4}L^{4}}+\norm{u_{\rho}-e^{it\De}u_{\rho,0}}_{L^{4}L^{4}}\les\kappa N^{-s}+\la^{1+a}\les\la^{1+a},\label{eq:u-p L4}
\end{equation}
which implies
\begin{equation}
m\left(\left\{ (t,x)\in[0,T]\times\T\mid\left|u_{\rho}-\rho\right|>\epsilon_{0}\rho\right\} \right)\les\left(\frac{\la^{1+a}}{\epsilon_{0}\rho}\right)^{4}\sim\la^{4a}.\label{eq:m1}
\end{equation}
Similarly, by (\ref{eq:u-e H^1}), we have
\begin{equation}
\norm{\d_{x}u_{\rho}-\kappa N^{\frac{1}{2}-s}e^{i(t-T)\De}\delta_{N}}_{L^{4}L^{4}}\les\la^{a}\kappa N^{1-s},\label{eq:du-d L^4}
\end{equation}
which implies
\begin{align}
 & m\left(\left\{ (t,x)\in[0,T]\times\T\mid\left|\d_{x}u_{\rho}-\kappa N^{\frac{1}{2}-s}e^{i(t-T)\De}\delta_{N}\right|>\iota\kappa N^{1-s}\right\} \right)\label{eq:m2}\\
 & \les\left(\frac{\la^{a}\kappa N^{1-s}}{\iota\kappa N^{1-s}}\right)^{4}=\frac{\la^{4a}}{\iota^{4}}.\nonumber 
\end{align}
Combining (\ref{eq:m1}) and (\ref{eq:m2}), since $\iota\ll1$, we
have
\begin{equation}
m(E^{c})\les\frac{\la^{4a}}{\iota^{4}}.\label{eq:m}
\end{equation}
For $(t,x)\in E$, by (\ref{eq:up/p}) and (\ref{eq:dxup-dN}), we
have
\begin{align*}
\Re\left[\overline{e^{i(t-T)\De}\delta_{N}}A\right] & \ge\frac{a^{2}}{100}\la^{a}\kappa N^{\frac{1}{2}-s}\left|e^{i(t-T)\De}\delta_{N}\right|^{2}\\
 & -100\la^{a}\iota\kappa N^{1-s}\left|e^{i(t-T)\De}\delta_{N}\right|.
\end{align*}
Since $m(E^{c})\ll1$, by (\ref{eq:m(kernel>sqrtN)}), we have
\[
\int_{E}\left|e^{i(t-T)\De}\delta_{N}\right|^{2}dxdt\sim N.
\]
Thus, by $\iota\ll1$, we have
\[
\left|\int_{E}\overline{e^{i(t-T)\De}\delta_{N}}Adxdt\right|\gtrsim\la^{a}\kappa N^{\frac{3}{2}-s}\gg\la^{\al}N^{\frac{3}{2}-s}.
\]

2. Proof of (\ref{eq:claimB;2})) By (\ref{eq:u L^2}) and (\ref{eq:u-e H^1}),
we have
\begin{align*}
\left|\int_{E}\overline{e^{i(t-T)\De}\delta_{N}}Bdxdt\right| & \les N^{\frac{1}{2}}\norm B_{L^{1}L^{2}}\\
 & \les N^{\frac{1}{2}}\sup_{\rho\in[\la,2\la]}\norm{u_{\rho}^{a}}_{L^{4}L^{4}}\norm{\d_{x}w}_{L^{4}L^{4}}\\
 & \les N^{\frac{1}{2}}\sup_{\rho\in[\la,2\la]}\norm{u_{\rho}}_{L^{4}L^{4}}^{a}\norm{w-\la e^{it\De}1}_{L^{4}W^{1,4}}\\
 & \les\la^{2a}\kappa N^{\frac{3}{2}-s}\ll\la^{\al}N^{\frac{3}{2}-s}.
\end{align*}

3. Proof of (\ref{eq:claimN;3})) For $\rho=\la$ and $\rho=2\la$,
we have
\begin{align*}
\left|\int_{E^{c}}\overline{e^{i(t-T)\De}\delta_{N}}\d_{x}(\calN(u_{\rho}))dxdt\right| & \les m(E^{c})^{\frac{1}{4}}\norm{\overline{e^{i(t-T)\De}\delta_{N}}\d_{x}(\calN(u_{\rho}))}_{L^{4/3}L^{4/3}}\\
 & \les m(E^{c})^{\frac{1}{4}}\norm{e^{i(t-T)\De}\delta_{N}}_{L^{4}L^{4}}\norm{u_{\rho}}_{L^{4}L^{4}}^{a}\norm{u_{\rho}}_{L^{4}W^{1,4}}\\
 & \les m(E^{c})^{\frac{1}{4}}\cdot N^{\frac{1}{2}}\la^{a}\kappa N^{1-s}\\
 & \les\la^{2a}\iota^{-1}\kappa N^{\frac{3}{2}-s}\ll\la^{\al}N^{\frac{3}{2}-s}.
\end{align*}
Combining (\ref{eq:claimA;1}), (\ref{eq:claimB;2}), and (\ref{eq:claimN;3}),
we have (\ref{eq:claim}), which finishes the proof.
\end{proof}

\section*{Appendix}

The following is a list of parameters used in Section \ref{sec:Z^s-spaces}
and Section \ref{sec:Proof-of-Theorem a<1}, rewritten in terms of
$\s_{j}$'s:

\begin{tabular}{|c|c|}
\hline 
Relation & First-appearing place\tabularnewline
\hline 
\hline 
$s=\frac{d}{2}-\frac{2}{a}$ & (\ref{eq:s_c})\tabularnewline
\hline 
$\frac{1}{p}=\frac{\frac{d}{2}-\s}{d+2}$ & (\ref{eq:=00005Cs,p def})\tabularnewline
\hline 
$\s\ll\s_{1}\ll\s_{2}\ll\s_{3}\ll\s_{4}\ll1$ & (\ref{eq:ss1s2<<})\tabularnewline
\hline 
$\frac{1}{q_{0}}=\frac{2+\s_{3}}{d+2}$ & Lemma \ref{lem:strip1 claim1-1}\tabularnewline
\hline 
$\frac{1}{r_{0}}=\frac{2+\s_{3}+\s_{2}}{d+2}$ & Lemma \ref{lem:strip1 claim1-1}\tabularnewline
\hline 
$\theta=\frac{2}{q_{0}}+\frac{d}{r_{0}}-2=\s_{3}+\frac{d\s_{2}}{d+2}$ & Lemma \ref{lem:strip1 claim1-1}\tabularnewline
\hline 
$\frac{1}{\hat{r}}=\frac{1+\s_{4}}{d+2}$ & Proof of Proposition \ref{prop:Holder strip decomposition}\tabularnewline
\hline 
$\zeta=\frac{1}{\hat{r}}-\frac{1}{2q_{0}}=\frac{\s_{4}-\s_{3}/2}{d+2}$ & Proof of Proposition \ref{prop:Holder strip decomposition}\tabularnewline
\hline 
$\eta=\frac{d}{\hat{r}}-\frac{d}{2r_{0}}+\frac{\theta}{2}=\frac{d}{d+2}\s_{4}+\frac{1}{d+2}\s_{3}$ & Proof of Proposition \ref{prop:Holder strip decomposition}\tabularnewline
\hline 
$\la=N^{-\frac{s}{1-\al+2a}+\s_{1}}$ & (\ref{eq:par def of prop a<1})\tabularnewline
\hline 
$\kappa=\la^{\al-a}\cdot N^{\s}$ & (\ref{eq:par def of prop a<1})\tabularnewline
\hline 
$\iota=N^{-\s}$ & (\ref{eq:par def of prop a<1})\tabularnewline
\hline 
\end{tabular}

\bibliographystyle{plain}
\bibliography{citationforTd}

\end{document}